\pdfoutput=1
\documentclass[12pt]{amsart}
\usepackage[margin=20truemm]{geometry}
\usepackage[backend=biber, style=alphabetic]{biblatex}
\DeclareSourcemap{
  \maps[datatype=bibtex]{
    \map{
      \step[fieldsource=url, 
            match=\regexp{https?://(dx\.)?doi(-|\.)org}, 
            final]
      \step[fieldset=url, null]
      \step[fieldset=urldate, null]
    }
  }
}

\AtEveryBibitem{%
    \ifentrytype{book}{\clearfield{pages}}{}%
}

\DeclareFieldFormat{pages}{#1}
\renewbibmacro{in:}{}

\DeclareFieldFormat[article]{number}{\bibstring{number}~#1}
\renewbibmacro*{journal+issuetitle}{%
	\usebibmacro{journal}%
	\setunit*{\addspace}%
	\iffieldundef{series}
	{}
	{\newunit
		\printfield{series}%
		\setunit{\addspace}}%
	\printfield{volume}%
	\setunit{\addspace}%
	\usebibmacro{issue+date}%
	\setunit{\addcolon\space}%
	\usebibmacro{issue}%
	\setunit{\addcomma\space}%
	\printfield{number}%
	\setunit{\addcomma\space}%
	\printfield{eid}%
	\newunit
}

\DeclareFieldFormat[article, inproceedings]{title}{#1}

\usepackage[svgnames]{xcolor}
\usepackage{mathtools}
\usepackage{tikz-cd}
\usepackage{amsmath, amssymb, amsthm}
\usepackage{here}
\usetikzlibrary{intersections, patterns, snakes}

\numberwithin{equation}{section}
\mathtoolsset{showonlyrefs=true}

\theoremstyle{plain}

\newtheorem{theorem}{Theorem}[section]

\newtheorem{lemma}[theorem]{Lemma}
\newtheorem{proposition}[theorem]{Proposition}
\newtheorem{corollary}[theorem]{Corollary}

\theoremstyle{definition}
\newtheorem{definition}[theorem]{Definition}
\newtheorem{example}[theorem]{Example}
\newtheorem{remark}[theorem]{Remark}

\DeclareMathOperator{\Hom}{\mathrm{Hom}}
\DeclareMathOperator{\Tor}{\mathrm{Tor}}
\DeclareMathOperator{\CHom}{\mathcal{H}\!\mathit{om}}

\DeclareMathOperator{\CExt}{\mathcal{E}\!\mathit{xt}}
\DeclareMathOperator{\Auteq}{\mathrm{Auteq}}
\DeclareMathOperator{\Cone}{\mathrm{Cone}}
\DeclareMathOperator{\ev}{\mathrm{ev}}
\DeclareMathOperator{\id}{\mathrm{id}}

\DeclareMathOperator{\Pic}{\mathrm{Pic}}
\DeclareMathOperator{\MCG}{\mathrm{MCG}}

\DeclareMathOperator{\Ker}{\mathrm{Ker}}
\DeclareMathOperator{\Image}{\mathrm{Im}}
\DeclareMathOperator{\Aut}{\mathrm{Aut}}
\DeclareMathOperator{\Inn}{\mathrm{Inn}}
\DeclareMathOperator{\Out}{\mathrm{Out}}
\DeclareMathOperator{\Supp}{\mathrm{Supp}}
\DeclareMathOperator{\SL}{\mathrm{SL}}
\DeclareMathOperator{\Spec}{\mathrm{Spec}}
\DeclareMathOperator{\Perf}{\mathrm{Perf}}

\DeclareMathOperator{\Ext}{\mathrm{Ext}}
\DeclareMathOperator{\Hilb}{\mathrm{Hilb}}
\DeclareMathOperator{\FM}{\mathrm{FM}}
\DeclareMathOperator{\res}{\mathrm{res}}
\DeclareMathOperator{\red}{\mathrm{red}}

\newcommand{\nc}{\newcommand}

\nc{\cA}{{\mathcal{A}}}
\nc{\cB}{{\mathcal{B}}}
\nc{\cE}{{\mathcal{E}}}
\nc{\cF}{{\mathcal{F}}}
\nc{\cG}{{\mathcal{G}}}
\nc{\cH}{{\mathcal{H}}}
\nc{\cL}{{\mathcal{L}}}
\nc{\cM}{{\mathcal{M}}}
\nc{\cN}{{\mathcal{N}}}

\nc{\cO}{{\mathcal{O}}}
\nc{\cP}{{\mathcal{P}}}
\nc{\cQ}{{\mathcal{Q}}}
\nc{\cR}{{\mathcal{R}}}

\nc{\cU}{{\mathcal{U}}}
\nc{\cW}{{\mathcal{W}}}

\nc{\bA}{{\mathbb{A}}}
\nc{\bC}{{\mathbb{C}}}
\nc{\bP}{{\mathbb{P}}}
\nc{\bQ}{{\mathbb{Q}}}
\nc{\bT}{{\mathbb{T}}}
\nc{\bZ}{{\mathbb{Z}}}

\makeatletter
\newcommand*{\rom}[1]{\expandafter\@slowromancap\romannumeral #1@}
\makeatother

\usepackage[
    colorlinks=true,
    citecolor=MediumBlue,
    urlcolor=MediumBlue,
    linkcolor=MediumBlue,
    breaklinks=true,
]{hyperref}

\title[Half-spherical twists]{
    Half-spherical twists on derived categories\\ of coherent sheaves
    }
\author[H.~Arai]{Hayato Arai}
\address{
Graduate School of Mathematical Sciences,
The University of Tokyo,
3-8-1 Komaba,
Meguro-ku,
Tokyo,
153-8914,
Japan.
}
\email{hayato@ms.u-tokyo.ac.jp}
\begin{document}
\begin{abstract}
    For a flat morphism $\pi \colon X \to T$ between smooth quasi-projective varieties and its fiber $X_0$,
    we prove that spherical objects on $D^b(X)$ pushed-forward from $D^b(X_0)$ induce autoequivalences of $D^b(X_0)$ itself.
    Our construction provides new derived symmetries for some singular varieties, which include singular fibers of elliptic surfaces (commonly referred to as Kodaira fibers) and type $\rom{3}$ degenerations of K3 surfaces.
    In the case of Kodaira fibers of type $\rom{1}_n$, we also show the induced autoequivalences of $D^b(X_0)$
    correspond to the \emph{half twists} on the $n$-punctured $2$-torus via homological mirror symmetry.
    As an application, we describe the autoequivalence groups of elliptic surfaces
    in terms of mapping class groups of punctured tori.
\end{abstract}
\maketitle

\section{Introduction}
\subsection{Background}
Let $X$ be a complex projective variety and $D^b(X)$ be the bounded derived category of coherent sheaves on $X$.
The group consisting of the isomorphism classes of $\bC$-linear exact autoequivalences of $D^b(X)$
is denoted by $\Auteq{D^b(X)}$.
There are always three fundamental types of autoequivalences for any $X$: pulling back by automorphisms of $X$, tensoring with line bundles, and the shift functor.
The subgroup generated by these autoequivalences
\begin{equation}
    \Aut{X} \ltimes \Pic{X} \times \bZ[1] \subset \Auteq{D^b(X)}
\end{equation}
is denoted by $A(X)$ and its elements
are called \emph{standard autoequivalences}.

Bondal and Orlov \cite{MR1818984} showed that $A(X) = \Auteq{D^b(X)}$ holds when $X$ is smooth projective, and either the canonical bundle $\omega_X$ or its dual is ample.
The simplest example that does not satisfy this condition is an elliptic curve.
In that case, we have an exact sequence
\begin{equation}\label{eq:autoequivalence-of-elliptic-curve}
    1 \to \Aut{X} \ltimes \Pic^{0}(X)\times \bZ[2] \to \Auteq{D^b(X)} \xrightarrow{\theta} \SL(2, \bZ) \to 1.
\end{equation}
This is a special case of Orlov's result \cite{MR1921811} for abelian varieties,
which goes back to the work of Mukai \cite{MR607081}
where non-trivial (auto-)equivalences of derived categories of abelian varieties were discovered.
The group homomorphism $\theta$ is given by the action on the numerical Grothendieck group $K_{num}(X) \cong \bZ^2$ and one has $A(X) \subsetneq \Auteq{D^b(X)}$ by computing the image $\theta(A(X))$.
To construct non-standard autoequivalences explicitly,
the notion of \emph{twist functors along spherical objects},
or \emph{spherical twists} for short,
introduced by Seidel and Thomas \cite{MR1831820}, is useful.
For example, the twist $T_{\cO_X}$ along the structure sheaf $\cO_X$ is not standard when $X$ is an elliptic curve.

There are several other varieties for which the autoequivalence groups are (partially) determined.
They include toric surfaces \cite{MR3162236}, $K3$ surfaces \cite{MR2553878, MR3592689}, elliptic surfaces \cite{MR3568337}, the minimal resolution of $A_n$-singularities \cite{MR2198807, MR2629510},
and some singular curves \cite{MR2264663, opperJEMS}.
The most fundamental tools to study autoequivalence groups in these studies are spherical twists and group actions on the Grothendieck group (or other suitable spaces) like $\theta$ in \eqref{eq:autoequivalence-of-elliptic-curve}.

Generalizations of spherical twists have been studied by various authors,
including
Horja \cite{MR2126495},
Huybrechts--Thomas \cite{MR2200048},
Toda \cite{MR2430202},
and
Anno--Logvinenko \cite{MR3692883},
to name a few.
Notably, \emph{$\bP^n$-twists along $\bP^n$-objects} introduced in \cite{MR2200048} are related to spherical twists in the following way.
Consider a smooth family of varieties $X \to C$  over a smooth curve $C$, with $i \colon X_{0} \hookrightarrow X$ denoting the inclusion of the fiber at $0 \in C$.
Under suitable assumptions, a $\bP^n$-object $\cE \in D^b(X_0)$ induces a spherical object $i_*\cE \in D^b(X)$.
The associated twists $P_{\cE} \in \Auteq{D^b(X_0)}$ and $T_{i_* \cE} \in \Auteq{D^b(X)}$ make the following diagram commutative \cite[Proposition 2.7]{MR2200048}:
\begin{equation} \label{eq:commutative-diagram-of-p-twist}
    \begin{tikzcd}
        D^b(X_0) \arrow[r,"i_*"]\arrow[d, "P_{\cE}"'] & D^b(X) \arrow[d,"T_{i_*{\cE}}"]\\
        D^b(X_0) \arrow[r, "i_*"]& D^b(X).
    \end{tikzcd}
\end{equation}

\subsection{Half-spherical twists}
The first goal of this paper is to generalize this picture to the case when $X$ is a flat but not necessarily smooth family over a smooth base $T$, and introduce a variant of spherical twists which we call \emph{half-spherical twists}.
First, let us recall the theory of \emph{relative integral functors}.
An \emph{integral functor} $\Phi_{\cP}$ with an \emph{integral kernel} $\cP \in D^b(X \times Y)$ is a functor of the form
\begin{equation}
    \Phi_{\cP} \colon D^b(X) \to D^b(Y), \quad \cF \mapsto \pi_{Y*}(\pi_X^*\cF \otimes \cP),
\end{equation}
where $X, Y$ are smooth projective varieties and $\pi_X, \pi_Y$ are the projections to $X$ and $Y$, respectively.
It was first introduced by Mukai \cite{MR607081} to construct a nontrivial derived equivalence between abelian varieties as mentioned above.
When an integral functor is an equivalence, it is also called a \emph{Fourier--Mukai transform}.
A relative integral functor is a generalization of this concept to the case when $X$ and $Y$ are defined over a common variety $T$.
In this version of integral functors, the projections are replaced by the relative ones $X \times_T Y \to X$ and $X \times_T Y \to Y$, and the kernel $\cP$ is taken from $D^b(X \times_T Y)$.
The general theory for relative integral functors has been developed in the literature such as \cite{MR1972146, MR2238172, MR2505443, MR2593258}, including the following fact.
\begin{proposition}[Proposition \ref{prop:restriction-to-fiber}]\label{prop:general-property-of-relative-integral-functors}
    Suppose $X \to T$ is a flat morphism or varieties.
    Let $T' \to T$ be a morphism and $X' = X \times_T T'$ be the base change of $X$.
    Let $f \colon X' \to X$ and $\varphi \colon X' \times_{T'} X' \to X \times_T X$ be the natural morphisms.
    Then for every integral kernel $\cP \in D_{qc}(X \times_T X)$ the integral functors $\Phi_{\cP}$ and $\Phi_{\varphi^* \cP}$ fit into the following commutative diagram:
    \begin{equation}
        \begin{tikzcd}
            D_{qc}(X') \ar[r, "f_*"] \ar[d, "\Phi_{\varphi^* \cP}"'] & D_{qc}(X) \ar[r, "f^*"]\ar[d, "\Phi_{\cP}"] & D_{qc}(X') \ar[d,  "\Phi_{\varphi^* \cP}"]\\
            D_{qc}(X') \ar[r, "f_*"'] & D_{qc}(X) \ar[r, "f^*"'] & D_{qc}(X').
        \end{tikzcd}
    \end{equation}
\end{proposition}
Utilizing this property, we show the first main result of this paper which generalizes the relation between $\bP^n$-twists and spherical twists:
\begin{theorem}[Proposition \ref{prop:twist-functor-is-relative-fm}, Definition \ref{def:half-spherical-twist}, and Corollary \ref{cor:compatibility-of-half-spherical-twists-and-spherical-twists}]\label{thm:main-theorem-1-half-spherical-twist}
    Let $\pi \colon X \to T$ be a flat morphism between smooth quasi-projective varieties over an algebraically closed field $k$ and $X_0$ be a fiber over a closed point $0 \in T$.
    Let $\cE \in D^b(X_0)$ be an object such that $i_*{\cE}$ is spherical in $D^b(X)$.
    Then the spherical twist $T_{i_*{\cE}}$ is a relative integral functor.
    In particular, there exists an autoequivalence $H_{\cE} \in \Auteq D^b(X_0)$ from $\cE$ and $i$ which makes the following diagram commutative:
    \begin{equation}\label{eq:half-spherical-twist-commutative-diagram}
        \begin{tikzcd}
            D^b(X_0) \ar[r, "i_*"] \ar[d, "H_{\cE}"'] & D^b(X) \ar[r, "i^*"]\ar[d, "T_{i_*{\cE}}"] & D^b(X_0) \ar[d,  "H_{\cE}"]\\
            D^b(X_0) \ar[r, "i_*"'] & D^b(X) \ar[r, "i^*"'] & D^b(X_0).
        \end{tikzcd}
    \end{equation}
\end{theorem}
Our contribution here is verifying that the twist $T_{i_*{\cE}}$ is indeed a relative integral functor.
We achieve this by chasing and carefully reformulating the argument of \cite{MR2200048} to make it valid in non-proper cases with singular fibers or higher-dimensional bases.
Once this is established, Theorem \ref{thm:main-theorem-1-half-spherical-twist} follows immediately from Proposition \ref{prop:general-property-of-relative-integral-functors}.

We call an object $\cE \in D^b(X_0)$ such that $i_*{\cE}$ is spherical in $D^b(X)$ a \emph{half-spherical object (with respect to $i$)} in this paper.
The associated autoequivalence $H_{\cE}$ is called a \emph{half-spherical twist along $\cE$}.
The reason of this naming is described in the next subsection.

\subsection{Applications}

Some interesting examples of half-spherical twists arise from flat but not necessarily smooth families of varieties, such as degenerations of elliptic curves (i.e.~elliptic surfaces, Example \ref{ex:half-spherical-twist-from-kodaira-fiber}) or K3 surfaces (Example \ref{ex:K3-degeneration}).
Among such families, we mainly focus on elliptic surfaces and their singular fibers, often referred to as \emph{Kodaira fibers}.
Let $\pi \colon S \to C$ be an elliptic surface and $F = \pi^{-1}(0)$ be a singular fiber.
If $F$ is reducible, then each irreducible component $G$ of $F$ is a $(-2)$-curve on $X$ and its structure sheaf $\cO_G$ becomes a spherical object in $D^b(X)$ so that Theorem \ref{thm:main-theorem-1-half-spherical-twist} provides the half-spherical twist $H_{\cO_G}$.
The properties of $H_{\cO_G}$ can be studied, especially for Kodaira fibers of type $\rom{1}_n$, through \emph{homological mirror symmetry (HMS)} for punctured tori \cite{MR3663596} (Theorem \ref{thm:mirror-symmetry-for-F_n}).
It relates the derived category of the type $\rom{1}_n$ Kodaira fiber $F_n$ to the Fukaya category of $n$-punctured $2$-torus $T_n$.
Our second result identifies the autoequivalence $H_{\cO_G}$ with a
\emph{half twist} on $T_n$ (see Section \ref{subsection:Dehn-twists-and-mapping-class-groups} for definition).

\begin{theorem}[Theorem \ref{thm:half-spherical-twist-and-half-twist}]\label{thm:main-theorem-2-half-twist}
    Let $\pi \colon S \to C$ be an elliptic surface and $F_n = \pi^{-1}(0)$ be a Kodaira fiber of type $\rom{1}_n$ with $n \geq 2$.
    Let $G \subset F_n$ be an irreducible component and $\gamma_{\cO_G}$ be the curve on $T_n$ corresponding to $\cO_G$ via homological mirror symmetry $D^b(F_n) \cong D^b(\cW(T_n))$, where $\cW(T_n)$ is the wrapped Fukaya category of $T_n$ \cite{MR3663596}.
    Then the half-spherical twist $H_{\cO_G}$ corresponds to the half twist $h_{\gamma_{\cO_G}}$ along $\gamma_{\cO_G}$ on $T_n$.
    More precisely, the twist $H_{\cO_G}$ is mapped to $h_{\gamma_{\cO_G}}$ by the morphism
    \begin{equation}
        \Upsilon \colon \Auteq{D^b(F_n)} \to \MCG(T_n)
    \end{equation}
    defined in \cite{opperJEMS} (see Theorem \ref{thm:autoequivalence-of-I_n-curve}).
\end{theorem}
The term \emph{half-spherical twists} is motivated by this result, as the original spherical twists are named after their analogy to Dehn twists.
The proof of Theorem \ref{thm:main-theorem-2-half-twist} relies on the following facts:
\begin{enumerate}
    \item[(A)] An element of the mapping class group of $T_n$ is completely determined by its action on $\pi_1(T_n)$ (this is a part of Dehn--Nielsen--Baer theorem, see Theorem \ref{thm:Dehn--Nielsen--Baer}).
    \item[(B)] HMS provides the correspondences between indecomposable objects of $D^b(F_n)$ and homotopy classes of curves on $T_n$ (along with some additional data), and between dimensions of $\Hom$-spaces in $D^b(F_n)$ and intersection numbers of curves on $T_n$.
    \item[(C)] The fundamental group $\pi_1(T_n)$ is generated by the curves $\gamma_0, \gamma_1, \dots, \gamma_n$ on $T_n$ corresponding to the objects $\cO_{F_n}, \cO_{x_1}, \dots, \cO_{x_n} \in D^b(F_n)$, where $x_1, \dots, x_n$ are certain points on $F_n$.
\end{enumerate}
Our strategy consists of the following steps.
\begin{enumerate}
    \item Utilizing the facts (A) and (C), the problem is reduced to determining the homotopy classes of curves $\Upsilon(H_{\cO_G})(\gamma_i)$ for $0 \leq i \leq n$.
    \item By the fact (B) and a certain property of $\Upsilon$, we can compute the intersection numbers $\#(\Upsilon(H_{\cO_G})(\gamma_i) \cap \gamma)$ for various $\gamma \in \pi_1(T_n)$ using homological algebra. This data is enough to determine the homotopy classes of curves $\Upsilon(H_{\cO_G})(\gamma_i)$.
\end{enumerate}

The final result of this paper is about autoequivalence groups of elliptic surfaces.
Our starting point is the following Uehara's result.
\begin{theorem}[{\cite[Theorem 4.1]{MR3568337}}]\label{thm:autoequivalence-of-elliptic-surface}
    Let $\pi \colon S \to C$ be a relatively minimal, smooth projective elliptic surface with non-zero Kodaira dimension.
    Let further
    \begin{equation}
        B = \langle T_{\cO_G(a)} \mid G \subset S \text{ is a $(-2)$-curve, } a \in \bZ \rangle
    \end{equation}
    be the subgroup of $\Auteq D^b(S)$ generated by twist functors coming from $(-2)$-curves.
    Suppose that all the reducible fibers of $\pi$ are non-multiple and of type $\rom{1}_n$ for some $n \geq 2$ (i.e.~the cycle of $n$ projective lines).
    Then there is the exact sequence
    \begin{align}
        1 \to \langle B, (-)\otimes \cO_S(D)\mid D.F=0, F \textrm{ is a fiber } \rangle & \rtimes \Aut{S} \times \bZ[2]                      \\
                                                                                        & \to \Auteq{D^b(S)} \xrightarrow{\Theta} \SL(2,\bZ)
    \end{align}
    of groups.
    Moreover,
    \begin{enumerate}
        \item there is an explicit characterization of the image $\Image(\Theta)$ in terms of certain moduli spaces of sheaves on $S$ (see \cite[Section 1]{MR3568337} for details);
        \item if $\pi$ has a section, then $\Theta$ is surjective.
    \end{enumerate}
\end{theorem}

We give a better understanding of the group $B$ in terms of mapping class groups of punctured tori.
Let $F^{(1)}, \dots, F^{(r)}$ be the reducible fibers of $\pi$.
We can show that there is a natural ``restriction'' morphism
\begin{equation}
    \res \colon B \to \prod_{j=1}^r \Auteq{D^b(F^{(j)})}
\end{equation}
whose kernel is generated by $(-) \otimes \cO_S(F^{(j)}), 1 \leq j \leq r$ (see \eqref{eq:definition-of-res} and Proposition \ref{prop:kernel-of-res} for definition).
Assuming all the reducible fibers are of type $\rom{1}_n$, we obtain the following by applying Theorem \ref{thm:main-theorem-2-half-twist},

\begin{theorem}[Theorem \ref{thm:description-of-B}]\label{thm:main-theorem-3-description-of-B}
    Let $\pi \colon S \to C$ be a relatively minimal, smooth projective elliptic surface with non-zero Kodaira dimension.
    Let $F^{(1)}, \dots, F^{(r)}$ be its reducible fibers and suppose that each $F^{(j)}$ is of type $\rom{1}_{n_j}$ for some $n_j \geq 2$.
    Then there is an exact sequence
    \begin{equation}
        1 \to \langle (-)\otimes \cO_S(F^{(j)}) \mid 1 \leq j \leq r \rangle \to B \xrightarrow{\psi} \prod_{j = 1}^r \MCG(T_{n_j}),
    \end{equation}
    where $\psi$ is the composition of $\res$ and $\Upsilon$.
    Moreover, we have the following description of the image $\Image(\psi)$:
    \begin{enumerate}
        \item Let $p_j \colon \prod_{j=1}^r \MCG(T_{n_j}) \to \MCG(T_{n_j})$ be the projection onto the $j$-th component. Then the image $\Image(\psi)$ is the product of its projections onto the individual components $\Image(\psi) = \prod_{j = 1}^r \Image(p_j \circ \psi)$.
        \item For the curves $\gamma_{\cO_{G_1}(-1)}, \dots, \gamma_{\cO_{G_{n_j}}(-1)}, \gamma_{\cO_{G_1}}, \dots, \gamma_{\cO_{G_{n_j}}}$ on $T_{n_j}$ shown in Figure~\ref{fig:generators-for-image-of-B-intro}, the half twists along them generates the subgroup $\Image(p_j \circ \psi) \subset \MCG(T_{n_j})$.
    \end{enumerate}
\end{theorem}

\begin{figure}[h]
    \centering
    \begin{displaymath}
        \begin{tikzpicture}[scale=1.2]
            \draw[dashed] (0,0)--(10,0);
            \draw[dashed] (0,0)--(0,4);
            \draw[dashed] (0,4)--(10,4);
            \draw[dashed] (10,4)--(10,0);

            \draw[thick] (0, 2)--(1, 2);
            \draw[thick] (1, 2)--(3, 2);
            \draw[thick] (3, 2)--(5, 2);
            \draw[thick] (5, 2)--(6, 2);
            \draw[thick] (8, 2)--(9, 2);
            \draw[thick] (9, 2)--(10, 2);

            \draw[thick] (0, 4)--(1, 2);
            \draw[thick] (1, 2)--(2, 0);
            \draw[thick] (2, 4)--(4, 0);
            \draw[thick] (4, 4)--(5, 2);
            \draw[thick] (5, 2)--(5.5, 1);
            \draw[thick] (8.5, 3)--(9, 2);
            \draw[thick] (9, 2)--(10, 0);

            \draw[dotted, thick] (5+1.5, 2)--(9-1.5, 2);
            \foreach \u in {1, 3, 5, 9}
                {
                    \filldraw[white] (\u, 2) circle (2pt);
                    \draw[black] (\u, 2) circle (2pt);
                }

            \draw(2, 0) node[below]{$\gamma_{\cO_{G_1}}$};
            \draw(4, 0) node[below]{$\gamma_{\cO_{G_2}}$};
            \draw(10, 0) node[below]{$\gamma_{\cO_{G_{n_j}}}$};

            \draw(2, 3) node[above left]{$\gamma_{\cO_{G_1}(-1)}$};
            \draw(1, 3) to[out=-90,in=135](1.5, 2);
            \draw(4, 3) node[above left]{$\gamma_{\cO_{G_2}(-1)}$};
            \draw(3, 3) to[out=-90,in=135](3.5, 2);

            \draw(10, 2) node[right]{$\gamma_{\cO_{G_{n_j}}(-1)}$};

        \end{tikzpicture}
    \end{displaymath}
    \caption{Curves along which the half twists generate $\Image(B \xrightarrow{p_j \circ \psi} \MCG(T_{n_j}))$. The big rectangle illustrates the torus $T_{n_j}$. The top and bottom edges (resp.~the left and right edges) are identified, and the white circles represent the punctures.} \label{fig:generators-for-image-of-B-intro}
\end{figure}

Note that the morphism $\psi \colon B \to \prod_{j=1}^r \MCG(T_{n_j})$ is not surjective as explained in Remark \ref{rm:psi-is-not-surjective}.

\subsection{Organization of the paper}
This paper is organized as follows.
In Section \ref{section:preliminaries}, we recall some basic facts about derived functors and mapping class groups.
In Section \ref{section:fourier-mukai-transforms-and-twist-functors}, we review the theory of relative integral functors and spherical twists.
In Section \ref{section:half-spherical-twists}, we give a sketch of the proof of Theorem \ref{thm:main-theorem-1-half-spherical-twist} and give some examples of half-spherical twists.
These results are applied to elliptic surfaces in Section \ref{section:applications-to-elliptic-surfaces} to show Theorem \ref{thm:main-theorem-2-half-twist} and Theorem \ref{thm:main-theorem-3-description-of-B}.
In Appendix \ref{section:appendix}, we provide detailed arguments that were omitted from the main text.

\subsection{Notations and conventions}\label{subsection:notations-and-conventions}
\subsubsection*{Schemes}
\begin{enumerate}
    \item All schemes are assumed to be quasi-compact and quasi-separated (qcqs, e.g.~noetherian). Consequently, all morphisms of schemes are assumed to be qcqs.
    \item For an algebraically closed field $k$, an \emph{algebraic variety} over $k$ is a separated and reduced scheme of finite type over $k$. A \emph{curve} (resp.~\emph{surface}) is an algebraic variety of dimension $1$ (resp.~$2$).
    \item A \emph{point} on a variety means a closed point unless otherwise specified. The skyscraper sheaf at a point $x$ is denoted by $\cO_x$.
\end{enumerate}

\subsubsection*{Derived categories}
Let $X$ be a scheme.
\begin{enumerate}
    \item $D(\cO_X)$ is the derived category of $\cO_X$-modules.
    \item $D_{qc}(X)$ is the full subcategory of $D(\cO_X)$ consisting of complexes with quasi-coherent cohomology.
    \item $D(X)$ is the full subcategory of $D(\cO_X)$ consisting of complexes with coherent cohomology.
    \item $D_{qc}^*(X)$ and $D^*(X)$ are the full subcategories of $D_{qc}(X)$ and $D(X)$ consisting of complexes with bounded cohomology, where $* = +$, $-$, and $b$ stand for bounded below, bounded above, and bounded, respectively.
    \item $\Perf(X)$ is the full subcategory of $D(\cO_X)$ consisting of perfect complexes.
    \item For a complex $\cF \in D(\cO_X)$ and $i \in \bZ$, its $i$-th cohomology sheaf is denoted by $\cH^i(\cF)$.
\end{enumerate}

\subsubsection*{Derived functors}
\begin{enumerate}
    \item We drop the letters $L$ and $R$ from the notation of derived functors unless otherwise specified. For example, $\CHom(-, -)$ and $(-)\otimes (-)$ denote the derived sheaf Hom and the derived tensor product, respectively.
    \item To avoid the confusion between the ordinary Hom and the derived Hom, we always use the symbol $\Ext^0(-, -)$ for the former and $\Ext^*(-, -)$ for the latter.
\end{enumerate}

\subsection{Acknowledgements}
The author expresses his sincere gratitude to his advisor, Kazushi Ueda, for the helpful discussions, suggestions, and encouragement provided during this research.
He is grateful to his family for their understanding and constant support throughout his life.
He is also thankful to Tomohiro Karube for valuable discussions and for informing him about the paper \cite{karube2023stability}, which includes ideas similar to those presented in the first result, Theorem \ref{thm:main-theorem-1-half-spherical-twist}.
Finally, he appreciates the anonymous referees for their careful reading and valuable comments that helped improve the presentation of this paper.

\section{Preliminaries}\label{section:preliminaries}
\subsection{Tor-independence and Kunneth formula}
Let $S$ be a scheme.
Let $X, Y$ be schemes over $S$.
We say $X$ and $Y$ are \emph{tor-independent over $S$}, if for every $x \in X$ and $y \in Y$ over the same $s \in S$ one has
\begin{equation}
    \Tor_i^{\cO_{S, s}}(\cO_{X, x}, \cO_{Y, y}) = 0 \quad \text{for all } i > 0.
\end{equation}
We say a cartesian diagram
\[
    \begin{tikzcd}
        X \times_S Y \ar[r] \ar[d] & X \ar[d]\\
        Y \ar[r] & S
    \end{tikzcd}
\]
is \emph{tor-independent} if $X$ and $Y$ are tor-independent over $S$.
For example, if either $X$ or $Y$ is flat over $S$, then they are tor-independent over $S$.
The following base change theorem includes the flat base change theorem as a special case.

\begin{theorem}[{\cite[Theorem 22.99]{MR4704076} or \cite[\href{https://stacks.math.columbia.edu/tag/08ET}{Tag 08ET}]{stacks-project}}]\label{thm:tor-independent-base-change-theorem}
    Let
    \[
        \begin{tikzcd}
            X \times_S Y \arrow[r,"u'"]\arrow[d, "f'"'] & X \arrow[d,"f"]\\
            Y \arrow[r, "u"']& S
        \end{tikzcd}
    \]
    be a cartesian diagram of schemes, and let $g = u \circ f' = f \circ u'$.
    The following are equivalent.
    \begin{enumerate}
        \item $X$ and $Y$ are tor-independent over $S$.
        \item The natural base change morphism $u^*f_* \to f'_*u'^*$ is an isomorphism of functors $D_{qc}(X) \to D_{qc}(Y)$.
        \item The natural morphism
              \begin{align}
                  f_*\cF \otimes u_*\cG & \xrightarrow{\eta_g} g_*g^*(f_*\cF \otimes u_*\cG)                                              \\
                                        & \xrightarrow{\sim} g_*(g^*f_*\cF \otimes g^*u_*\cG)                                             \\
                                        & \xrightarrow{\sim} g_*(u'^*f^*f_*\cF \otimes f'^*u^*u_*\cG)                                     \\
                                        & \xrightarrow{g_*(u'^*(\varepsilon_f) \otimes f'^*(\varepsilon_u))} g_*(u'^*\cF \otimes f'^*\cG)
              \end{align}
              is an isomorphism for all $\cF \in D_{qc}(X)$ and $\cG \in D_{qc}(Y)$.
              Here $\eta_g \colon \id \to g_*g^*$ is the adjunction unit, and $\varepsilon_f \colon f^*f_* \to \id$ and $\varepsilon_u \colon u^*u_* \to \id$ are the adjunction counits.
    \end{enumerate}
\end{theorem}
\begin{remark}
    A cartesian square satisfying these conditions is called \emph{exact cartesian} in \cite{MR2238172}.
\end{remark}

As an application of the tor-independent base change theorem, we have the following Kunneth formula.
Let $k$ be a field, $f \colon Y \to X$ be a morphism of $k$-schemes, and $g = f\times f \colon Y \times_k Y \to X \times_k X$.
Consider the diagram of cartesian squares
\begin{equation}\label{eq:Kunneth-formula-situation}
    \begin{tikzcd}
        Y\times_k Y \arrow[r,"g_2"]\arrow[d, "g_1"] & X\times_k Y \arrow[d,"f_2"] \arrow[r, "p_2"] &Y\arrow[d, "f"]\\
        Y \times_k X \arrow[r, "f_1"]\arrow[d, "p_1"]& X \times_k X \arrow[r, "\pi_2"]\arrow[d, "\pi_1"]&X\\
        Y \arrow[r, "f"]&X&
    \end{tikzcd}
\end{equation}
in which $\pi_1, \pi_2, p_1, p_2$ are the natural projections, and $f_1, f_2, g_1, g_2$ are the natural morphisms induced by $f$.
All the squares here are tor-independent by the flatness of the projections and \cite[Lemma 2.25]{MR2238172}.

\begin{proposition}[Kunneth formula]\label{prop:Kunneth-formula}
    Let $k$ be a field, $f \colon Y \to X$ be a morphism of $k$-schemes, and $g = f\times f \colon Y \times_k Y \to X \times_k X$.
    Then for $\cF, \cG \in D_{qc}(Y)$, there is a natural isomorphism
    \begin{equation}
        f_*\cF \boxtimes f_*\cG \xrightarrow{\sim}g_*(\cF \boxtimes \cG).
    \end{equation}
\end{proposition}
\begin{proof}
    Since the natural maps $\pi_i^*j_* \to f_{i*}p_i^*$ for $i = 1, 2$ are isomorphisms by Theorem \ref{thm:tor-independent-base-change-theorem}(2), we have an isomorphism
    \begin{equation}
        f_*\cF \boxtimes f_*\cG \xrightarrow{\sim} f_{1*}p_1^*\cF \otimes f_{2*}p_2^*\cG.
    \end{equation}
    In addition, there is a natural isomorphism
    \begin{equation}
        f_{1*}p_1^*\cF \otimes f_{2*}p_2^*\cG \xrightarrow{\sim} g_*(\cF \boxtimes \cG)
    \end{equation}
    by Theorem \ref{thm:tor-independent-base-change-theorem}(3).
    This finishes the proof.
\end{proof}

\subsection{Grothendieck duality}
Next, we briefly recall Grothendieck duality.
Let $f \colon X \to Y$ be a proper morphism between noetherian schemes.
Then the push-forward functor $f_* \colon D_{qc}^+(X) \to D_{qc}^+(Y)$ has the right adjoint $f^! \colon D^+(Y) \to D^+(X)$, which is called the \emph{twisted inverse image functor} or \emph{upper shriek functor} \cite[\href{https://stacks.math.columbia.edu/tag/0A9Y}{Tag 0A9Y}]{stacks-project}.
The adjoint isomorphism
\begin{equation}
    \Hom_{D^+(Y)}(f_*\cF, \cG) \cong \Hom_{D^+(X)}(\cF, f^!\cG)
\end{equation}
have the following refinement for suitable $\cF$ and $\cG$.

\begin{theorem}[Grothendieck duality, {\cite[\href{https://stacks.math.columbia.edu/tag/0AU3}{Tag 0AU3}]{stacks-project}}]\label{grothendieck_duality}
    Let $f \colon X \to Y$ be a proper morphism between noetherian schemes.
    Then for $\cF \in D^-(X)$ and $\cG \in D_{qc}^+(Y)$ there is a natural isomorphism
    \begin{equation}\label{eq:grothendieck_duality}
        f_*\CHom_X(\cF, f^!\cG) \xrightarrow{\sim} \CHom_Y(f_*\cF, \cG).
    \end{equation}

\end{theorem}
\begin{remark}\label{remark:grothendieck-duality-isomorphism}
    For any $f \colon X \to Y$ and $\cA, \cB \in D(\cO_X)$, consider the morphism
    \begin{equation}\label{eq:natural-morphism1}
        f_*\CHom_X(\cA, \cB) \otimes f_*\cA \to f_*(\CHom_X(\cA, \cB) \otimes \cA) \to f_*\cB.
    \end{equation}
    Here the first one is the adjoint to
    \begin{equation}
        f^*(f_*\CHom_X(\cA, \cB) \otimes f_*\cA) \cong f^*f_*\CHom_X(\cA, \cB) \otimes f^*f_*\cA \xrightarrow{\varepsilon_f \otimes \varepsilon_f} \CHom_X(\cA, \cB) \otimes \cA
    \end{equation}
    with $\varepsilon_f$ being the counit of the adjunction $f^* \dashv f_*$.
    It is called the \emph{relative cup product} (see Definition \ref{def:relative-cup-product}).
    The second one is induced by the natural evaluation.
    By taking the adjoint to \eqref{eq:natural-morphism1} we have another natural morphism
    \begin{equation}\label{eq:push-forward-hom}
        f_*\CHom_X(\cA, \cB) \to \CHom_Y(f_*\cA, f_*\cB).
    \end{equation}
    The isomorphism of Grothendieck duality is then given by the following composition \cite[\href{https://stacks.math.columbia.edu/tag/0A9D}{Tag 0A9D}]{stacks-project}:
    \begin{equation}
        f_*\CHom_X(\cF, f^!\cG) \xrightarrow{\eqref{eq:push-forward-hom}} \CHom_Y(f_*\cF, f_*f^!\cG) \xrightarrow{(\text{counit of } f_* \dashv f^!) \circ -} \CHom_Y(f_*\cF, \cG).
    \end{equation}
\end{remark}

\subsection{Relatively perfect complexes}
In order to introduce a result from \cite{MR3720794} in the next section (Theorem \ref{thm:adjoint-to-integral-functors}), we need to recall the notion of relatively perfect complexes.
The original concept was developed by Illusie in a general setting in SGA6 \cite{zbMATH03363721}, and \cite{MR3720794} adopts an equivalent definition for certain special cases.
The following definition is a simplified version of Illusie's one under the noetherian assumption\footnote{In \cite[Definition 4.1]{zbMATH03363721}, a complex $\cF$ is said to be $f$-perfect if it is \emph{pseudo-coherent} and satisfies the condition (2) in Definition \ref{def:relative-perfection}. For a noetherian scheme $X$, \cite[\href{https://stacks.math.columbia.edu/tag/08E4}{Tag 08E4}]{stacks-project} says that the pseudo-coherence is equivalent to the condition (1) in our definition. Therefore, Definition \ref{def:relative-perfection} is equivalent to \cite[Definition 4.1]{zbMATH03363721} when $X$ is noetherian.}, which covers the cases considered in \cite{MR3720794}.
\begin{definition}[{\cite[Definition 4.1]{zbMATH03363721}}]\label{def:relative-perfection}
    Let $f \colon X \to Y$ be a morphism of finite type between noetherian schemes.
    A complex $\cF \in D(\cO_X)$ is said to be \emph{$f$-perfect (or relatively perfect to $f$)} if
    \begin{enumerate}
        \item $\cF \in D^-(X)$, and
        \item $\cF$ is locally isomorphic in $D(f^{-1}\cO_Y)$ to a bounded complex of flat $f^{-1}\cO_Y$ modules.
    \end{enumerate}
\end{definition}
For example, if $f$ is flat, then every perfect complex on $X$ is $f$-perfect.
Note that $f$-perfect complexes form a triangulated subcategory of $D^-(X)$.
We present a useful criterion from \cite{MR4604981} for relative perfection.
\begin{proposition}\label{prop:crirtion_for_relative_perfection}
    Let $f \colon X \to Y$ be a morphism of finite type between noetherian schemes.
    For every complex $\cF \in D(\cO_X)$, the following are equivalent:
    \begin{enumerate}
        \item $\cF$ is $f$-perfect.
        \item  $\cF \in D^-(X)$, and $\cF \otimes f^*\cG$ is bounded for every $\cG \in D_{qc}^b(Y)$.
    \end{enumerate}
    In particular, an $f$-perfect complex is automatically in $D^b(X)$.
\end{proposition}
\begin{proof}
    The condition (2) is used as a definition of relative perfection \cite[(3.1)]{MR4604981}\footnote{Here we replaced the pseudo-coherence condition in \cite[(3.1)]{MR4604981} with the condition $\cF \in D^-(X)$ as above.}.
    Then \cite[Corollary 4.2]{MR4604981} and the subsequent remark establish the equivalence between their definition and Illusie's one when $f$ is \emph{pseudo-coherent}, which is the case when $f$ is of finite type between noetherian schemes \cite[\href{https://stacks.math.columbia.edu/tag/0684}{Tag 0684}]{stacks-project}.
    In particular, taking $\cG = \cO_Y$ in (2) shows $\cF \in D^b(X)$.
\end{proof}

Relative perfection is compatible with
proper push-forward
in the following sense.
\begin{proposition}[{\cite[Proposition 4.8]{zbMATH03363721}}]\label{prop:proper-push-forward-of-f-perfect}
    Let
    \[
        \begin{tikzcd}
            X \ar[rr, "f"] \ar[rd, "g"']& & Y \ar[ld, "h"]\\
            & S &
        \end{tikzcd}
    \]
    be a diagram of noetherian schemes and morphisms of finite type.
    If $f$ is proper, then $f_*$ maps $g$-perfect complexes to $h$-perfect complexes.
\end{proposition}

\subsection{Dehn twists, half twists, and mapping class groups of surfaces}\label{subsection:Dehn-twists-and-mapping-class-groups}
Let $\Sigma$ be a compact oriented real surface with (or without) boundary $\partial \Sigma$
and a set of finite marked points $P \subset \Sigma \setminus \partial \Sigma$.
An \emph{ambient isotopy} (or just \emph{isotopy} by abuse of notation) of $(\Sigma, P)$ is
a homotopy $H \colon \Sigma \times [0, 1] \to \Sigma$ through homeomorphisms fixing $P \cup \partial \Sigma$ pointwise.

\begin{definition}
    The \emph{extended mapping class group} $\MCG^{\pm}(\Sigma, P)$ of $\Sigma$ with marked points $P$ is
    the group of (ambient) isotopy classes of all homeomorphisms of $\Sigma$
    fixing $\partial \Sigma$ pointwise
    and preserving $P$ as a set.
    The \emph{mapping class group} $\MCG(\Sigma, P)$ is
    the subgroup of $\MCG^{\pm}(\Sigma, P)$ consisting of all orientation-preserving homeomorphisms (modulo isotopy).
\end{definition}
\begin{remark}
    The pair $(\Sigma, P)$ is sometimes identified with the \emph{surface with punctures} $\Sigma \setminus P$.
    With this identification, the mapping class group $\MCG(\Sigma, P)$ is often denoted by $\MCG(\Sigma \setminus P)$.
\end{remark}

The most fundamental ways to construct mapping classes are \emph{Dehn twists} and \emph{half twists}.
Let us recall these notions.

A \emph{closed curve} or a \emph{loop} on $\Sigma$ is an immersion $\gamma \colon S^1 \to \Sigma$ whose image lies inside the interior of $\Sigma \setminus P$.
An \emph{arc} on $\Sigma$ is an immersion $\gamma \colon [0, 1] \to \Sigma$ such that the image of the open interval $(0, 1)$ is in the interior of $\Sigma \setminus P$, and $\gamma(0), \gamma(1) \in P$.
A \emph{curve} on $\Sigma$ is either a loop or an arc.
A curve is said to be \emph{simple} if it has no self-intersections.
A simple loop is said to be \emph{non-separating} if $\Sigma \setminus \Image(\gamma)$ is connected.

\begin{remark}
    We often identify a curve $\gamma$ with its image $\Image(\gamma)$.
\end{remark}

Let $\gamma \colon S^1 \to \Sigma$ be a simple closed curve and let $N \subset \Sigma$ be a tubular neighborhood of $\gamma$.
We choose an orientation-preserving homeomorphism $i \colon S^1 \times [0, 1]\xrightarrow{\sim} N$ such that $i \vert_{S^1 \times \{1/2\} } = \gamma$.
Then the \emph{Dehn twist along $\gamma$} is a homeomorphism $t_\gamma \colon \Sigma \to \Sigma$ (or its mapping class) defined by
\begin{equation}
    t_\gamma(x) =
    \begin{cases}
        i(z \cdot e^{2\pi \sqrt{-1} s}, s) & (x = i(z, s) \in N),        \\
        x                                  & (x \in \Sigma \setminus N),
    \end{cases}
\end{equation}
where $z \in S^1 = \{z \in \bC \mid |z| = 1\}$ and $s \in [0, 1]$ (see Figure \ref{fig:Dehn-twist}).
Although the map $t_\gamma$ depends on the choice of $N$ and $i$, its mapping class depends only on the isotopy class of $\gamma$.

\begin{figure}[h]
    \centering
    \begin{displaymath}
        \begin{tikzpicture}[scale=0.8]
            \draw (0,0+5)--(6,0+5);
            \draw (0,3+5)--(6,3+5);

            \draw[thick] (0.5, 1.5+5)--(6.5, 1.5+5);

            \draw (0, 1.5+5) circle[x radius=0.5, y radius=1.5];

            \draw(6, 0+5)arc(-90:90:0.5 and 1.5);
            \draw[dashed](6, 3+5)arc(90:270:0.5 and 1.5);

            \draw[thick](3, 0+5)arc(-90:90:0.5 and 1.5);
            \draw[thick, dashed](3, 3+5)arc(90:270:0.5 and 1.5);

            \draw(3, 3+5) node[above]{$\gamma$};

            \draw(3, -0.5+5) node[below]{\Large$\downarrow$};

            \draw (0,0)--(6,0);
            \draw (0,3)--(6,3);

            \draw (0, 1.5) circle[x radius=0.5, y radius=1.5];

            \draw(6, 0)arc(-90:90:0.5 and 1.5);
            \draw[dashed](6, 3)arc(90:270:0.5 and 1.5);

            \draw[thick](0.5, 1.5) to[out=0,in=100](2.5, 0);
            \draw[thick, dashed](2.5, 0) to[out=70,in=-120](3.5, 3);
            \draw[thick](3.5, 3) to[out=-60,in=180](6.5, 1.5);
        \end{tikzpicture}
    \end{displaymath}
    \caption{The Dehn twist $t_\gamma$ along a simple loop $\gamma$.}
    \label{fig:Dehn-twist}
\end{figure}
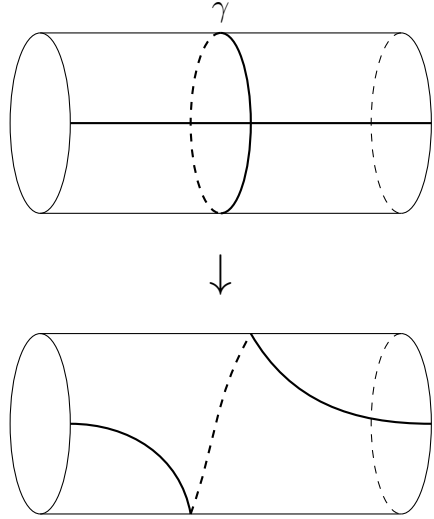

Similarly, the half twist along a simple arc is defined as follows. Given a simple arc $\gamma$ on $\Sigma$, choose an orientation-preserving embedding $i \colon D \to \Sigma$ of the unit disk $D = \{z \in \bC \mid |z| \leq 1\}$ satisfying $i(t - 1/2) = \gamma(t)$ for all $t \in [0, 1] \subset D$.
Then we define the \emph{half twist along $\gamma$} to be the homeomorphism $h_\gamma \colon \Sigma \to \Sigma$ (or its mapping class) satisfying
\begin{equation}
    h_\gamma(x) = \begin{cases}
        i(z \cdot e^{2\pi \sqrt{-1} |z|}) & (x = i(z) \in i(D)),           \\
        x                                 & (x \in \Sigma \setminus i(D)),
    \end{cases}
\end{equation}
which is illustrated in Figure \ref{fig:half-twist}.
Its mapping class is also determined by the isotopy class of $\gamma$.

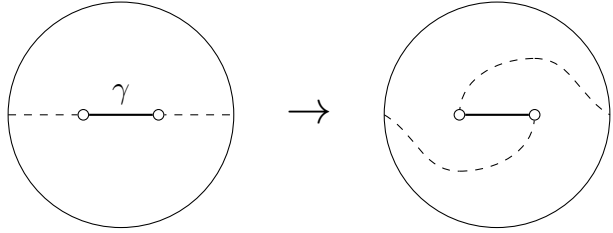
\begin{figure}[h]
    \centering
    \begin{displaymath}
        \begin{tikzpicture}
            \draw (0, 1.5) circle[radius=1.5];
            \draw[thick] (-0.5, 1.5)--(0.5, 1.5);
            \draw[dashed] (-1.5, 1.5)--(-0.5, 1.5);
            \draw[dashed] (0.5, 1.5)--(1.5, 1.5);

            \foreach \u in {-0.5, 0.5}
                {
                    \filldraw[white] (\u, 1.5) circle (2pt);
                    \draw[black] (\u, 1.5) circle (2pt);
                }
            \draw(0, 1.5) node[above]{$\gamma$};

            \draw(2.5, 1.5) node{\Large$\rightarrow$};

            \draw (0+5, 1.5) circle[radius=1.5];
            \draw[thick] (-0.5+5, 1.5)--(0.5+5, 1.5);

            \draw[dashed](-1.5+5, 1.5) to[out=-30,in=180](-0.5+5, 0.75);
            \draw[dashed](-0.5+5, 0.75) to[out=0,in=-90](0.5+5, 1.5);

            \draw[dashed](-0.5+5, 1.5) to[out=90,in=180](0.5+5, 2.25);
            \draw[dashed](0.5+5, 2.25) to[out=0,in=150](1.5+5, 1.5);

            \foreach \u in {-0.5, 0.5}
                {
                    \filldraw[white] (\u+5, 1.5) circle (2pt);
                    \draw[black] (\u+5, 1.5) circle (2pt);
                }
        \end{tikzpicture}
    \end{displaymath}
    \caption{The half twist $h_\gamma$ along a simple arc $\gamma$.}
    \label{fig:half-twist}
\end{figure}

\subsection{The actions of mapping class groups on the fundamental groups}
Let $\Sigma$, $\partial \Sigma$, and $P$ be as above.
Assume $\Sigma$ is connected, and fix a base point $x \in \Sigma \setminus (P \cup \partial \Sigma)$.
A homeomorphism $f \colon \Sigma \to \Sigma$ fixing $P$ as a set induces an isomorphism $f_* \colon \pi_1(\Sigma \setminus P, x) \to \pi_1(\Sigma \setminus P, f(x)), $ which in turn induces an automorphism $[E] \mapsto [\gamma_{-1} * f(E) * \gamma]$ of $\pi_1(\Sigma \setminus P, x) = \pi_1(\Sigma \setminus P)$ by choosing a path $\gamma$ from $x$ to $f(x)$.
The class of this automorphism in $\Out(\pi_1(\Sigma \setminus P)) \coloneqq \Aut(\pi_1(\Sigma \setminus P))/\Inn(\pi_1(\Sigma \setminus P))$, where $\Inn(\pi_1(\Sigma \setminus P))$ is the inner automorphism group of $\pi_1(\Sigma \setminus P)$, does not depend on the choice of $\gamma$, and the resulting group homomorphism will be denoted by
\begin{equation}
    \sigma \colon \MCG^{\pm}(\Sigma, P) \to \Out(\pi_1(\Sigma \setminus P)).
\end{equation}
There is the following remarkable theorem for this map.
\begin{theorem}[Dehn--Nielsen--Baer theorem, {\cite[Theorem 8.8]{MR2850125}}]\label{thm:Dehn--Nielsen--Baer}
    Let $\Out^\star(\pi_1(\Sigma \setminus P))$ be the subgroup of $\Out(\pi_1(\Sigma \setminus P))$
    consisting of the elements that preserve the set of conjugacy classes of the simple closed curves surrounding each point in $P$.
    If $\partial \Sigma = \emptyset$ and $\chi(\Sigma \setminus P)<0$, then the natural map
    \begin{equation}
        \sigma \colon \MCG^{\pm}(\Sigma, P) \to \Out^\star(\pi_1(\Sigma \setminus P))
    \end{equation}
    is an isomorphism.
\end{theorem}
We will use the injectivity of $\sigma$ to identify certain elements of mapping class groups.

\section{Fourier--Mukai transforms and twist functors}\label{section:fourier-mukai-transforms-and-twist-functors}
In this section, we briefly review (relative) Fourier--Mukai transforms and twist functors.
Although the content presented here does not contain new results, for the reader's convenience we have collected the basic facts in sufficiently general settings for our purpose.
We refer the reader to \cite{MR2238172, MR2505443} for relative Fourier--Mukai transforms and \cite{MR1831820, MR3692883} for twist functors.
\subsection{Fourier--Mukai transforms}\label{subsection:fourier-mukai-transforms}
Let $T$, $X$, and $Y$ be schemes.
Let\ $p \colon X \to T$ and $q \colon Y \to T$ be flat and separated morphisms.
For an object $\cP \in D_{qc}(X \times_T Y)$, we define the \emph{(relative) integral functor $\Phi_{\cP} = \Phi^{X \to Y}_P \colon D_{qc}(X) \to D_{qc}(Y)$ with kernel $\cP$} by
\begin{equation}
    \Phi_{\cP}(-) = \pi_{Y*}(\pi_{X}^*(-) \otimes \cP),
\end{equation}
where $\pi_X \colon X \times_T Y \to X$ and $\pi_Y \colon X \times_T Y \to Y$ be the natural projections.
If an integral functor $\Phi_{\cP}$ is an equivalence, then we call $\Phi_{\cP}$ a \emph{Fourier--Mukai transform} and $\cP$ a \emph{Fourier--Mukai kernel}.

\begin{example}
    \begin{enumerate}
        \item Let $f \colon X \to Y$ be a morphism of $T$-schemes. Let $\Gamma \subset X \times_T Y$ be the graph of $f$ and $\overline{\Gamma} \subset Y \times_T X$ be its transpose. Then we have $f_* \cong \Phi_{\cO_{\Gamma}} \colon D_{qc}(X) \to D_{qc}(Y)$ and $f^* \cong \Phi_{\cO_{\overline{\Gamma}}} \colon D_{qc}(Y) \to D_{qc}(X)$.
        \item Let $\cL \in \Pic(X)$ be a line bundle on $X$ and $\Delta = \Delta_T \colon X \to X \times_T X$ be the diagonal morphism. Then we have $\cL \otimes (-) \cong \Phi_{\Delta_*\cL} \colon D_{qc}(X) \to D_{qc}(X)$.
    \end{enumerate}
\end{example}

The composition $\Phi_{\cQ} \circ \Phi_{\cP} \colon D_{qc}(X) \to D_{qc}(Z)$ of two integral functors $\Phi_{\cP} \colon D_{qc}(X) \to D_{qc}(Y)$ and $\Phi_{\cQ} \colon D_{qc}(Y) \to D_{qc}(Z)$  is isomorphic to an integral functor $\Phi_{\cQ * \cP}$ \cite[\href{https://stacks.math.columbia.edu/tag/0FYS}{Tag 0FYS}]{stacks-project}, where the kernel $\cQ * \cP \in D_{qc}(X \times_T Z)$ is the \emph{convolution} of $\cP$ and $\cQ$ defined as follows:
\begin{equation}
    \cQ * \cP = \pi_{XZ*}(\pi_{XY}^*\cP \otimes \pi_{YZ}^*\cQ).
\end{equation}
Here $\pi_{XY}$, $\pi_{YZ}$, and $\pi_{XZ}$ are the natural projections from $X \times_T Y \times_T Z$ to $X \times_T Y$, $Y \times_T Z$, and $X \times_T Z$, respectively.

The adjoint functors to an integral functor are also given by integral functors under a mild finiteness assumption.
\begin{theorem}[{\cite[Theorem 1.1]{MR3720794}}]\label{thm:adjoint-to-integral-functors}
    Let $T$ be a noetherian scheme.
    Let $p \colon X \to T$ and $q \colon Y \to T$ be flat and separated morphisms of finite type, and let $\cP \in D_{qc}(X \times_T Y)$.
    \begin{enumerate}
        \item Suppose that $Y$ is quasi-projective over $T$. If $\cP$ is $\pi_X$-perfect and has proper support over $X$, then the integral functor $\Phi_{\cP_L} \colon D_{qc}(Y) \to D_{qc}(X)$ with the kernel $\cP_L = \CHom(\cP, \pi_X^!\cO_X)$ is the left adjoint functor to $\Phi_{\cP}$.
        \item Suppose that $X$ is quasi-projective over $T$. If $\cP$ is $\pi_Y$-perfect and has proper support over $Y$, then the integral functor $\Phi_{\cP_R} \colon D_{qc}(Y) \to D_{qc}(X)$ with the kernel $\cP_R = \CHom(\cP, \pi_Y^!\cO_Y)$ is the right adjoint functor to $\Phi_{\cP} $.
    \end{enumerate}
\end{theorem}

We will mainly focus on the case where $X = Y$, $p = q$, and $\Phi_{\cP}$ is an equivalence (i.e.~a Fourier--Mukai transform).
Let us introduce the group of invertible integral kernels $\FM_T(X)$.
By the projection formula we have $\cP * \cO_{\Delta} \cong \cP \cong \cO_{\Delta} * \cP$ for the structure sheaf $\cO_{\Delta} = \Delta_{T*}\cO_X$ of the diagonal $\Delta_T \colon X \hookrightarrow X \times_T X$ and any complex $\cP \in D_{qc}(X \times_T X)$.
This means that $\cO_{\Delta}$ behaves as the unit element for the convolution.
We can also check that the convolution is associative. 
Based on these observations we define:
\begin{definition}
    Let $X \to T$ be a flat and separated morphism.
    The group of Fourier--Mukai kernels $\FM_T(X)$ is
    \begin{equation}
        \FM_T(X) \coloneqq \{\cP \in D_{qc}(X \times_T X) \mid \text{$\cP$ is invertible w.r.t.~ the convolution}\}/\cong
    \end{equation}
    with the group operation being the convolution.
\end{definition}

There is a natural homomorphism to the group of autoequivalences of $D_{qc}(X)$
\begin{equation}\label{eq:FM-to-Auteq}
    \FM_T(X) \to \Auteq(D_{qc}(X)), \quad \cP \mapsto \Phi_{\cP}
\end{equation}
by definition.
In a nice situation, the functor $\Phi_{\cP}$ restricts to an autoequivalence of the full subcategory $D^b(X) \subset D_{qc}(X)$ and the map $\cP \mapsto \Phi_{\cP}$ is injective by the following arguments.
\begin{lemma}\label{lem:autoequivalence-preserve-D^b}
    Let $X$ be a separated noetherian scheme.
    Then every autoequivalence $\Phi \in \Auteq D_{qc}(X)$ preserves the full subcategory $D^b(X) \subset D_{qc}(X)$.
\end{lemma}
\begin{proof}
    The categorical characterization \cite[Proposition 6.9]{MR2434186} of the subcategory $D^b_{qc}(X) \subset D_{qc}(X)$ implies $\Phi(D^b_{qc}(X)) = D^b_{qc}(X)$.
    The induced autoequivalence $\Phi \colon D^b_{qc}(X) \xrightarrow{\sim} D^b_{qc}(X)$ preserves the subcategory of compact objects $D^b_{qc}(X)^c$, which coincides with $D^b(X)$ by \cite[Corollary 6.17]{MR2434186}.
    See also the discussion in \cite[Proposition 7.4]{MR3861804}.
\end{proof}
\begin{lemma}\label{lem:uniqueness-of-FM-kernel}
    Let $k$ be a field and $X$ be a quasi-projective scheme over $k$.
    Then for every Fourier--Mukai transform $\Phi_{\cP} \colon D^b(X) \xrightarrow{\sim} D^b(X)$ the kernel $\cP \in D_{qc}(X \times_k X)$ is unique up to isomorphism.
\end{lemma}
\begin{proof}
    Let $\cP, \cQ \in D_{qc}(X \times_k X)$ be two Fourier--Mukai kernels whose Fourier--Mukai transforms on $D^b(X)$ are isomorphic:
    \begin{equation}
        \Phi_{\cP} \cong \Phi_{\cQ} \colon D^b(X) \xrightarrow{\sim} D^b(X).
    \end{equation}
    It yields an isomorphism of functors
    \begin{equation}
        \Phi_{\cP}\vert_{\Perf(X)} \cong \Phi_{\cQ}\vert_{\Perf(X)} \colon \Perf(X) \to D_{qc}(X)
    \end{equation}
    by restricting the source and extending the target.
    In addition, for a fixed ample line bundle $\cO_X(1)$, $\Phi_{\cP}$ satisfies the condition
    \begin{eqnarray}
        \Ext^j(\Phi_{\cP}(\cO_X(n)), \Phi_{\cP}(\cO_X(m))) = 0
    \end{eqnarray}
    for all $n, m \in \bZ$ and $j < 0$, as $\Phi_{\cP}$ is an equivalence.
    Then we can apply \cite[Theorem 1.5]{MR3556457} to deduce that $\cP \cong \cQ$.
\end{proof}
\begin{corollary}\label{cor:FM-to-Auteq}
    Let $k$ be a field and $X$ be a quasi-projective scheme over $k$.
    Then there is an injective group homomorphism
    \begin{equation}
        \FM_k(X) \hookrightarrow \Auteq(D^b(X)), \quad \cP \mapsto \Phi_{\cP}.
    \end{equation}
\end{corollary}
\begin{proof}
    By Lemma \ref{lem:autoequivalence-preserve-D^b} there is a well-defined group homomorphism
    \begin{equation}
        \Auteq(D_{qc}(X)) \to \Auteq(D^b(X)), \quad \Phi \mapsto \Phi\vert_{D^b(X)}.
    \end{equation}
    It induces a group homomorphism $\FM_k(X) \to \Auteq(D_{qc}(X)) \to \Auteq(D^b(X))$, which is injective by Lemma \ref{lem:uniqueness-of-FM-kernel}.
\end{proof}

The next proposition describes the behavior of Fourier--Mukai transforms when we forget the base.
\begin{proposition}\label{prpo:forget-base}
    Let $X \to T$ and $T \to S$ be flat and separated morphisms.
    Then there is an injective group homomorphism
    \begin{equation}
        \FM_T(X) \hookrightarrow \FM_S(X), \quad \cP \mapsto \iota_* \cP,
    \end{equation}
    where $\iota \colon X \times_T X \to X \times_S X$ is the natural morphism.
    In addition, the following diagram is commutative:
    \[
        \begin{tikzcd}
            \FM_T(X) \ar[r, "\iota_*"] \ar[rd, out=-90, in=180, "\Phi_{(-)}"'] & \FM_S(X) \ar[d, "\Phi_{(-)}"] \\
            & \Auteq(D_{qc}(X)).
        \end{tikzcd}
    \]
\end{proposition}
\begin{proof}
    First of all, the natural morphism $\iota \colon X \times_T X \to X \times_S X$ is a closed immersion because there is a cartesian diagram
    \[
        \begin{tikzcd}
            X \times_T X \ar[r] \ar[d, "\iota"'] & T \ar[d, "\text{diagonal}"]\\
            X \times_S X \ar[r] & T \times_S T
        \end{tikzcd}
    \]
    and $T \to S$ is separated.

    Let $\Delta_T \colon X \to X \times_T X$ and $\Delta_S \colon X \to X \times_S X$ be the diagonal morphisms.
    For every $\cP, \cQ \in D_{qc}(X \times_T X)$ we have $\iota_*(\cP * \cQ) \cong (\iota_* \cP) * (\iota_* \cQ)$ by the projection formula. We also have $\iota_* \Delta_{T*}\cO_X \cong \Delta_{S*}\cO_X$.
    Thus, there is a well-defined group homomorphism  $\FM_T(X) \to \FM_S(X), \cP \mapsto \iota_* \cP$.
    To show the injectivity, let $\cP \in \FM_T(X)$ and assume $\iota_* \cP \cong \Delta_{S*}\cO_X$.
    Since $\iota$ is a closed immersion, the isomorphism $\iota_* \cP \cong \Delta_{S*}\cO_X \cong \iota_*\Delta_{T*}\cO_X$ implies that $\cP$ has non-vanishing cohomology sheaves only in degree zero and hence $\cP \cong \Delta_{T*}\cO_X$, which means that the homomorphism is injective.

    The commutativity of the last diagram follows from the projection formula.
\end{proof}

The following result says that relative Fourier--Mukai transforms are compatible with base change in some sense.
In particular, we can ``restrict'' them to (the derived category of) each fiber of $X \to T$.
This construction has appeared in the literature such as \cite{MR1972146, MR2238172, MR2505443, MR2593258} and will play a central role in our arguments.
We give a formulation that applies to our setting.
\begin{proposition}\label{prop:restriction-to-fiber}
    Let $X \to T$ be a flat and separated morphism, $T' \to T$ be an arbitrary morphism, and $X' = X \times_T T'$ be the base change.
    Let $f \colon X' \to X$ and $\varphi \colon X' \times_{T'} X' \to X \times_T X$ be the natural morphisms.
    Then we have a group homomorphism
    \begin{equation}
        \FM_T(X) \to \FM_{T'}(X'), \quad \cP \mapsto \varphi^* \cP.
    \end{equation}
    Moreover, the Fourier--Mukai transforms $\Phi_{\cP}$ and $\Phi_{\varphi^* \cP}$ fit into the following commutative diagram:
    \begin{equation}\label{diagram:restriction-to-fiber}
        \begin{tikzcd}
            D_{qc}(X') \ar[r, "f_*"] \ar[d, "\Phi_{\varphi^* \cP}"'] & D_{qc}(X) \ar[r, "f^*"]\ar[d, "\Phi_{\cP}"] & D_{qc}(X') \ar[d,  "\Phi_{\varphi^* \cP}"]\\
            D_{qc}(X') \ar[r, "f_*"'] & D_{qc}(X) \ar[r, "f^*"'] & D_{qc}(X').
        \end{tikzcd}
    \end{equation}
\end{proposition}
\begin{proof}
    Let $\Delta_T \colon X \to X \times_T X$ and $\Delta_{T'} \colon X' \to X' \times_{T'} X'$ be the diagonal morphisms.
    We need to show that
    \begin{enumerate}
        \item $(\varphi^* \cP) * (\varphi^* \cQ) \cong \varphi^*(\cP * \cQ)$ for $\cP, \cQ \in D_{qc}(X \times_T X)$ and
        \item $\varphi^*\Delta_{T*}\cO_X \cong \Delta_{T'*}\cO_{X'}$.
    \end{enumerate}
    Let $\psi \colon X' \times_{T'} X' \times_{T'} X' \to X \times_T X \times_T X$ be the natural morphism.
    For $i, j \in \{1, 2, 3\}$ let $\pi_{ij} \colon X \times_T X \times_T X \to X \times_T X$ and $p_{ij} \colon X' \times_{T'} X' \times_{T'} X' \to X' \times_{T'} X'$ be the natural projections to the $(i, j)$ components.
    Then for $\cP, \cQ \in D_{qc}(X \times_T X)$ we have
    \begin{align}
        (\varphi^* \cP) * (\varphi^* \cQ) & =p_{13*}(p_{12}^* \varphi^* \cP \otimes p_{23}^* \varphi^* \cQ)   \\
                                          & \cong p_{13*}(\psi^*\pi_{12}^* \cP \otimes \psi^* \pi_{23}^* \cQ) \\
                                          & \cong p_{13*}\psi^*(\pi_{12}^* \cP \otimes \pi_{23}^* \cQ)        \\
                                          & \cong \varphi^*\pi_{13*}(\pi_{12}^* \cP \otimes \pi_{23}^* \cQ)   \\
                                          & = \varphi^*(\cP * \cQ).
    \end{align}
    Here $p_{13*}\psi^* \cong \varphi^*\pi_{13*}$ holds by Theorem \ref{thm:tor-independent-base-change-theorem}, since the cartesian square of the diagram
    \[
        \begin{tikzcd}
            X' \times_{T'} X' \times_{T'} X' \ar[r, "\psi"] \ar[d, "p_{13}"] & X \times_{T} X \times_{T} X \ar[d, "\pi_{13}"] \\
            X' \times_{T'} X' \ar[r, "\varphi"]& X \times_{T} X
        \end{tikzcd}
    \]
    is tor-independent as $\pi_{13}$ is flat.
    In addition, applying \cite[Lemma 2.25]{MR2238172} to the diagram
    \[
        \begin{tikzcd}
            X' \ar[r, "f"] \ar[d, "\Delta_{T'}"] &  X \ar[d, "\Delta_T"] \\
            X' \times_{T'} X' \ar[r, "\varphi"]\ar[d] & X \times_{T} X \ar[d]\\
            T' \ar[r] &T
        \end{tikzcd}
    \]
    we can show the upper square is tor-independent, and hence (2) holds.
    Finally, the commutativity of the last diagram \eqref{diagram:restriction-to-fiber} is a consequence of a standard argument using the projection formula
    and flat base change. See \cite[Lemma 2.41]{MR2238172}.
\end{proof}

\begin{example}\label{ex:relative-fourier-mukai}
    Let $X \to T$ be a flat family of varieties and $X_0$ be the fiber at a closed point $0 \in T$.
    \begin{enumerate}
        \item Let $f \colon X \to X$ be an automorphism of $X$ over $T$.
              It induces an automorphism $f_0 \colon X_0 \to X_0$ of $X_0$.
              The functors $f_*$ and $f^*$ are relative integral functors with respect to $T$. Their restrictions to the fiber are $f_{0*}$ and $f_0^*$, respectively.
        \item Let $\cL$ be a line bundle on $X$.
              The autoequivalence $\cL \otimes (-)$ of $D^b(X)$ is a relative integral functor with respect to $T$, and its restriction to the fiber is $\cL\vert_{X_0} \otimes (-)$.
    \end{enumerate}
\end{example}

We present a criterion to determine when an object $\cP \in D_{qc}(X \times_T X)$ belongs to $\FM_T(X)$.
Roughly speaking, this criterion says it is enough to check that $\Phi_{\cP}$ is an equivalence and whose left (or, equivalently, right) adjoint is also an integral functor.
It will be used in combination with Theorem \ref{thm:adjoint-to-integral-functors}.
\begin{lemma}\label{lem:criterion-for-invertible-kernel}
    Let $X \to T \to \Spec k$ be morphisms of schemes in which
    \begin{itemize}
        \item $k$ is a field,
        \item $T \to \Spec k$ is separated,
        \item $X \to T$ is flat, and
        \item $X \to \Spec k$ is quasi-projective.
    \end{itemize}
    Let $\cP \in D_{qc}(X \times_T X)$ be an object and suppose that $\Phi_{\cP} \colon D_{qc}(X) \to D_{qc}(X)$ has a left (or right) adjoint that is also an integral functor $\Phi_{\cQ}$ for $\cQ \in D_{qc}(X \times_T X)$.
    Then the following are equivalent:
    \begin{enumerate}
        \item $\cP$ is an element of $\FM_T(X)$.
        \item $\Phi_{\cP} \colon D_{qc}(X) \to D_{qc}(X)$ is an equivalence.
        \item $\Phi_{\cP}$ preserves $D^b(X) \subset D_{qc}(X)$ and $\Phi_{\cP} \colon D^b(X) \to D^b(X)$ is an equivalence.
    \end{enumerate}
\end{lemma}
\begin{proof}
    (1) $\Rightarrow$ (2) is clear.
    (2) $\Rightarrow$ (3) follows from Corollary \ref{cor:FM-to-Auteq}.

    We show (3) $\Rightarrow$ (1).
    Note that both $X \to T$ and $T \to \Spec k$ are flat and separated.
    We only discuss the case when $\Phi_{\cQ}$ is the left adjoint to $\Phi_{\cP}$ (since the right adjoint case is similar).
    Let $\iota \colon X \times_T X \hookrightarrow X \times_k X$ be the natural inclusion, $\Delta_T \colon X \to X \times_T X$ and $\Delta_k \colon X \to X \times_k X$ be the diagonal morphisms.
    By the assumption and the proof of Proposition \ref{prpo:forget-base} we have
    \begin{equation}
        \Phi_{\iota_*(\cP * \cQ)} \cong \Phi_{\iota_* \cP} \circ \Phi_{\iota_* \cQ} \cong \id \cong \Phi_{\Delta_{k*}\cO_X} \cong \Phi_{\iota_*\Delta_{T*}\cO_X} \colon D^b(X) \to D^b(X).
    \end{equation}
    Then $\iota_*(\cP * \cQ) \cong \iota_*\Delta_{T*}\cO_X$ holds by Lemma \ref{lem:uniqueness-of-FM-kernel}.
    Since $\iota$ is a closed immersion  and $\Delta_{T*}\cO_X$ is a sheaf, we have $\cP * \cQ \cong \Delta_{T*}\cO_X$.
    By symmetry $\cQ * \cP \cong \Delta_{T*}\cO_X$ also holds.
    Then we have $\cP \in \FM_T(X)$.
\end{proof}

Finally, we close this subsection by giving the following useful lemma.
It is well-known for smooth projective cases (e.g.~{\cite[(3.3)]{MR3713877}}), and also known for projective but not necessarily smooth cases {\cite[Lemma 2.14]{MR2155085}}.
We give a statement for quasi-projective and not necessarily smooth cases.

\begin{lemma}\label{lem:criterion-to-be-standard-functor}
    Let $k$ be an algebraically closed field with characteristic zero.
    Let $X$ and $Y$ be connected quasi-projective schemes over $k$, with $X$ reduced.
    Let $\cP \in D^b(X \times_k Y)$ be an object whose support is proper over $X$.
    Suppose that the integral functor $\Phi = \Phi_{\cP}$ satisfies the following condition:

    \begin{quote}
        For any closed point $x \in X$, there exists a closed point $y \in Y$ and an integer $n_x$ such that $\Phi(\cO_x) \cong \cO_y[n_x]$.
    \end{quote}

    Then there exists a morphism $f \colon X \to Y$, a line bundle $\cL \in \Pic X$, and an integer $n$ such that $\Phi \cong f_*(\cL \otimes -)[n]$.
\end{lemma}
\begin{proof}
    See Appendix \ref{subsection:proof-of-criterion-to-be-standard-functor}.
\end{proof}

\subsection{Spherical objects and twist functors}
In \cite{MR1831820}, Seidel and Thomas introduced the notion of spherical objects of derived categories and associated autoequivalences called twist functors.
This construction is an analog or a counterpart of Dehn twists along Lagrangian spheres in symplectic geometry, through homological mirror symmetry.

Let $X$ be a quasi-projective Gorenstein scheme over a field $k$ and $f \colon X \to \Spec k$ be the structure morphism.
Let $\Delta_k \colon X \to X \times_k X$ be the diagonal morphism.
We assume $X$ to be connected (or more generally equidimensional) with $\dim X = d$.
Under these assumptions the dualizing complex $f^!\cO_{\Spec k}$ is of the form $\omega_X[-d]$ for some line bundle $\omega_X$ on $X$, which we call the \emph{canonical bundle} of $X$.
\begin{definition}[{\cite{MR1831820}}]
    We say $\cE \in D^b(X)$ is a \emph{($d$-)spherical object} if
    \begin{enumerate}
        \item $\cE$ is perfect and has proper support,
        \item $\cE \otimes \omega_X \cong \cE$, and
        \item $\Ext^*_X(\cE, \cE) \cong k \oplus k[-d]$.
    \end{enumerate}
\end{definition}

\begin{theorem}[{\cite{MR1831820}}]
    Let $\cE \in D^b(X)$ be a spherical object.
    Let $T_{\cE} = \Phi_{\cP_{\cE}}$ be the integral functor with the kernel
    \begin{equation}
        \cP_{\cE} = \Cone(\cE^\vee \boxtimes \cE \xrightarrow{\ev} \Delta_{k*}\cO_X),
    \end{equation}
    where the natural ``evaluation on diagonal'' morphism $\ev$ is the composition of
    \begin{enumerate}
        \item the restriction $\cE^\vee \boxtimes \cE \to \Delta_{k*}\Delta_k^*(\cE^\vee \boxtimes \cE) \cong \Delta_{k*}(\cE^\vee \otimes \cE)$ and
        \item $\Delta_{k*}(\cE^\vee \otimes \cE) \to \Delta_{k*}\cO_X$ induced by the natural evaluation $\cE^\vee \otimes \cE \to \cO_X$.
    \end{enumerate}
    Then $T_{\cE}$ is an autoequivalence of $D^b(X)$.
\end{theorem}
The autoequivalence $T_{\cE}$ is called the \emph{twist functor (or spherical twist) along $\cE$}.
We remark that there exists an exact triangle
\begin{equation}\label{eq:exact-triangle-of-spherical-twitst}
    \Ext^*_X(\cE, \cF)\otimes_k \cE \to \cF \to T_{\cE}(\cF) \xrightarrow{+1}
\end{equation}
for every $\cF \in D^b(X)$.
As an immediate consequence of this triangle, we have the following:
\begin{corollary}\label{cor:orthogonal-to-spherical}
    Let $\cF \in \langle \cE \rangle^\perp = \{\cF \mid \Ext^*_X(\cE, \cF) = 0\} \subset D^b(X)$ be an object that is right orthogonal to $\cE$ (e.g. $\Supp \cF \cap \Supp \cE = \emptyset$).
    Then $T_{\cE}(\cF) \cong \cF$.
\end{corollary}
\begin{example}[{\cite{MR1831820}}]\label{ex:spherical-objects}
    \begin{enumerate}
        \item Let $C$ be a (Gorenstein) curve. Then for any smooth point $x \in C$, its structure sheaf $\cO_x$ is spherical. There is an isomorphism $T_{\cO_x} \cong (-) \otimes \cO(x)$.
        \item Let $X$ be a strict Calabi--Yau variety, i.e.~a smooth projective variety whose canonical bundle is trivial and satisfying $H^i(X, \cO)=0$ for $0 < i < \dim X$. Then any line bundle on $X$ is spherical.
        \item Let $S$ be a smooth surface and $G \subset S$ be a $(-2)$-curve. Then the sheaf $\cO_G(a)$, the push-forward of the line bundle on $G \cong \bP^1$ with degree $a \in \bZ$, is spherical.
        \item Let $X$ be a smooth threefold and $C \subset X$ be a $(-1, -1)$-curve, i.e.~a smooth rational curve with normal bundle $\cN_{C/X} \cong \cO_{\bP^1}(-1)\oplus\cO_{\bP^1}(-1)$. Then the structure sheaf $\cO_C$ is spherical in $D^b(X)$.
    \end{enumerate}
\end{example}
\begin{proposition}[{\cite[Proposition 3.15]{MR1831820}}]\label{prop:exceptional-to-spherical}
    Let $X$ be a smooth quasi-projective variety and $\iota \colon Y \hookrightarrow X$ be a projective hypersurface with $\iota^*\omega_X$ is trivial.
    If $\cE \in \Perf(Y)$ is an \emph{exceptional object} (i.e.~$\Ext^*_X(\cE, \cE) \cong k$), then $\iota_* \cE \in D^b(X)$ is a spherical object.
\end{proposition}
\begin{example}\label{ex:spherical-object-from-K3-degeneration}
    Let $\pi \colon X \to C$ be a type $\rom{3}$ degeneration of K3 surfaces \cite{Kulikov1977,Persson--Pinkham1981} (over $\bC$).
    It has the following properties by definition.
    \begin{enumerate}
        \item The canonical bundle $\omega_X$ is trivial.
        \item The central fiber $\iota \colon X_0 = \pi^{-1}(0) \hookrightarrow X$ is a reducible surface $X_0 = \bigcup_i V_i$, and each component $V_i$ is a complete rational surface.
    \end{enumerate}
    Additionally, assume that $X$ and $C$ are quasi-projective and each $V_i$ is smooth.
    Then the structure sheaf $\cO_{V_i}$ of $V_i$ is an exceptional object in $\Perf(V_i)$, since the Hodge numbers $h^{0,1}$ and $h^{0,2}$ of a smooth rational surface are zero.
    By the property (1) and Proposition \ref{prop:exceptional-to-spherical}, the structure sheaves $\cO_{V_i}$ are spherical objects in $D^b(X)$.
    The same argument applies to arbitrary line bundles (or exceptional objects) on $V_i$.

\end{example}
The following basic properties of twist functors will be useful.
\begin{lemma}
    Let $\cE, \cE'$ be spherical objects in $D^b(X)$.
    \begin{enumerate}
        \item For any autoequivalence $\Phi \in \Auteq{D^b(X)}$, we have \begin{equation}
                  \Phi \circ T_{\cE} \circ \Phi^{-1} \cong T_{\Phi(\cE)}.
              \end{equation}
        \item If $\Ext^*_X(\cE, \cE') = 0$, then $T_{\cE}T_{\cE'} \cong T_{\cE'}T_{\cE}$.
        \item If $\sum_i \dim \Ext_X^i(\cE, \cE') = 1$, then $T_{\cE}T_{\cE'} T_{\cE} \cong T_{\cE'}T_{\cE}T_{\cE'}$ (braid relation).
    \end{enumerate}
\end{lemma}
\begin{proof}
    The first statement follows by definition.
    The others are part of \cite[Theorem 1.2]{MR1831820}.
\end{proof}

\section{Half-spherical twists}\label{section:half-spherical-twists}
In this section, we introduce the notion of half-spherical twists and give some examples.
Throughout the section, we work over an algebraically closed field $k$.
Let $X$ and $T$ be smooth quasi-projective varieties and $\pi \colon X \to T$ be a flat morphism.
Let $0 \in T$ be a closed point and $i \colon X_0 = \pi^{-1}(0) \hookrightarrow X$ be the fiber over $0$.
Our situation is summarized into a diagram
\[
    \begin{tikzcd}
        X_0 \arrow[r,"i"]\arrow[d, "\delta"'] & X \arrow[d,"\Delta_T"]\arrow[rd, bend left, "\Delta_k"]&\\
        X_0 \times_k X_0 \arrow[r, "j"] \ar[d]& X \times_T X\arrow[r, "\iota"] \ar[d]&X \times_k X \\
        \{0\} \ar[r] & T &
    \end{tikzcd}
\]
in which $j$ and $\iota$ are the natural inclusions and $\delta$, $\Delta_T$, and $\Delta_k$ are the diagonal morphisms.
Note that the two squares are cartesian and tor-independent by \cite[Lemma 2.25]{MR2238172}.

\begin{definition}[Half-spherical object]
    We call an object $\cE \in D^b(X_0)$ such that $i_*{\cE}$ is spherical in $D^b(X)$ a \emph{half-spherical object} (with respect to $i$)\footnote{We sometimes say ``a half-spherical twist" without mentioning the map $i$ for simplicity. See also Remark \ref{remark:not-presice-notation}.}.
\end{definition}
Suppose we are given a half-spherical object $\cE \in D^b(X_0)$.
Our goal is to show that the twist functor $T_{i_*\cE}$ is a relative integral functor with respect to $T$, and then introduce an autoequivalence $H_{\cE}$ of $D^b(X_0)$ by applying Proposition \ref{prop:restriction-to-fiber}.
\begin{proposition}\label{prop:twist-functor-is-relative-fm}
    Let $\cE \in D^b(X_0)$ be a half-spherical object.
    \begin{enumerate}
        \item  The twist functor $T_{i_*\cE}$ is a relative integral functor with respect to $T$, i.e.~there exists an object $\cR \in D_{qc}(X \times_T X)$ such that $\Phi_{\cR} \cong T_{i_*\cE}$.
        \item The choice of $\cR$ in (1) is unique up to isomorphism and $\cR \in \FM_T(X)$.
    \end{enumerate}
\end{proposition}
\begin{proof}[Sketch of the proof]
    We give an outline of the proof. See Appendix \ref{subsection:compatibilities} and \ref{subsection:proof-of-twist-functor-is-relative-fm} for full details.

    For (1), the strategy is parallel to the first half of the proof of \cite[Proposition 2.7]{MR2200048}.
    Since the kernel $\cP_{\cE}$ of $T_{i_*\cE}$ is the cone of the map
    \begin{equation}
        \ev \colon (i_*\cE)^\vee \boxtimes i_*\cE \to \Delta_{k*}\cO_X,
    \end{equation}
    one simply has to find a morphism $\widetilde{\ev}$ in $D_{qc}(X \times_T X)$ that satisfies $\iota_*(\widetilde{\ev}) = \ev$, up to isomorphism.
    Denote $\CHom_{X_0}(\cE, i^!\cO_X)$ by $\cE'$.
    Consider the morphism
    \begin{equation}
        \widetilde{\ev} \colon j_*(\cE' \boxtimes \cE) \to j_*\delta_*(\cE' \otimes \cE)\cong \Delta_{T*}i_*(\cE' \otimes \cE) \to \Delta_{T*}\cO_X
    \end{equation}
    with
    \begin{itemize}
        \item $\cE' \boxtimes \cE \to \delta_*(\cE' \otimes \cE)$ is the adjoint to the natural isomorphism $\delta^*(\cE' \boxtimes \cE) \cong \cE' \otimes \cE$,
        \item $j_*\delta_* \cong \Delta_{T*}i_*$ is the natural isomorphism, and
        \item $i_*(\cE' \otimes \cE) \to \cO_X$ is the adjoint to the natural pairing $\cE' \otimes \cE \to i^!\cO_X$.
    \end{itemize}
    The domain and codomain of this map satisfy
    \begin{enumerate}
        \item[(a)] $(i_*\cE)^\vee \boxtimes i_*\cE \cong \iota_*j_*(\cE' \boxtimes \cE)$ by Grothendieck duality and Kunneth formula, and
        \item [(b)] $\Delta_{k*}\cO_X \cong \iota_*\Delta_{T*}\cO_X$.
    \end{enumerate}
    Under these identifications, one can show that $\iota_*(\widetilde{\ev}) = \ev$ holds, i.e.~they make the following diagram commutative:
    \begin{equation}\label{eq:evaluation-commutative-diagram}
        \begin{tikzcd}
            (i_*{\cE})^\vee \boxtimes i_*{\cE}\ar[r, "\ev"] \ar[d, "\sim", sloped] & \Delta_{k*}\cO_X \ar[d, "\sim", sloped]\\
            \iota_*j_*(\cE' \boxtimes \cE) \ar[r, "\iota_*(\widetilde{\ev})"]& \iota_*\Delta_{T*}\cO_X.
        \end{tikzcd}
    \end{equation}

    For (2), we first show that the object $\cR$ constructed above is an element of $\FM_T(X)$ by using the criterion in Theorem \ref{thm:adjoint-to-integral-functors} and Lemma \ref{lem:criterion-for-invertible-kernel}.
    Next, if we have another choice of $\cR' \in D_{qc}(X \times_T X)$ such that $\Phi_{\cR'} \cong T_{i_*\cE}$, it is also in $\FM_T(X)$ as $\Phi_{\cR'} \cong \Phi_{\cR}$ satisfies the conditions of Lemma \ref{lem:criterion-for-invertible-kernel}.
    Then any choice of $\cR$ is in $\FM_T(X)$ and thus unique up to isomorphism since the map $\Phi_{(-)} \colon \FM_T(X) \to \Auteq D^b(X)$ is injective by Corollary \ref{cor:FM-to-Auteq} and Proposition \ref{prpo:forget-base}.
\end{proof}
\begin{remark}
    The proof of \cite[Proposition 2.7]{MR2200048} omits detailed discussion on the commutativity of \eqref{eq:evaluation-commutative-diagram}, referring instead to the functoriality of duality.
    We provide a complete proof in our setting in the appendix.
\end{remark}

Once this is established, we can define half-spherical twists.
Let $j^* \colon \FM_T(X) \to \FM_k(X_0)$ be the ``restriction to the fiber'' morphism obtained by Proposition \ref{prop:restriction-to-fiber}.
\begin{definition}[Half-spherical twist]\label{def:half-spherical-twist}
    Let $\cE$ be a half-spherical object, and $\cR_{\cE} \in \FM_T(X)$ be the kernel such that $\Phi_{\cR_\cE} \cong T_{i_*{\cE}}$ as in Proposition \ref{prop:twist-functor-is-relative-fm}.
    We define the \emph{half-spherical twist} $H_{\cE}$ as the ``restriction of $T_{i_*{\cE}}$ to $X_0$'', i.e.~the integral functor with kernel $j^*\cR_{\cE} \in \FM_k(X_0)$:
    \[
        \begin{tikzcd}\label{diagram:definition-of-half-spherical-twist}
            \FM_k(X_0) \ar[d, "\Phi_{(-)}"'] & \FM_T(X)\ar[l, "j^*"'] \ar[d, "\Phi_{(-)}"] & j^*\cR_{\cE} \ar[d, mapsto]& \cR_{\cE} \ar[l, mapsto] \ar[d, mapsto]\\
            \Auteq D^b(X_0) & \Auteq D^b(X) & H_{\cE} & T_{i_*{\cE}}.
        \end{tikzcd}
    \]
\end{definition}
\begin{remark}\label{remark:not-presice-notation}
    Since the autoequivalence $H_{\cE}$ depends not only on the object $\cE$ but also on the data $(X, T, \pi, i \colon X_0 \hookrightarrow X)$, the notation $H_{\cE}$ is not precise. However, we will use it for simplicity.
\end{remark}

\begin{corollary}\label{cor:compatibility-of-half-spherical-twists-and-spherical-twists}
    The functor $H_{\cE}$ is an autoequivalence of $D^b(X_0)$ and makes the following diagram commutative:
    \begin{equation}
        \begin{tikzcd}
            D^b(X_0) \ar[r, "i_*"] \ar[d, "H_{\cE}"'] & D^b(X) \ar[r, "i^*"]\ar[d, "T_{i_*{\cE}}"] & D^b(X_0) \ar[d,  "H_{\cE}"]\\
            D^b(X_0) \ar[r, "i_*"'] & D^b(X) \ar[r, "i^*"'] & D^b(X_0).
        \end{tikzcd}
    \end{equation}
\end{corollary}

\begin{example}\label{ex:half-spherical-twist-from-kodaira-fiber}
    Let $\pi \colon S \to C$ be a relatively minimal elliptic surface and $i_* \colon F \hookrightarrow S$ be a reducible fiber.
    For any irreducible component $\bP^1 \cong G \subset F$ and an integer $a \in \bZ$, the degree $a$ line bundle $\cO_G(a)$ on $G$ (viewed as an object of $D^b(F)$) induces a spherical object $i_*\cO_G(a) \in D^b(S)$.
    Then we have the half-spherical twist $H_{\cO_G(a)} \in \Auteq D^b(F)$.
    This example will be discussed in detail in Section 4.
\end{example}

\begin{example}\label{ex:K3-degeneration}
    Let $\pi \colon X \to C$ be a type $\rom{3}$ degeneration of K3 surfaces, with additional assumptions as in Example \ref{ex:spherical-object-from-K3-degeneration}.
    Then the structure sheaf $\cO_{V_i}$ of a component $V_i$ of the central fiber $X_0$ is a spherical object in $D^b(X)$, and induces the half-spherical twist $H_{\cO_{V_i}} \in \Auteq D^b(X_0)$.
\end{example}

\begin{example}
    Let $\pi \colon X \to C$ be a smooth family of varieties over a smooth curve $C$ and $i \colon X_0 \hookrightarrow X$ be a fiber.
    A $\bP^n$-object $\cE \in D^b(X_0)$ with a suitable condition gives a spherical object $i_*{\cE} \in D^b(X)$ \cite[Proposition 1.4]{MR2200048}.
    Comparing the definition of half-spherical twists and the proof of \cite[Proposition 2.7]{MR2200048}, one can show $H_{\cE} \in \Auteq{D^b(X_0)}$ is isomorphic to the $\bP^n$-twist $P_{\cE}$.
\end{example}

\begin{example}\label{example_from_Atiyah_flop}
    An example with a higher dimensional base space $T$ is obtained as follows.
    Let $\overline{X} = \{xy=zw\} \subset \bA^4 = \Spec \bC[x, y, z, w]$ be the quadric cone and $\overline{f} \colon \overline{X} \to \bA^2 = \Spec \bC[z, w]$ be the natural projection to the $(z, w)$-plane.
    By composing the blow up $\pi \colon X \to \overline{X}$ along the plane $\{y=w=0\} \subset \overline{X}$, we have a morphism $f \colon X \to \bA^2$ between smooth varieties.
    It is flat since all the fibers are one-dimensional.
    The special fiber $X_0=f^{-1}(0) \cong \bA^1 \cup \bP^1 \cup \bA^1$ contains
    the exceptional curve $C \cong \bP^1$ of the blow-up $\pi$.
    The curve $C$ is a $(-1, -1)$-curve in $X$ and therefore its structure sheaf $\cO_C \in D^b(X_0)$ induces a spherical object in $D^b(X)$ (Example \ref{ex:spherical-objects} (4)).
\end{example}

We list some basic properties of half-spherical twists inherited from the ones of twist functors.
\begin{proposition}\label{prop:empty-intersection}
    Let $\cE, \cF \in D^b(X_0)$ be half-spherical objects.
    \begin{enumerate}
        \item For any $\cP \in \FM_T(X)$, we have
              \begin{equation}
                  \Phi_{j^* \cP} \circ H_{\cE} \circ \Phi_{j^* \cP}^{-1} \cong H_{\Phi_{j^* \cP}(\cE)}.
              \end{equation}
        \item If $\Supp \cE \cap \Supp \cF = \emptyset$, then $H_{\cE} \circ H_{\cF} \cong H_{\cF} \circ H_{\cE}$.
        \item For $\cG \in D^b(X_0)$ such that $\Supp \cE \cap \Supp \cG = \emptyset$, we have $H_{\cE} (\cG) \cong \cG$.
    \end{enumerate}
\end{proposition}
\begin{proof}
    For (1) and (2), there are corresponding relations of twist functors in $\Auteq D^b(X)$. They also hold in $\FM_T(X)$ since the map $\FM_T(X) \to \Auteq D^b(X)$ is an injective morphism by Corollary \ref{cor:FM-to-Auteq} and Proposition \ref{prpo:forget-base}. Then the ``restriction'' morphism $j^*$ (followed by $\Phi_{(-)}$) in Definition \ref{def:half-spherical-twist} maps them to the desired relations in $\Auteq D^b(X_0)$.

    For (3), recall the concrete presentation of the integral kernel $j^*\cR_{\cE}$ of $H_{\cE}$:
    \begin{equation}
        j^*\cR_{\cE} = \Cone(j^*j_*(\cE' \boxtimes \cE) \to j^*\Delta_{T*}\cO_X).
    \end{equation}
    In particular, there is an exact triangle
    \begin{equation}
        \Phi_{j^*j_*(\cE' \boxtimes \cE)}(\cG) \to \cG \to H_{\cE}(\cG) \xrightarrow{+1}.
    \end{equation}
    Note that $\Phi_{j^*\Delta_{T*}\cO_X} \cong \Phi_{\delta_*\cO_{X_0}} \cong \id$ holds by tor-independent base change $j^*\Delta_{T*} \cong \delta_* i^*$.
    On the other hand, the object $j^*j_*(\cE' \boxtimes \cE) \in D^b(X_0 \times_k X_0)$ is supported on $\Supp \cE \times_k \Supp \cE \subset X_0 \times_k X_0$.
    This implies $\Phi_{j^*j_*(\cE' \boxtimes \cE)}(\cG) = 0$.
    Therefore $H_{\cE}(\cG) \cong \cG$.
\end{proof}

\begin{example}
    Let $\cE \in D^b(X_0)$ be a half spherical object.
    In the setting of Example \ref{ex:relative-fourier-mukai}, we have the relations
    \begin{equation}
        H_{f_0^*\cE} \cong f_0^* \circ H_{\cE} \circ f_{0*}
    \end{equation}
    and
    \begin{equation}
        H_{\cL\vert_{X_0} \otimes \cE} \cong (\cL\vert_{X_0} \otimes -) \circ H_{\cE} \circ (\cL\vert_{X_0}^{-1} \otimes -).
    \end{equation}
\end{example}

\section{Applications to elliptic surfaces}\label{section:applications-to-elliptic-surfaces}
Throughout this section, we will work over $\bC$.
We denote by $\pi \colon S \to C$ an \emph{elliptic surface}, which means in this paper a projective flat morphism $\pi$ from a smooth quasi-projective surface $S$ to a smooth quasi-projective curve $C$ such that
\begin{enumerate}
    \item the general fiber is a smooth projective curve of genus one, and
    \item no fiber contains a $(-1)$-curve (i.e.~$\pi \colon S \to C$ is \emph{relatively minimal}).
\end{enumerate}

The goal of this section is to study the half-spherical twist $H_{\cO_G(a)}$ in Example \ref{ex:half-spherical-twist-from-kodaira-fiber}.
We will show that $H_{\cO_G(a)}$ corresponds to the \emph{half twist} along an arc on a real surface via homological mirror symmetry.
By using this correspondence, we will also give a new description of the autoequivalence group $\Auteq D^b(S)$ of certain classes of elliptic surfaces.

\subsection{Kodaira fibers and half-spherical twists}
Let $i \colon F \hookrightarrow S$ be a fiber of $\pi \colon S \to C$.
We say $F$ is \emph{multiple} if there exists a divisor $D$ on $S$ and an integer $m \geq 2$ such that $F = mD$ (as divisors on $S$), and \emph{non-multiple} otherwise.

The possible (singular) fibers of an elliptic surface were first classified by Kodaira \cite{MR132556, MR184257} and N\'{e}ron \cite{MR179172}, and it was discovered that they follow an ADE classification.
We call them \emph{Kodaira fibers}.
The following list provides the classification.
It includes Kodaira's notation and, for reducible fibers, the Dynkin types determined by the intersection matrices of their irreducible components.

\begin{enumerate}
    \item[(0)] $\rom{1}_0$, a smooth fiber.
    \item If $F$ is non-multiple and irreducible then it is isomorphic to
          \begin{itemize}
              \item $\rom{1}_1$, a rational curve with one node, or
              \item $\rom{2}$, a rational curve with one cusp.
          \end{itemize}
    \item If $F$ is non-multiple but reducible, then they are classified into two infinite families
          \begin{itemize}
              \item $(\rom{1}_n, \tilde{A}_{n-1})$, a cycle of $n$ projective lines ($n \geq 2$),
              \item $(\rom{1}^*_n, \tilde{D}_{n+4})$, ($n \geq 0$),
          \end{itemize}
          and exceptional ones
          \begin{itemize}
              \item $(\rom{3}$, $\tilde{A}_1)$, two projective lines tangent at one point with multiplicity $2$,
              \item $(\rom{4}$, $\tilde{A}_2)$, three projective lines intersecting at one point,
              \item $(\rom{2}^*$, $\tilde{E}_8)$,
              \item $(\rom{3}^*$, $\tilde{E}_7)$, and
              \item $(\rom{4}^*$, $\tilde{E}_6)$.
          \end{itemize}
    \item If $F$ is multiple, then it is an $m$-multiple of a fiber of type $\rom{1}_n$ ($m \geq 2, n \geq 0$), which is denoted by ${}_m\rom{1}_n$.
\end{enumerate}

For example, the Kodaira fibers of type $\rom{1}_n$ look like Figure \ref{fig:kodaira-fibers}.
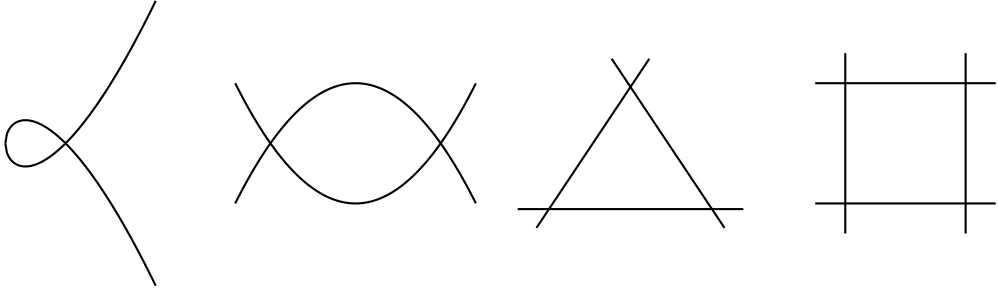
\begin{figure}[h]
    \centering
    \begin{minipage}{.2\textwidth}
        \centering
        \begin{tikzpicture}[samples=300, scale=0.8]
            \draw[thick, domain=-1:1.5, samples=100] plot(\x,{sqrt(\x+1)*(abs(\x))});
            \draw[thick, domain=-1:1.5, samples=100] plot(\x,{-sqrt(\x+1)*(abs(\x))});
        \end{tikzpicture}
    \end{minipage}
    \begin{minipage}{.2\textwidth}
        \centering
        \begin{tikzpicture}[scale=0.8]
            \draw[thick, domain=-2:2, samples=100] plot(\x, {(1/2)*(\x)^2 - 1});
            \draw[thick, domain=-2:2, samples=100] plot(\x, {(-1/2)*(\x)^2 + 1});
        \end{tikzpicture}
    \end{minipage}
    \begin{minipage}{.2\textwidth}
        \centering
        \begin{tikzpicture}[scale=0.25]
            \draw[thick] (-9,-1)--(-3,8);
            \draw[thick] (-10,0)--(2,0);
            \draw[thick] (1,-1)--(-5,8);
        \end{tikzpicture}
    \end{minipage}
    \begin{minipage}{.2\textwidth}
        \centering
        \begin{tikzpicture}[scale=0.4]
            \draw[thick] (-3, -2)--(3, -2);
            \draw[thick] (-3, 2)--(3, 2);
            \draw[thick] (-2, -3)--(-2, 3);
            \draw[thick] (2, -3)--(2, 3);
        \end{tikzpicture}
    \end{minipage}
    \caption{Kodaira fibers of type $\rom{1}_1$, $\rom{1}_2$, $\rom{1}_3$, and $\rom{1}_4$, respectively.}
    \label{fig:kodaira-fibers}
\end{figure}

If $F$ is reducible, then every irreducible component $G$ (with the reduced subscheme structure) of $F$ is a $(-2)$-curve on $S$ \cite[Section V.7]{MR2030225}.
The line bundle $\cO_G(a)$ on $G \cong \bP^1$ with degree $a \in \bZ$, viewed as an object of the category $D^b(F)$, induces a spherical object $i_*\cO_G(a) \in D^b(S)$, and gives the half-spherical twist $H_{\cO_G(a)} \in \Auteq D^b(F)$ (Example \ref{ex:half-spherical-twist-from-kodaira-fiber}).
They satisfy the following relations in general.

\begin{proposition}[braid relations]\label{prop:braid-relation}
    Let $G$ and $G'$ be irreducible components of a reducible fiber $F$.
    \begin{enumerate}
        \item If $G \cdot G' = 1$, then we have
              \begin{equation}
                  H_{\cO_G(a)} \circ H_{\cO_{G'}(a')} \circ H_{\cO_G(a)} \cong H_{\cO_{G'}(a')} \circ H_{\cO_G(a)} \circ H_{\cO_{G'}(a')}
              \end{equation}
              for any $a, a' \in \bZ$.
        \item If $G \cap G' = \emptyset$, then we have
              \begin{equation}
                  H_{\cO_G(a)} \circ H_{\cO_{G'}(a')} \cong H_{\cO_{G'}(a')} \circ H_{\cO_G(a)}
              \end{equation}
              for any $a, a' \in \bZ$.
    \end{enumerate}
\end{proposition}
\begin{proof}
    There are corresponding braid relations for the twist functors $T_{\cO_G(a)}$ and $T_{\cO_{G'}(a')}$ in $\Auteq D^b(S)$ \cite[Example 3.5]{MR1831820}.
    Then the statement follows from the same argument as in the proof of Proposition \ref{prop:empty-intersection}.
\end{proof}

\begin{proposition}\label{prop:half-spherical-twist-and-line-bundle}
    Let $G$ be an irreducible component of a reducible fiber $F$.
    Then we have
    $H_{\cO_G(a-1)} \circ H_{\cO_G(a)} \cong (-) \otimes \cO_{S}(G)\lvert_{F}$ for every $a \in \bZ$.
\end{proposition}
\begin{proof}
    For a smooth surface $S$, a $(-2)$-curve $G \subset S$, and an integer $a \in \bZ$, we have a relation $T_{\cO_G(a-1)} \circ T_{\cO_G(a)} \cong (-) \otimes \cO_{S}(G)$ in $\Auteq D^b(S)$ by \cite[Lemma 4.15 (i)(2)]{MR2198807}.
    Then the proposition follows as in the proof of the previous proposition.
\end{proof}

\subsection{Homological mirror symmetry}
We will study the half-spherical twists $H_{\cO_G(a)}$ for type $\rom{1}_n$ Kodaira fibers in the context of homological mirror symmetry.
In this subsection, we denote by $F_n$ a Kodaira fiber of type $\rom{1}_n$, and \textbf{we always assume} $\mathbf{n \geq 2}$\footnote{If $n = 1$, the component $G = F$ is not a $(-2)$ curve, and does not give half-spherical twists.}.
To begin with, we prepare some computations on $\Ext$ groups.
In Lemmas~\ref{lem:ext-on-surface}--\ref{lem:preparation-2}, we denote by $G$ an irreducible component of $F_n$ and by $H = H_{\cO_G(-1)}$ the half-spherical twist along $\cO_G(-1)$.

\begin{lemma}\label{lem:ext-on-surface}
    \begin{enumerate}
        \item $\Ext^*_S(\cO_G(-1), \cO_{F_n}) = 0$.
        \item $\Ext^*_S(\cO_G(-1), \cO_x) = \begin{cases}
                      \bC \oplus \bC[-1] & x \in G              \\
                      0                  & x \in S \setminus G.
                  \end{cases}$
    \end{enumerate}
\end{lemma}
\begin{proof}
    (1). Since $\cO_G(-1)$ is spherical, we have $\cO_G(-1) \otimes \omega_S = \cO_G(-1)$. Then by Serre duality, it is enough to show $\Ext^*_S(\cO_{F_n}, \cO_G(-1)) = 0$.
    The locally free resolution
    \begin{equation}
        0 \to \cO_S(-F_n) \to \cO_S \to \cO_{F_n} \to 0
    \end{equation}
    computes the $\CExt$ sheaves as
    \begin{align}
        \CExt^0_S(\cO_{F_n}, \cO_G(-1))        & = \cO_G(-1)                                \\
        \CExt^1_S(\cO_{F_n}, \cO_G(-1))        & = \cO_G(-1) \otimes \cO_S(F_n) = \cO_G(-1) \\
        \CExt^{\geq 2}_S(\cO_{F_n}, \cO_G(-1)) & = 0.
    \end{align}
    Note that $\cO_S(F_n)\vert_G \cong \cO_G$ holds.
    In addition, we have $H^*(\cO_G(-1)) = 0$.
    Then the local-to-global spectral sequence shows the statement.

    (2). The case $x \in S \setminus G$ is trivial. Let $x \in G$. Consider the exact sequence
    \begin{equation}\label{eq:locally-free-resolution-of-O_G}
        0 \to \cO_S(-G) \to \cO_S \to \cO_{G} \to 0.
    \end{equation}
    Since the two sheaves $\cO_S(-G)$ and $\cO_{G}$ are isomorphic on a neighborhood of $x$, we have $\CExt^*_S(\cO_{G}(-1), \cO_x) = \CExt^*_S(\cO_{G}, \cO_x)$.
    The latter can be computed by the locally free resolution \eqref{eq:locally-free-resolution-of-O_G} as
    \begin{align}
        \CExt^0_S(\cO_{G}, \cO_x)        & = \cO_x                          \\
        \CExt^1_S(\cO_{G}, \cO_x)        & = \cO_G(S) \otimes \cO_x = \cO_x \\
        \CExt^{\geq 2}_S(\cO_{G}, \cO_x) & = 0.
    \end{align}
    Then the local-to-global spectral sequence applies.
\end{proof}

\begin{lemma}\label{lem:preparation}
    \begin{enumerate}
        \item $H(\cO_{F_n}) \cong \cO_{F_n}$.
        \item $H(\cO_x) \cong \cO_x$ for every point $x \in F \setminus G$.
        \item For every point $x \in G$, there is a non-splitting exact triangle \begin{equation}
                  \cO_G(-2)[1] \to H(\cO_x)\to \cO_G(-1) \xrightarrow{+1}
              \end{equation}
              in $D^b(F_n)$.
    \end{enumerate}
\end{lemma}
\begin{proof}
    By Lemma \ref{lem:ext-on-surface}, we have $\cO_{F_n} \in \langle \cO_G(-1) \rangle^\perp \subset D^b(S)$ and hence $T_{\cO_G(-1)}(\cO_{F_n}) \cong \cO_{F_n}$ (Corollary \ref{cor:orthogonal-to-spherical}).
    Combining $i_* \circ H \cong T_{\cO_G(-1)} \circ i_*$ (Corollary \ref{cor:compatibility-of-half-spherical-twists-and-spherical-twists}), there is an isomorphism $i_*H(\cO_{F_n}) \cong T_{\cO_G(-1)}(\cO_{F_n}) \cong \cO_{F_n}$ in $D^b(S)$.
    Since $i$ is a closed immersion, the complex $H(\cO_{F_n})$ is a coherent sheaf on $F_n$ and therefore $H(\cO_{F_n}) \cong \cO_{F_n}$ in $D^b(F_n)$.
    This proves (1).
    The proof of (2) is similar.

    For (3), first notice that there are isomorphisms of coherent sheaves
    \begin{equation}
        \cH^i(T_{\cO_G(-1)}(\cO_x)) \cong \begin{cases}
            \cO_G(-1) & i = 0,            \\
            \cO_G(-2) & i = -1,           \\
            0         & \text{otherwise},
        \end{cases}
    \end{equation}
    for $x \in G$.
    This follows from the long exact sequence associated with the exact triangle \eqref{eq:exact-triangle-of-spherical-twitst}
    \begin{equation}
        \Ext^*_S(\cO_G(-1), \cO_x) \otimes_\bC \cO_G(-1)\to \cO_x \to T_{\cO_G(-1)}(\cO_x) \xrightarrow{+1}
    \end{equation}
    and Lemma \ref{lem:ext-on-surface}.
    Since $T_{\cO_G(-1)}(\cO_x) = i_*H(\cO_x)$ and the pushforward $i_*$ by the closed immersion $i$ preserves the cohomology sheaves, we have
    \begin{equation}
        \cH^i(H(\cO_x)) \cong \begin{cases}
            \cO_G(-1) & i = 0,            \\
            \cO_G(-2) & i = -1,           \\
            0         & \text{otherwise}.
        \end{cases}
    \end{equation}
    In particular, there is a natural morphism $H(\cO_x) \to \cO_G(-1)$, from the complex $H(\cO_x)$ to its rightmost non-zero cohomology sheaf $\cO_G(-1)$ (\cite[Excersice 2.32]{MR2244106}).
    In addition, its cocone must be $\cO_G(-2)[1]$ by the long exact sequence of cohomology sheaves.
    Then there is an exact triangle
    \begin{equation}
        \cO_G(-2)[1] \to H(\cO_x)\to \cO_G(-1) \xrightarrow{+1},
    \end{equation}
    which is non-splitting since $\cO_G(-2)[1]$, $H(\cO_x)$, and $\cO_G(-1)$ are indecomposable.
\end{proof}

\begin{remark}
    \begin{enumerate}
        \item The object $H(\cO_x)$ is perfect if and only if $x$ is a smooth point of $F_n$.
        \item We have $\Ext^2_{F_n}(\cO_G(-2), \cO_G(-1)) = \bC^2$ as follows: First, $\CHom_{F_n}(\cO_G(-2), \cO_G(-1)) = \cO_G(1)$. Second, by local computations, $\CExt^i_{F_n}(\cO_G(-2), \cO_G(-1)) = \cO_s \oplus \cO_t$ for even $i \geq 2$ and $=0$ for odd $i \geq 1$, where $s$ and $t$ are the two singular points of $F_n$ on $G$. Finally, the local-to-global spectral sequence computes the $\Ext$ group.
        \item Since the automorphisms of $\cO_G(-2)[1]$ and $\cO_G(-1)$ are only scalars, the set of objects which is obtained as a non-splitting extension of $\cO_G(-1)$ by $\cO_G(-2)[1]$ is given by $(\Ext^2_{F_n}(\cO_G(-2), \cO_G(-1)) \setminus \{0\}) / \bC^*$.
              Together with (2), the natural map
              \begin{equation}
                  G \to (\Ext^2_{F_n}(\cO_G(-2), \cO_G(-1)) \setminus \{0\}) / \bC^* = \bP^1, \quad x \mapsto H(\cO_x)
              \end{equation}
              turns out to be a bijection.
    \end{enumerate}
\end{remark}

\begin{lemma}\label{lem:resolution-of-H-Ox}
    Let $\overline{G}$ be the union of all the components other than $G$ of $F_n$, and $\cL$ be a line bundle on $F_n$ such that $\cL|_G \cong \cO_G(-1)$ and $\cL|_{\overline{G}} \cong \cO_{\overline{G}}$\footnote{We can always find such a line bundle because there is an exact sequence \begin{equation}0 \to \bC^* \to \Pic(F_n) \to \Pic(G)\oplus \Pic(\overline{G}) \to 0 \end{equation}}.
    Then for any smooth point $x \in G$ of $F_n$ there is an exact triangle
    \begin{equation}
        \cM^x \to \cL \to H(\cO_x) \xrightarrow{+1}
    \end{equation}
    where $\cM^x$ is a line bundle on $F_n$ such that $\cM^x|_G \cong \cO_{G}$ and $\cM^x|_{\overline{G}} \cong \cO_{\overline{G}}(-Z)$ with $Z = G \cap \overline{G}$.
\end{lemma}
\begin{proof}
    Since we have $\Ext^*(\cL, \cO_G(-2)) = 0$ and $\Ext^*(\cL, \cO_G(-1)) = \bC$, the exact triangle in Lemma \ref{lem:preparation} (3) implies $\Ext^*(\cL, H(\cO_x)) = \bC$.
    Let $\cL \to H(\cO_x)$ be a non-zero morphism, which is unique up to scaling.
    By the long exact sequence of cohomology sheaves, the cocone $\cM^x$ of this morphism is a coherent sheaf on $F_n$ which fits into an exact sequence
    \begin{equation}\label{eq:exact-sequence-of-cM}
        0 \to \cO_G(-2) \to \cM^x \to \cO_{\overline{G}}(-Z) \to 0.
    \end{equation}
    In addition, $\cM^x$ is a perfect complex as $\cL$ and $H(\cO_x)$ are perfect, and hence $\cM^x$ is a line bundle by \eqref{eq:exact-sequence-of-cM}.
    This shows the statement.
\end{proof}

In the following lemma, we say that two different irreducible components $G, G' \subset F_n$ are \emph{adjacent} if they intersect.
When $n = 2$, the two components are adjacent.
When $n \geq 3$, each component has two adjacent components.

\begin{lemma}\label{lem:preparation-2}
    Let $x \in G$ be a point which is a smooth point of $F_n$.
    Then we have the following.
    \begin{enumerate}
        \item $\Ext^*_{F_n}(H(\cO_{x}), \cO_{y}) =
                  \begin{cases}
                      \bC \oplus \bC[-1] & y \in G \text{: a smooth point of }F_n, \\
                      0                  & y \in F_n \setminus G.
                  \end{cases}$
        \item $\Ext^*_{F_n}(H(\cO_{x}), \cO_{F_n}) = \bC[-1]$.
        \item Let $G' \subset F_n$ be an irreducible component.
              \begin{enumerate}
                  \item[(a)] When $n = 2$, we have \begin{equation}
                            \Ext^*_{F_n}(H(\cO_{x}), \cO_{G'}(-1)) = \begin{cases}
                                \bC       & G=G',                             \\
                                \bC^2[-1] & \text{$G$ and $G'$ are adjacent}.
                            \end{cases}
                        \end{equation}
                  \item[(b)] When $n \geq 3$, we have \begin{equation}
                            \Ext^*_{F_n}(H(\cO_{x}), \cO_{G'}(-1)) = \begin{cases}
                                \bC     & G=G',                                         \\
                                \bC[-1] & \text{$G$ and $G'$ are adjacent},             \\
                                0       & \text{otherwise, i.e.~}G \cap G' = \emptyset.
                            \end{cases}
                        \end{equation}
              \end{enumerate}
        \item $\Ext^*_{F_n}(H(\cO_{x}), \cO_{G}) = \bC[-1] \oplus \bC^2$.
    \end{enumerate}
\end{lemma}
\begin{proof}
    We use the notations in Lemma \ref{lem:resolution-of-H-Ox}. For any coherent sheaf $\cF$ on $F_n$, we have a long exact sequence
    \begin{equation}\label{eq:long-exact-sequence-for-ext}
        0 \to \Ext^0(H(\cO_x), \cF) \to \Ext^0(\cL, \cF) \to \Ext^0(\cM^x, \cF) \to \Ext^1(H(\cO_x), \cF) \to 0
    \end{equation}
    and $\Ext^i(H(\cO_x), \cF) = 0$ for $i \neq 0, 1$ by Lemma \ref{lem:resolution-of-H-Ox}.

    (1). The case $y \in F_n \setminus G$ is clear.
    For a smooth point $y \in G$ of $F_n$, we have $\Ext^0(\cL, \cO_y) = \bC$ and  $\Ext^0(\cM^x, \cO_y) = \bC$.
    On the other hand, the map $\cM^x \to \cL$ is zero when restricted to $G$ since $\cM^x|_G \cong \cO_G$ and $\cL|_G \cong \cO_G(-1)$, and hence $\Ext^0(\cL, \cO_y) \to \Ext^0(\cM^x, \cO_y)$ is zero.
    Then \eqref{eq:long-exact-sequence-for-ext} looks like
    \begin{equation}
        0 \to \Ext^0(H(\cO_x), \cO_{G'}(-1)) \to \bC \xrightarrow{0} \bC \to \Ext^1(H(\cO_x), \cO_{G'}(-1)) \to 0
    \end{equation}
    so that $\Ext^0(H(\cO_x), \cO_{G'}(-1)) = \bC$ and $\Ext^1(H(\cO_x), \cO_{G'}(-1)) = 0$.

    (2). We have $\Ext^*(\cO_{F_n}, H(\cO_x)) = \bC$ by $\Ext^*(\cO_{F_n}, \cO_G(-2)) = \bC$, $\Ext^*(\cO_{F_n}, \cO_G(-1)) = 0$, and the exact triangle
    \begin{equation}
        \cO_G(-2)[1] \to H(\cO_x) \to \cO_G(-1) \xrightarrow{+1}.
    \end{equation}
    Then the statement follows from Serre duality.

    (3). The case $G \cap G' = \emptyset$ case is clear. When $G = G'$, \eqref{eq:long-exact-sequence-for-ext} looks like
    \begin{equation}
        0 \to \Ext^0(H(\cO_x), \cO_G(-1)) \to \bC \to 0 \to \Ext^1(H(\cO_x), \cO_G(-1)) \to 0.
    \end{equation}
    When $n=2$ and $G \neq G'$, \eqref{eq:long-exact-sequence-for-ext} looks like
    \begin{equation}
        0 \to \Ext^0(H(\cO_x), \cO_{G'}(-1)) \to 0 \to \bC^2 \to \Ext^1(H(\cO_x), \cO_{G'}(-1)) \to 0.
    \end{equation}
    When $n \geq 3$ and $G$ and $G'$ are adjacent, \eqref{eq:long-exact-sequence-for-ext} looks like
    \begin{equation}
        0 \to \Ext^0(H(\cO_x), \cO_{G'}(-1)) \to 0 \to \bC \to \Ext^1(H(\cO_x), \cO_{G'}(-1)) \to 0.
    \end{equation}
    Then the statement follows.

    (4). Similar to (1), the map $\Ext^0(\cL, \cO_G) \to \Ext^0(\cM^x, \cO_G)$ is zero and hence \eqref{eq:long-exact-sequence-for-ext} looks like
    \begin{equation}
        0 \to \Ext^0(H(\cO_x), \cO_G) \to \bC^2 \xrightarrow{0} \bC \to \Ext^1(H(\cO_x), \cO_G) \to 0.
    \end{equation}
    This shows (4).
\end{proof}

We now introduce several results from the perspective of homological mirror symmetry.
Let $T_n$ denote the $2$-torus with $n$-punctures.
According to a result of Lekili and Polishchuk, the curve $F_n$ and the real surface $T_n$ form a mirror pair.
\begin{theorem}[{\cite[Theorem B]{MR3663596}}]\label{thm:mirror-symmetry-for-F_n}
    There is a $\bC$-linear triangulated equivalence
    \begin{equation}
        D^b(F_n) \cong D^b(\cW(T_n))
    \end{equation}
    between the derived category of coherent sheaves of $F_n$ and the derived category of the wrapped Fukaya category $\cW(T_n)$ of $T_n$.
\end{theorem}
\begin{remark}
    \begin{enumerate}
        \item The corresponding equivalence for type $\rom{1}_0$ (i.e.~elliptic curves) and $\rom{1}_1$ Kodaira fibers are known by \cite{MR1633036} and \cite{lekili2012arithmeticmirrorsymmetry2torus}, respectively.
        \item For the other types of Kodaira fibers, their mirror partners (or even candidates of mirrors) are not known yet.
    \end{enumerate}

\end{remark}
This equivalence suggests that spherical objects and some autoequivalences of $D^b(F_n)$ correspond to closed curves and mapping classes of $T_n$, respectively.
For more detailed formulations, we refer to \cite{opperJEMS}.
\begin{theorem}[{\cite[Theorem A, Proposition 8.13]{opperJEMS}}]\label{thm:bijection-between-objects-and-curves}
    There exists a bijection between isomorphism classes of indecomposable objects of $D^b(F_n)$ up to shift, and homotopy classes of curves on $T_n$ equipped with indecomposable local systems.
    Moreover, under this bijection, spherical objects correspond to non-separating simple loops with one-dimensional local systems.
\end{theorem}
\begin{theorem}[{\cite{opperJEMS}}]\label{thm:intersections-are-morphisms}
    Let $\cF$ and $\cG$ be indecomposable objects of $D^b(F_n)$ which correspond to $(\gamma_{\cF}, V_{\cF})$ and $(\gamma_{\cG}, V_{\cG})$ by Theorem \ref{thm:bijection-between-objects-and-curves}, respectively.
    Here $\gamma_{\cF}, \gamma_{\cG}$ are homotopy classes of curves on $T_n$ and $V_{\cF}, V_{\cG}$ are local systems on these curves.
    Suppose that either $\cF$ or $\cG$ is perfect, $\cF$ is not a shift of $\cG$, and $\dim V_{\cF} = \dim V_{\cG} = 1$. Then $\gamma_{\cF}$ and $\gamma_{\cG}$ have precisely $\sum_{i}\dim\Ext^i(\cF, \cG)$ intersections.
\end{theorem}
\begin{proof}
    The following triangulated categories and exact functors are introduced in \cite{opperJEMS}:
    \begin{enumerate}
        \item The derived category $D^b(\Lambda_n)$ of an algebra $\Lambda_n$ with an autoequivalence $\tau = \vartheta^{-1}$ \cite[Section 2]{opperJEMS}.
        \item A full subcategory $D_{\partial} \subset D^b(\Lambda_n)$ \cite[Section 8]{opperJEMS}
        \item A fully faithful functor $\Perf(F_n) \xrightarrow{F} D^b(\Lambda_n)$ whose essential image consists of $\tau$-invariant objects \cite[Corollaries 3.5 and 6.3]{MR2854109} (see also \cite[Theorem 2.2]{opperJEMS}).
        \item The Verdier quotient $\pi \colon D^b(\Lambda_n) \to D^b(\Lambda_n)/D_{\partial}$ \cite[Section 8]{opperJEMS}.
        \item An equivalence $G \colon D^b(\Lambda_n)/D_{\partial} \xrightarrow{\sim} D^b(F_n)$ \cite[Proposition 8.6]{opperJEMS}.
    \end{enumerate}
    In addition, the composition $G \circ \pi \circ F$ is isomorphic to the identity functor when restricted to $\Perf(F_n)$ \cite[Remark 8.7]{opperJEMS}.

    Let $\cF, \cG \in D^b(F_n)$ be objects satisfying the conditions in the statement.
    By Serre duality, we may assume $\cG$ is perfect without loss of generality.
    Let $\cF' \in D^b(\Lambda_n)$ be an object such that $G(\pi(\cF')) \cong \cF$.
    We have a canonical map
    \begin{equation}\label{eq:verdier-quotient-morphism}
        \Ext^*_{D^b(\Lambda_n)}(\cF', F(\cG)) \to \Ext^*_{D^b(F_n)}(\cF, \cG)
    \end{equation}
    by $G \circ \pi$.
    Since $\cG$ is perfect, $F(\cG)$ is a $\vartheta$-invariant object by (3) and $\tau = \vartheta^{-1}$.
    Then \cite[Proposition 8.21]{opperJEMS} says the map \eqref{eq:verdier-quotient-morphism} is an isomorphism, and \cite[Proposition 4.10]{opperJEMS} can be applied to conclude the statement.
\end{proof}
\begin{theorem}[{\cite[Theorem D]{opperJEMS}}]\label{thm:autoequivalence-of-I_n-curve}
    The group of autoequivalences $\Auteq{D^b(F_n)}$ fits into the short exact sequence
    \begin{equation}
        1 \to \Aut^0(F_n) \times \Pic^0(F_n) \times \bZ[1] \to \Auteq{D^b(F_n)} \xrightarrow{\Upsilon} \MCG(T_n) \to 1,
    \end{equation}
    where $\Pic^0(F_n)$ denotes the group of line bundles with multi-degree zero, and $\Aut^0(F_n) \cong (\bC^\times)^n$ denotes the group of autoequivalences fixing the irreducible components of $F_n$.
\end{theorem}
\begin{theorem}[{\cite[Corollary 8.37]{opperJEMS}}]\label{thm:definition-of-upsilon}
    The morphism $\Upsilon$ in Theorem \ref{thm:autoequivalence-of-I_n-curve} respects the correspondence in Theorem \ref{thm:bijection-between-objects-and-curves} in the following sense:
    For an element $\Phi \in \Auteq{D^b(F_n)}$, the mapping class $\Upsilon(\Phi)$ acts on the set of homotopy classes of curves on $T_n$ by $\gamma_{\cF} \mapsto \gamma_{\Phi(\cF)}$, where $\gamma_{\cF}$ is the curve on $T_n$ corresponding to an indecomposable object $\cF \in D^b(F_n)$.
\end{theorem}
\begin{theorem}[{\cite{opperJEMS}}]\label{thm:spherical-twist-and-dehn-twist}
    Let $\cE \in D^b(F_n)$ be a spherical object and $(\gamma_{\cE}, V_{\cE})$ be the corresponding loop and local system. Then the morphism $\Upsilon$ in Theorem \ref{thm:autoequivalence-of-I_n-curve} maps the twist functor $T_{\cE}$ to the Dehn twist $T_{\gamma_{\cE}}$ along $\gamma_{\cE}$.
\end{theorem}
\begin{proof}
    We use the notations in the proof of Theorem \ref{thm:intersections-are-morphisms}.
    Since any autoequivalence of $D^b(\Lambda_n)$ restricts to $D_{\partial}$ \cite[Lemma 8.5]{opperJEMS}, the localization $G \circ \pi$ induces a natural morphism $\Auteq{D^b(\Lambda_n)} \to \Auteq{D^b(F_n)}$.
    On the other hand, let $\bT_n$ denotes a torus with $n$ boundary components \cite[Section 4]{opperJEMS}.
    There is a morphism $\Psi \colon \Auteq{D^b(\Lambda_n)} \to \MCG(\bT_n)$ \cite[Section 6]{opperJEMS}, and the proof of \cite[Proposition 8.36]{opperJEMS} shows that there is a commutative diagram
    \begin{equation}
        \begin{tikzcd}
            \Auteq{D^b(\Lambda_n)} \ar[r, "\Psi"] \ar[d] & \MCG(\bT_n) \ar[d]\\
            \Auteq{D^b(F_n)} \ar[r, "\Upsilon"'] & \MCG(T_n)
        \end{tikzcd}
    \end{equation}
    where the morphism $\MCG(\bT_n) \to \MCG(T_n)$ is the ``radial extension'' \cite[(6.3)]{opperJEMS}, which preserves Dehn twists.

    Now let $\cE \in D^b(F_n)$ be a spherical object.
    By \cite[Lemma 5.1]{opperJEMS}, $F(\cE)$ is a spherical object in $D^b(\Lambda_n)$.
    The twist functor $T_{F(\cE)}$ is then mapped to the Dehn twist $T_{\gamma_{G(\pi (F(\cE)))}} \in \MCG(\bT_n)$ by \cite[Theorem 6.10, Proposition 8.13]{opperJEMS}.
    Since $G(\pi (F(\cE)))$ is isomorphic to $\cE$ and the radial extension preserves Dehn twists, the statement follows.
\end{proof}

From now on, we denote the curve corresponding to an indecomposable object $\cE \in D^b(F_n)$ by $\gamma_{\cE}$.

\begin{example}\label{ex:corresponding-curves-via-hms}
    Figure \ref{fig:corresponding-curves-via-hms} shows some examples of Theorem \ref{thm:bijection-between-objects-and-curves}.
    The big rectangle illustrates the $n$-punctured torus $T_n$.
    The top and bottom edges (resp.~the left and right edges) are identified, and the white circles represent the punctures.

    Take an irreducible decomposition $F_n = G_1 \cup \cdots \cup G_n$ with cyclic order and pick a smooth point $x_i$ from each component $G_i$.
    The structure sheaves $\cO_{F_n}, \cO_{x_1}, \dots, \cO_{x_n}$ of $F_n$ and the points $x_i$ are spherical objects of $D^b(F_n)$.
    They correspond to non-separating simple loops $\gamma_{\cO_{F_n}}, \gamma_{\cO_{x_1}}, \dots, \gamma_{\cO_{x_n}}$ on $T_n$ which are shown in Figure \ref{fig:corresponding-curves-via-hms}.
    Additionally, every sheaf of the form $\cO_{G_i}(a)$ is an indecomposable object that is not perfect (and hence not spherical).
    As a result, the curves $\gamma_{\cO_{G_i}(-1)}$ are not loops but arcs connecting two punctures.
\end{example}

\begin{figure}[h]
    \centering
    \begin{displaymath}
        \begin{tikzpicture}[scale=1.2]
            \draw[dashed] (0,0)--(11,0);
            \draw[dashed] (0,0)--(0,4);
            \draw[dashed] (0,4)--(11,4);
            \draw[dashed] (11,4)--(11,0);

            \draw[thick] (0, 2)--(1, 2);
            \draw[thick] (1, 2)--(3, 2);
            \draw[thick] (3, 2)--(5, 2);
            \draw[thick] (9, 2)--(11, 2);

            \draw[thick] (0, 0.5)--(11, 0.5);

            \draw[thick] (2, 0)--(2, 4);
            \draw[thick] (4, 0)--(4, 4);
            \draw[thick] (10, 0)--(10, 4);

            \draw[dotted, thick] (5+1, 2)--(9-1, 2);
            \foreach \u in {1, 3, 5, 9}
                {
                    \filldraw[white] (\u, 2) circle (2pt);
                    \draw[black] (\u, 2) circle (2pt);
                }

            \draw(0, 0.5) node[left]{$\gamma_{\cO_{F_n}}$};
            \draw(2, 4) node[above]{$\gamma_{\cO_{x_1}}$};
            \draw(4, 4) node[above]{$\gamma_{\cO_{x_2}}$};
            \draw(10, 4) node[above]{$\gamma_{\cO_{x_n}}$};

            \draw(2, 3) node[above left]{$\gamma_{\cO_{G_1}(-1)}$};
            \draw(1, 3) to[out=-90,in=135](1.5, 2);
            \draw(4, 3) node[above left]{$\gamma_{\cO_{G_2}(-1)}$};
            \draw(3, 3) to[out=-90,in=135](3.5, 2);

            \draw(11, 2) node[right]{$\gamma_{\cO_{G_n}(-1)}$};

        \end{tikzpicture}
    \end{displaymath}
    \caption{The curves $\gamma_{\cO_{F_n}}, \gamma_{\cO_{x_i}}$ and $\gamma_{\cO_{G_i}(-1)}$ on the $n$-punctured torus $T_n$.}
    \label{fig:corresponding-curves-via-hms}
\end{figure}
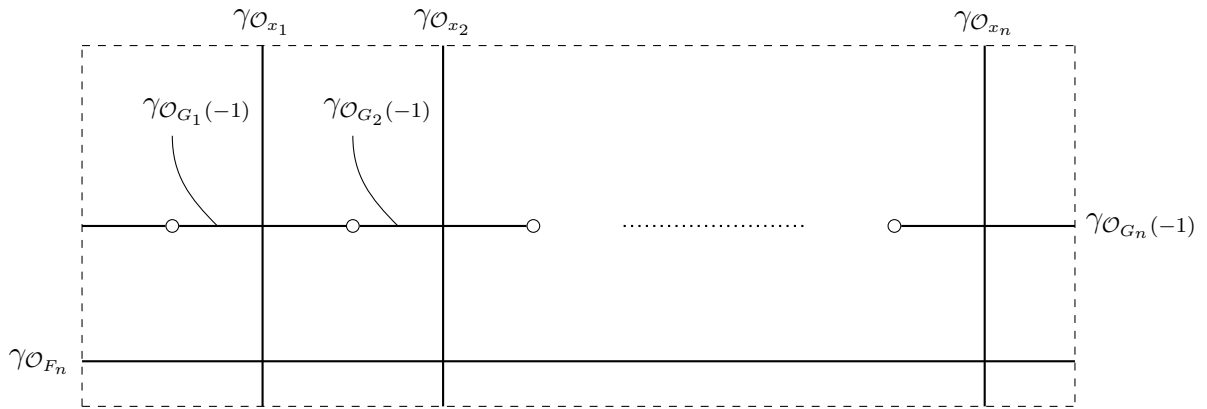

\begin{example}
    For an irreducible component $G \subset F_n$ and $a \in \bZ$, the curve $\gamma_{\cO_G(a)}$ corresponding to the sheaf $\cO_G(a) \in D^b(F_n)$ is shown in Figure \ref{geometric_picture_of_induced_autoequivalences_1}.
    It is an arc wrapping around the torus $|a + 1|$ times.
    When $a + 1 <0$ the curve $\gamma_{\cO_G(a)}$ wraps around the torus in the opposite direction.
    When $a = -1$ we get the curve $\gamma_{\cO_G(-1)}$ which is already shown in Figure \ref{fig:corresponding-curves-via-hms}.
\end{example}

\begin{figure}[h]
    \centering
    \begin{displaymath}
        \begin{tikzpicture}
            \draw[dashed] (0,0)--(11,0);
            \draw[dashed] (0,0)--(0,4);
            \draw[dashed] (0,4)--(11,4);
            \draw[dashed] (11,4)--(11,0);

            \draw[thick] (2, 2)--(2.5, 0);
            \draw[thick] (2.5, 4)--(3.5, 0);
            \draw[thick] (3.5, 4)--(4.5, 0);

            \draw[dotted, thick] (4.5, 2)--(6, 2);

            \draw[thick] (6.5, 4)--(7, 2);

            \foreach \u in {2, 7}
                {
                    \filldraw[white] (\u, 2) circle (2pt);
                    \draw[black] (\u, 2) circle (2pt);
                }

            \draw(2.5, 4) node[above]{$\gamma$};

            \draw[
                thick,
                decoration={
                        brace,
                        mirror,
                        raise=0.5cm
                    },
                decorate
            ] (2.5, 0) -- (7, 0);
            \draw(4.5, -1) node{$a+1$ times};

        \end{tikzpicture}
    \end{displaymath}
    \caption{The curve $\gamma = \gamma_{\cO_G(a)}$.}
    \label{geometric_picture_of_induced_autoequivalences_1}
\end{figure}

The following theorem is the main result of this section, which says that half-spherical twists correspond to half twists under the mirror symmetry.

\begin{theorem}\label{thm:half-spherical-twist-and-half-twist}
    Let $F_n$ be a reducible fiber of type $\rom{1}_n$ with $n \geq 2$.
    Let $G \subset F_n$ be an irreducible component, $a \in \bZ$ be an integer, and $H_{\cO_{G}(a)} \in \Auteq D^b(F_n)$ be the half-spherical twist along the sheaf $\cO_{G}(a)$.
    Then $H_{\cO_{G}(a)}$ is mapped to the half twist $h_{\gamma_{\cO_{G}(a)}}$ along the arc $\gamma_{\cO_{G}(a)}$ on $T_n$ by the morphism $\Upsilon$ in Theorem \ref{thm:definition-of-upsilon}:
    \begin{equation}
        \Upsilon \colon \Auteq D^b(F_n) \to \MCG(T_n), \quad H_{\cO_{G}(a)} \mapsto h_{\gamma_{\cO_{G}(a)}}.
    \end{equation}
\end{theorem}
\begin{proof}
    First, we prove the statement in the case $a = -1$.
    Let us denote the half-spherical twist along the sheaf $\cO_{G}(-1)$ by $H$ and the half twist along the curve $\gamma_{\cO_{G}(-1)}$ by $h$.
    Our goal is to show that $\Upsilon(H) = h$.
    By Dehn--Nielsen--Baer theorem (Theorem \ref{thm:Dehn--Nielsen--Baer}),
    it is enough to show that $\Upsilon(H)$ and $h$ induce the same action on $\pi_1(T_n)$.
    Moreover, we only need to examine the finite number of curves $\gamma_{\cO_{F_n}}, \gamma_{\cO_{x_1}}, \dots, \gamma_{\cO_{x_n}}$ whose homotopy classes generate $\pi_1(T_n)$ (see Figure \ref{fig:corresponding-curves-via-hms}).
    We may assume $G = G_1$ and $x_1 \in G_1$ without loss of generality.
    Then both $\Upsilon(H)$ and $h$ fix $\gamma_{\cO_{F_n}}, \gamma_{\cO_{x_2}}, \dots, \gamma_{\cO_{x_n}}$ by Proposition \ref{prop:empty-intersection} and
    the definition of half twists.
    It remains to show that the curves $\gamma \coloneq \Upsilon(H)(\gamma_{\cO_{x_1}})$ and $h(\gamma_{\cO_{x_1}})$ define the same element in $\pi_1(T_n)$.
    Note that (the homotopy class of) $\gamma$ is characterized as the curve corresponding to the object $H(\cO_{x_1})$ by Theorem \ref{thm:definition-of-upsilon}, which is a simple non-separating loop as $H(\cO_{x_1})$ is spherical.
    To attack this problem, we will mimic the proof of \cite[Corollary 8.27]{opperJEMS}.
    We only consider the case $n \geq 3$.
    The case $n = 2$ can be treated similarly.
    There are identities
    \begin{align}
         & \Ext^*_{F_n}(\cO_{x_i}, H(\cO_{x_1})) = 0 \quad (i \neq 1),                \\
         & \Ext^*_{F_n}(\cO_{F_n}, H(\cO_{x_1})) = \bC,                               \\
         & \Ext^*_{F_n}(\cO_{G_i}(-1), H(\cO_{x_1})) = \begin{cases}
                                                           \bC[-1] & i = 1,               \\
                                                           \bC     & i = 2 \text{ or } n, \\
                                                           0       & \text{otherwise},
                                                       \end{cases} \\
         & \Ext^*_{F_n}(\cO_{x_1}, H(\cO_{x_1})) = \bC \oplus \bC[-1],                \\
         & \Ext^*_{F_n}(\cO_{G}, H(\cO_{x})) = \bC \oplus \bC^2[-1]
    \end{align}
    by Lemma \ref{lem:preparation-2}.
    They imply, by Theorem \ref{thm:intersections-are-morphisms},
    that the curve $\gamma = \gamma_{H(\cO_{x_1})}$
    \begin{enumerate}
        \item is disjoint from $\gamma_{\cO_{x_2}}, \dots, \gamma_{\cO{x_n}}$,
        \item intersects
              $\gamma_{O_{F_n}}$, $\gamma_{\cO_{G_1}(-1)}$,
              $\gamma_{\cO_{G_2}(-1)}, \gamma_{\cO_{G_n}(-1)}$ exactly once,
        \item intersects $\gamma_{\cO_{x_1}}$ exactly twice, and
        \item intersects $\gamma_{\cO_{G}}$ exactly three times.
    \end{enumerate}
    The first three data of intersections, combined with the fact that $\gamma$ is a simple non-separating loop, determine that the curve $\gamma$ is homotopic to either $h(\gamma_{\cO_{x_1}})$ or $h^{-1}(\gamma_{\cO_{x_1}})$, as depicted in Figure \ref{candidates_for_gamma}.
    The last condition (4) distinguishes these two as in Figure \ref{candidates_for_gamma}, and we conclude that $\gamma = h(\gamma_{\cO_{x_1}})$ in $\pi_1(T_n)$.

    For general $a \in \bZ$, we employ Proposition \ref{prop:half-spherical-twist-and-line-bundle} and similar relations for half twists to use induction on $a$.
    Let us denote the twist functor $T_{\cO_{x_i}} = (-)\otimes \cO(x_i)$ by $T_i$ and the Dehn twist along $\gamma_{\cO_{x_i}}$ by $t_i$.
    They are related by $\Upsilon(T_i) = t_i$.
    Additionally, we denote $H_{\cO_{G_i}(a)}$ by $H_{i, a}$ and the half twist along $\gamma_{\cO_{G_i}(a)}$ by $h_{i, a}$.
    The Dehn twists and the half twists are known to satisfy the relations
    \begin{align}
        t_i h_{i, a} t_i^{-1}          & = h_{i, a+1},            \\
        h_{i, a} t_i h_{i, a} t_i^{-1} & = t_{i-1}t_i^{-2}t_{i+1}
    \end{align}
    for $i \in \bZ/n\bZ$ and $a \in \bZ$.
    The first relation follows from the definition of half twists, and the second one is found in the proof of \cite[Corollary 2.11]{MR1805936} for instance.
    Combining them, we have relations
    \begin{equation}\label{eq:relation-between-half-twists-and-dehn-twists}
        h_{i, a}h_{i, a+1} = t_{i-1}t_i^{-2}t_{i+1}
    \end{equation}
    in $\MCG(T_n)$ for all $i$ and $a$.
    On the other hand, by Proposition \ref{prop:half-spherical-twist-and-line-bundle}, we have
    \begin{equation}
        H_{i, a}\circ H_{i, a+1} = (-)\otimes \cO_S(G_i)\lvert_{F_n}
    \end{equation}
    in $\Auteq D^b(F_n)$.
    The morphism $\Upsilon$ maps the right-hand side to $t_{i-1}t_i^{-2}t_{i+1}$ by the following observations:
    \begin{enumerate}
        \item The line bundles $\cO_S(G_i)\lvert_{F_n}$ and $\cO_{F_n}(x_{i-1}-2x_{i}+x_{i+1})$ have the same multi-degree.
        \item The group $\Pic^0(F_n)$ of multi-degree zero line bundles is contained in $\Ker \Upsilon$ (Theorem \ref{thm:autoequivalence-of-I_n-curve}).
        \item $(-) \otimes \cO(x_i) \cong T_i$ (Example \ref{ex:spherical-objects}).
        \item $\Upsilon(T_i) = t_i$ (Theorem \ref{thm:spherical-twist-and-dehn-twist}).
    \end{enumerate}
    Therefore, we have relations
    \begin{equation}\label{eq:relation-between-half-spherical-twists-and-spherical-twists}
        \Upsilon(H_{i, a})\Upsilon(H_{i, a+1}) = t_{i-1}t_i^{-2}t_{i+1}
    \end{equation}
    in $\MCG(T_n)$.
    Comparing \eqref{eq:relation-between-half-twists-and-dehn-twists} and \eqref{eq:relation-between-half-spherical-twists-and-spherical-twists}, the (forward and backward) induction on $a$ starting from $a = -1$ provides the desired identities $\Upsilon(H_{i, a}) = h_{i, a}$ for all $i$ and $a$.
\end{proof}

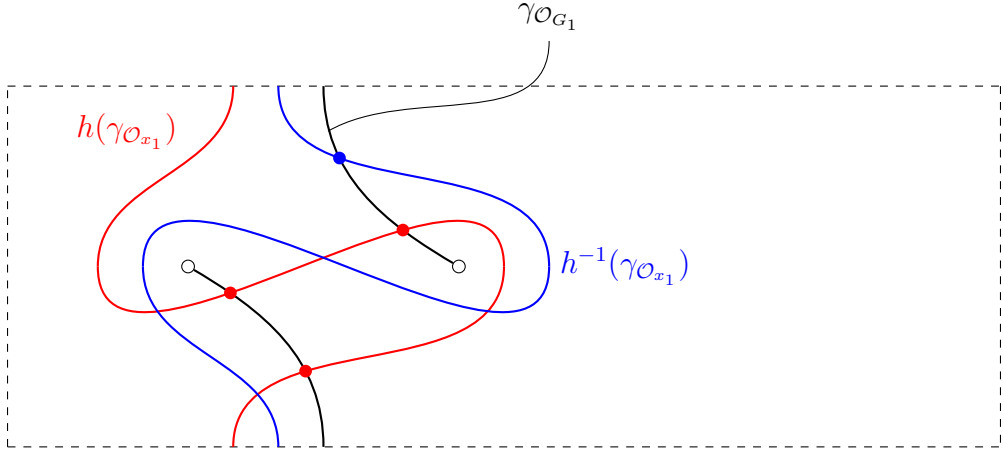
\begin{figure}[h]
    \centering
    \begin{displaymath}
        \begin{tikzpicture}[scale=1.2]
            \draw[dashed] (0,0)--(11,0);
            \draw[dashed] (0,0)--(0,4);
            \draw[dashed] (0,4)--(11,4);
            \draw[dashed] (11,4)--(11,0);

            \draw[thick, name path=black_1](3.5, 4) to[out=-90,in=150](5, 2);
            \draw[thick, name path=black_2](2, 2) to[out=-30,in=90](3.5, 0);

            \draw[thick, color=red](2.5, 4) to[out=-90,in=90](1, 2);
            \draw[thick, name path=red_2, color=red](1, 2) to[out=-90,in=90](5.5, 2);
            \draw[thick, name path=red_3, color=red](5.5, 2) to[out=-90,in=90](2.5, 0);

            \draw[thick, name path=blue_1, color=blue](3, 4) to[out=-90,in=90](6, 2);
            \draw[thick, color=blue](6, 2) to[out=-90,in=90](1.5, 2);
            \draw[thick, color=blue](1.5, 2) to[out=-90,in=90](3, 0);

            \foreach \u in {2, 5}
                {
                    \filldraw[white] (\u, 2) circle (2pt);
                    \draw[black] (\u, 2) circle (2pt);
                }

            \fill [name intersections={of=red_2 and black_1}, color=red]
            (intersection-1) circle (2pt) node [above right]{};
            \fill [name intersections={of=red_2 and black_2}, color=red]
            (intersection-1) circle (2pt) node [above right]{};
            \fill [name intersections={of=red_3 and black_2}, color=red]
            (intersection-1) circle (2pt) node [above right]{};

            \fill [name intersections={of=blue_1 and black_1}, color=blue]
            (intersection-1) circle (2pt) node [above right]{};

            \draw[](2, 3.5) node[left, color=red]{$h(\gamma_{\cO_{x_1}})$};
            \draw[](6, 2) node[right, color=blue]{$h^{-1}(\gamma_{\cO_{x_1}})$};

            \draw(6, 4.5) node[above]{$\gamma_{\cO_{G_1}}$};
            \draw(6, 4.5) to[out=-90,in=30](3.55, 3.5);
        \end{tikzpicture}
    \end{displaymath}
    \caption{The two possible candidates for $\gamma$.} \label{candidates_for_gamma}
\end{figure}

\subsection{Autoequivalences of elliptic surfaces}
In this subsection, we give a refinement of Uehara's result Theorem \ref{thm:autoequivalence-of-elliptic-surface} by using half-spherical twists.
Following \cite{MR3568337} we define a subgroup $B \subset \Auteq D^b(S)$ by
\begin{equation}
    B = \langle T_{\cO_G(a)} \mid G \subset S \text{ is an irreducible component of a reducible fiber, } a \in \bZ\rangle.
\end{equation}

\begin{remark}\label{rem:alternative-description-of-minus-two-curves}
    If we assume that $S$ is projective and its Kodaira dimension $\kappa(S)$ is non-zero, the group $B$ coincides with the one in Theorem \ref{thm:autoequivalence-of-elliptic-surface}
    \begin{equation}
        B = \langle T_{\cO_G(a)} \mid G \subset S \text{ is a $(-2)$-curve, } a \in \bZ\rangle
    \end{equation}
    as mentioned in \cite[Section 3]{MR3568337}.
    We give a brief explanation of this fact.
    Let $G \subset S$ be a $(-2)$-curve and set $d \coloneq G \cdot F$.
    The adjunction formula and the canonical bundle formula for elliptic surfaces \cite[Corollary V.12.3]{MR2030225} computes $K_S \cdot G$ as
    \begin{equation}
        0 = K_S \cdot G = d \cdot \left(\chi(\cO_S) - 2 \chi(\cO_C) + \sum_i (1-m_i^{-1})\right)
    \end{equation}
    where $m_i$ is the multiplicity of the $i$-th multiple fiber.
    On the other hand, \cite[Proposition V.12.5]{MR2030225} says that the Kodaira dimension $\kappa(S)$ is non-zero if and only if the value $\delta \coloneq \chi(\cO_S) - 2 \chi(\cO_C) + \sum_i (1-m_i^{-1})$ is non-zero.
    Therefore, we have $d = 0$ and hence $G$ is contained in a fiber.
\end{remark}
By Corollary \ref{cor:FM-to-Auteq}, Proposition \ref{prpo:forget-base}, and Proposition \ref{prop:twist-functor-is-relative-fm}, the group $B$ can be viewed as a subgroup of $\FM_T(S)$. We have a natural ``restriction'' morphism
\begin{equation}\label{eq:definition-of-res}
    \res_F \colon B \hookrightarrow \FM_T(S) \to \FM_{\bC}(F) \hookrightarrow \Auteq D^b(F), \quad T_{\cO_G(a)} \mapsto \begin{cases}
        H_{\cO_G(a)} & G \subset F      \\
        \id          & \text{otherwise}
    \end{cases}
\end{equation}
for each fiber $F$.
\begin{remark}\label{rem:action-on-grothendieck-group-is-trivial}
    Consider the composition
    \begin{equation}
        B \xrightarrow{\res_F} \Auteq D^b(F) \xrightarrow{(-)\vert_{\Perf(F)}} \Auteq \Perf(F) \to \Aut K_0(F),
    \end{equation}
    where $K_0(F)$ is the Grothendieck group of $\Perf(F)$.
    Note that the map $\Auteq D^b(F) \xrightarrow{(-)\vert_{\Perf(F)}} \Auteq \Perf(F)$ is well-defined due to \cite[Proposition 1.11]{MR2437083}.

    We claim that this map is trivial.
    For $\cF \in \Perf(F)$ and $\cE = \cO_G(a)$, recall that there is an exact triangle
    \begin{equation}\label{eq:exact-triangle-of-half-spherical-twist}
        \Phi_{j^*j_*(\cE' \boxtimes \cE)}(\cF) \to \cF \to H_{\cE}(\cF) \xrightarrow{+1}
    \end{equation}
    in $\Perf(F)$, where $\cE' = \CHom(\cE, i^!\cO_S)$ (see proof of Proposition \ref{prop:empty-intersection}).
    On the other hand, there is another exact triangle
    \begin{equation}
        \cE' \boxtimes \cE[1] \to j^*j_*(\cE' \boxtimes \cE) \to \cE'\boxtimes \cE \xrightarrow{+1}
    \end{equation}
    by \cite[Corollary 11.4]{MR2244106}\footnote{Although the statement given there is for smooth hypersurface, the proof remains valid for our cases.} and $\cO_{S \times_C S}(-F\times F)\vert_{F \times F} \cong \cO_{F \times F}$.
    It induces an exact triangle
    \begin{equation}\label{eq:exact-triangle-of-cocone}
        \Gamma(\cE' \otimes \cF) \otimes_\bC \cE[1] \to \Phi_{j^*j_*(\cE' \boxtimes \cE)}(\cF) \to  \Gamma(\cE' \otimes \cF) \otimes_\bC \cE \xrightarrow{+1}.
    \end{equation}
    Here $\Gamma(\cE' \otimes \cF)$ is a bounded complex of finite dimensional vector spaces as follows: First, since $\cE' \in D^b(F)$ and $\cF \in \Perf(F)$, one has $\cE' \otimes \cF \in D^b(F)$. Second, the functor $\Gamma$ is the pushforward along the proper morphsim $F \to \Spec \bC$ so that $\Gamma(\cE' \otimes \cF) \in D^b(\Spec \bC)$.
    Therefore, the triangle \eqref{eq:exact-triangle-of-cocone} is contained in $\Perf(F)$ and implies $[\Phi_{j^*j_*(\cE' \boxtimes \cE)}(\cF)] = 0$ in $K_0(F)$.
    Then \eqref{eq:exact-triangle-of-half-spherical-twist} gives $[H_{\cE}(\cF)] = [\cF]$ in $K_0(F)$.

    Note that a similar claim for non-perfect objects does not hold.
    For example, the $2$-spherical object $i_*\cO_G(a)$ satisfies $T_{i_*\cO_G(a)}(i_*\cO_G(a)) \cong i_*\cO_G(a)[-1]$ in $D^b(S)$ \cite[Excersice 8.5]{MR2244106}.
    Then we have $i_*H_{\cO_G(a)}(\cO_G(a)) \cong i_*\cO_G(a)[-1]$ by Corollary \ref{cor:compatibility-of-half-spherical-twists-and-spherical-twists}, and we can remove $i_*$ from both sides as $\cO_G(a)$ is a sheaf.
\end{remark}

In order to capture the whole information of the group $B$, we consider the restriction morphisms $\res_F$ for all the reducible fibers $F$ at the same time.
Let $\res = \prod_{F \colon \text{reducible fiber}} \res_F$ be the product of these restriction morphisms.
\begin{proposition}\label{prop:kernel-of-res}
    We have an exact sequence
    \begin{equation}
        1 \to \langle (-) \otimes \cO_S(F) \mid F \text{: reducible fiber }\rangle \to B \xrightarrow{\res} \prod_{F \colon \text{reducible fiber}} \Auteq D^b(F).
    \end{equation}
\end{proposition}
\begin{proof}
    Let $D \subset C$ be the union of all points whose fibers are reducible.
    Let $U \coloneq S \setminus \bigcup_{F \colon \text{reducible fiber}}F = \pi^{-1} (C \setminus D)$ be the complement of the reducible fibers.
    Notice that the restriction $\FM_C(S) \to \FM_{C\setminus D}(U)$ (Proposition \ref{prop:restriction-to-fiber}) where takes all the elements of $B$ to the identity.
    In particular, every equivalence $\Phi \in B$ satisfies $\Phi(\cO_x) \cong \cO_x$ for every point $x \in U$.
    For $x \in S\setminus U$, let $F \subset S$ be the reducible fiber containing $x$ and $i_F \colon F \hookrightarrow S$ be the inclusion.
    We have
    \begin{equation}
        \Phi(\cO_x) \cong i_{F*}(\res_F(\Phi)(\cO_x))
    \end{equation}
    by Proposition \ref{prop:restriction-to-fiber}.
    Therefore, $\Phi(\cO_x) \cong \cO_x$ holds for every $\Phi \in \Ker(\res) = \bigcap_{F \colon \text{reducible fiber}} \Ker(\res_F)$ and every $x \in S$.
    This implies the inclusion $\Ker(\res) \subset \Pic(S)$ by Lemma \ref{lem:criterion-to-be-standard-functor} and hence $\Ker(\res) = B \cap \Ker \left(\Pic(S) \to \prod_{F \colon \text{reducible fiber}} \Pic(F)\right)$.

    We then prove the identity
    \begin{equation}\label{eq:kernel-of-res}
        B \cap \Ker \left(\Pic(S) \to \prod_{F \colon \text{reducible fiber}} \Pic(F)\right) = \langle (-) \otimes \cO_S(F) \mid F \text{: reducible fiber }\rangle.
    \end{equation}
    The inclusion $\Ker \left(\Pic(S) \to \prod_{F \colon \text{reducible fiber}} \Pic(F)\right) \supset (\text{RHS})$ is satisfied as $\cO_S(F)\vert_F$ is trivial by general theory of fibration.
    We also have $B \supset (\text{RHS})$ by \cite[Lemma 4.15 (i)(2)]{MR2198807}.

    To prove the other inclusion, let $\cL \in \Pic(S)$ be a line bundle such that $(-) \otimes \cL$ is an element of the left-hand side.
    The first observation in this proof shows $\cL|_U \cong \cO_U$.
    Together with $\cL \in \Ker \left(\Pic(S) \to \prod_{F \colon \text{reducible fiber}} \Pic(F)\right)$, $\cL$ restricts trivially to all fibers of $\pi \colon S \to C$.
    Then \cite[Corollary III.12.9]{MR0463157} says the sheaf $\cM \coloneq R^0\pi_*\cL$ is a line bundle on $C$, and the counit morphism $\pi^*\cM \to \cL$ is an isomorphism.
    Also, we have $\cM \vert_{C \setminus D} \cong R^0\pi_*(\cL\vert_U) \cong R^0\pi_*(\cO_U) \cong\cO_{C \setminus D}$ (the first isomorphism is projection formula and the third one is Zariski's main theorem \cite[Corollary III.11.3]{MR0463157}).
    Therefore, $\cM$ is of the form $\cO_C(\sum_{p \in D} a_p p)$ for some integers $a_p$ and hence $\cL \cong \pi^*\cM \in (\text{RHS})$.
\end{proof}

\begin{theorem}\label{thm:description-of-B}
    Let $\pi \colon S \to C$ be a relatively minimal, smooth projective elliptic surface and $F^{(1)}, \dots, F^{(r)}$ be its reducible fibers.
    Suppose that each $F^{(j)}$ is of type $\rom{1}_{n_j}$ for some $n_j \geq 2$.
    Let $B \subset \Auteq D^b(S)$ be the subgroup defined by
    \begin{equation}
        B = \langle T_{\cO_G(a)} \mid G \subset S \text{ is an irreducible component of a reducible fiber, } a \in \bZ\rangle.
    \end{equation}
    Then the following statements hold.
    \begin{enumerate}
        \item The exact sequence in Proposition \ref{prop:kernel-of-res} together with the morphism $\Upsilon$ in Theorem \ref{thm:spherical-twist-and-dehn-twist} induces the exact sequence
              \begin{equation}
                  1 \to \langle (-)\otimes \cO_S(F^{(j)}) \mid 1 \leq j \leq r \rangle \to B \xrightarrow{\psi} \prod_{j = 1}^r \MCG(T_{n_j}),
              \end{equation}
              where $\psi$ is the composition of $\res$ and $\Upsilon$.
        \item Let $p_j \colon \prod_{j=1}^r \MCG(T_{n_j}) \to \MCG(T_{n_j})$ be the projection onto the $j$-th component. Then the image $\Image(\psi)$ is the product of its projections onto the individual components $\Image(\psi) = \prod_{j = 1}^r \Image(p_j \circ \psi)$.
        \item The set of half twists along the curves $\gamma_{\cO_{G_1}(-1)}, \dots, \gamma_{\cO_{G_{n_j}}(-1)}, \gamma_{\cO_{G_1}}, \dots, \gamma_{\cO_{G_{n_j}}}$ on $T_{n_j}$ generates the subgroup $\Image(p_j \circ \psi) \subset \MCG(T_{n_j})$ (see Figure \ref{fig:generators-for-image-of-B}).
        \item If the Kodaira dimension $\kappa(S)$ is non-zero, then the group $B$ coincides with the group
              \begin{equation}
                  \langle T_{\cO_G(a)} \mid G \subset S \text{ is a $(-2)$ curve, } a \in \bZ\rangle,
              \end{equation}
              which was originally denoted by $B$ in \cite{MR3568337}.
    \end{enumerate}
\end{theorem}

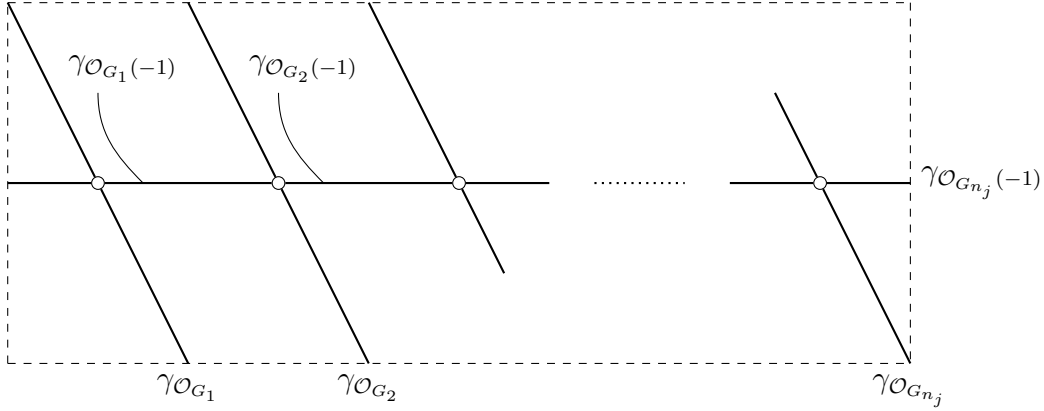
\begin{figure}[h]
    \centering
    \begin{displaymath}
        \begin{tikzpicture}[scale=1.2]
            \draw[dashed] (0,0)--(10,0);
            \draw[dashed] (0,0)--(0,4);
            \draw[dashed] (0,4)--(10,4);
            \draw[dashed] (10,4)--(10,0);

            \draw[thick] (0, 2)--(1, 2);
            \draw[thick] (1, 2)--(3, 2);
            \draw[thick] (3, 2)--(5, 2);
            \draw[thick] (5, 2)--(6, 2);
            \draw[thick] (8, 2)--(9, 2);
            \draw[thick] (9, 2)--(10, 2);

            \draw[thick] (0, 4)--(1, 2);
            \draw[thick] (1, 2)--(2, 0);
            \draw[thick] (2, 4)--(4, 0);
            \draw[thick] (4, 4)--(5, 2);
            \draw[thick] (5, 2)--(5.5, 1);
            \draw[thick] (8.5, 3)--(9, 2);
            \draw[thick] (9, 2)--(10, 0);

            \draw[dotted, thick] (5+1.5, 2)--(9-1.5, 2);
            \foreach \u in {1, 3, 5, 9}
                {
                    \filldraw[white] (\u, 2) circle (2pt);
                    \draw[black] (\u, 2) circle (2pt);
                }

            \draw(2, 0) node[below]{$\gamma_{\cO_{G_1}}$};
            \draw(4, 0) node[below]{$\gamma_{\cO_{G_2}}$};
            \draw(10, 0) node[below]{$\gamma_{\cO_{G_{n_j}}}$};

            \draw(2, 3) node[above left]{$\gamma_{\cO_{G_1}(-1)}$};
            \draw(1, 3) to[out=-90,in=135](1.5, 2);
            \draw(4, 3) node[above left]{$\gamma_{\cO_{G_2}(-1)}$};
            \draw(3, 3) to[out=-90,in=135](3.5, 2);

            \draw(10, 2) node[right]{$\gamma_{\cO_{G_{n_j}}(-1)}$};

        \end{tikzpicture}
    \end{displaymath}
    \caption{The simple arcs such that the half twists along them generate $\Image(B \xrightarrow{p_j \circ \psi} \MCG(T_{n_j}))$.} \label{fig:generators-for-image-of-B}
\end{figure}
\begin{proof}
    The morphism $\psi$ in (1) is defined by the composition
    \begin{equation}
        \psi = \prod_{j=1}^r \Upsilon_j \circ \res_j \colon B \to \Auteq D^b(F^{(j)}) \to \MCG(T_{n_j}),
    \end{equation}
    where $\Upsilon_j$ is the morphism $\Upsilon$ in Theorem \ref{thm:spherical-twist-and-dehn-twist} for the reducible fiber $F^{(j)}$.
    We need to prove that the morphism
    \begin{equation}
        \Image\left(\res_j \colon B \to \Auteq D^b(F^{(j)})\right) \xrightarrow{\Upsilon_j} \MCG(T_{n_j})
    \end{equation}
    is injective for each $j$.

    Suppose that $\Phi \in B$ satisfies
    \begin{equation}\label{eq:condition-for-Phi}
        \res_j(\Phi) = \varphi_*(\cL \otimes (-))[n] \in \Ker(\Upsilon_j) = \Aut^0(F^{(j)})\times \Pic^0(F^{(j)}) \times \bZ[1]
    \end{equation}
    for some $\varphi \in \Aut^0(F^{(j)})$, $\cL \in \Pic(F^{(j)})$, and $n \in \bZ$.
    We want to show $\res_j(\Phi) = \id$.
    Note that we may assume that $\Phi$ is a composition of the functors
    \begin{equation}
        \{T_{\cO_{G}(a)} \mid G \subset F^{(j)} \colon \text{irreducible component}, a \in \bZ\}
    \end{equation}
    because the contribution from the other reducible fibers does not change the image $\res_j(\Phi)$.
    Then we have
    \begin{equation}\label{eq:condition_for_Phi_1}
        \Phi(\cO_x) = \begin{cases}
            \cO_x               & x \in S \setminus F^{(j)}, \\
            \cO_{\varphi(x)}[n] & x \in F^{(j)}.
        \end{cases}
    \end{equation}
    Now Lemma \ref{lem:criterion-to-be-standard-functor} shows
    \begin{equation}\label{eq:condition_for_Phi_2}
        \Phi(-) = \widetilde{\varphi}_*(- \otimes \widetilde{\cL})[m]
    \end{equation}
    for some $\widetilde{\varphi}\in \Aut{S}$, $\widetilde{L} \in \Pic{S}$, and $m \in \bZ$.
    Comparing \eqref{eq:condition_for_Phi_1} and \eqref{eq:condition_for_Phi_2}, we have $m = n = 0$ and
    \begin{equation}
        \widetilde{\varphi}(x) = \begin{cases}
            x          & x \in S \setminus F^{(j)}, \\
            \varphi(x) & x \in F^{(j)}.
        \end{cases}
    \end{equation}
    Since $S \setminus F^{(j)}$ is dense open in $S$, we have $\widetilde{\varphi} = \id_S$ and then $\varphi = \id_{F^{(j)}}$.
    Therefore, $\Phi = (-) \otimes \widetilde{\cL}$ and $\res_j(\Phi) = (-) \otimes \widetilde{\cL}\vert_{F^{(j)}}$.
    Combined with \eqref{eq:condition-for-Phi}, $\cL = \widetilde{\cL}\lvert_{F^{(j)}} \in \Pic^0(F^{(j)})$ follows.
    Since $\Phi = (-) \otimes \widetilde{\cL}$ is in $B \cap \Pic(S)$, the line bundle $\widetilde{\cL}$ is an element of the group $\Gamma \subset \Pic(S)$ appearing in the proof of the previous proposition.
    The condition $\widetilde{\cL}\lvert_{F^{(j)}} \in \Pic^0(F^{(j)})$ means that the restriction $\widetilde{\cL} \lvert_{G_{jk}}$ to each irreducible component $G_{jk} \subset F^{(j)}$ is trivial.
    Again by the same argument in the proof of the previous proposition, the line bundle $\cL = \widetilde{\cL}\lvert_{F^{(j)}}$ must be trivial.
    This shows $\res_j(\Phi) = (-) \otimes \cL = \id$.

    The statement (2) follows from the fact that for any generator $T_{\cO_G(a)} \in B$ there exists a reducible fiber $F$ such that $G \subset F$.

    The image $\Image(p_j \circ \psi \colon B \to \MCG(T_{n_j}))$ does not change when we restrict the domain to the subgroup $B_j = \langle T_{\cO_G(a)} \mid G \subset F^{(j)}, a \in \bZ \rangle \subset B$.
    To show (3), it is enough to prove the group $B_j$ is generated by $T_{\cO_{G_1}}, \dots, T_{\cO_{G_{n_j}}}, T_{\cO_{G_1}(-1)}, \dots, T_{\cO_{G_{n_j}}(-1)}$.
    This is an immediate consequence of the relation $T_{\cO_G(a)} \circ T_{\cO_G(a+1)} = (-)\otimes \cO_S(G)$ for a surface $S$ and a $(-2)$-curve $G$ \cite[Lemma 4.15 (i)]{MR2198807}.

    The last statement (4) is just a rephrasing of Remark \ref{rem:alternative-description-of-minus-two-curves}.
\end{proof}

\begin{remark}\label{rm:psi-is-not-surjective}
    The morphism $\psi \colon B \to \prod_{j=1}^r \MCG(T_{n_j})$ is not surjective.
    We can prove this as follows.
    Assume $r=1$ for simplicity and set $n=n_1$, $F = F^{(1)} = G_1 \cup \cdots \cup G_n$.
    The morphism $\psi$ is given by the composition
    \begin{equation}
        B \xrightarrow{\res} \Auteq D^b(F) \xrightarrow{\Upsilon} \MCG(T_n).
    \end{equation}
    Consider the Dehn twist $t_{\gamma_{\cO_{x_1}}}$ along the curve $\gamma_{\cO_{x_1}} \subset T_{n}$ (see Example \ref{ex:corresponding-curves-via-hms}).
    We know the twist functor $T_{\cO_{x_1}} = - \otimes \cO(x_1)$ is mapped to $t_{\gamma_{\cO_{x_1}}}$ by $\Upsilon$.
    Then the set $\Upsilon^{-1}(\gamma_{\cO_{x_1}})$ consists of the elements of the form $T_{\cO_{x_1}} \circ \Phi$ for $\Phi \in \Aut^0(F) \times \Pic^0(F) \times \bZ[1]$ (Theorem \ref{thm:autoequivalence-of-I_n-curve}).
    In particular, all such elements non-trivially act on the Grothendieck group $K_0(F)$ of $\Perf(F)$, as we have $K_0(F) \cong H^0(F, \bZ) \oplus \Pic(F)$ \cite[Proposition 4.2]{MR3713873}.
    Then by Remark \ref{rem:action-on-grothendieck-group-is-trivial}, we have $\res(B) \cap \Upsilon^{-1}(\gamma_{\cO_{x_1}}) = \emptyset$.
    This shows that $\psi$ is not surjective.
\end{remark}
\begin{remark}
    Let $B_j = \langle T_{\cO_G(a)} \mid G \subset F^{(j)}, a \in \bZ \rangle$ be the subgroup of $B$.
    Since $B_j$ and $B_{j'}$ commute with each other for $j \neq j'$, we have a surjective morphism $\prod_j B_j \twoheadrightarrow B$.
    On the other hand, this morphism is not injective in general.
    For example, if $C = \bP^1$ then all the fibers of $\pi$ are linearly equivalent to each other.
    Then the elements $- \otimes \cO_S(F^{(j)}) \in B_j$ have nontrivial relations in $B$.
\end{remark}

\begin{remark}
    Let
    $
        \overline{S} \coloneqq \Spec \bC[x,y,z]/(x y - z^{n+1})
        \subset \Spec \bC[x,y,z] \cong \bA^3
    $
    be an affine surface with an $A_n$-singularity at the origin
    and
    $
        \overline{f} \colon \overline{S} \to \Spec \bC[z]
    $
    be the projection to the $z$-axis.
    Let further
    $
        \pi \colon S \to \overline{S}
    $
    be the minimal resolution
    and
    $
        f \coloneqq \overline{f} \circ \pi
    $
    be the composite morphism.
    The fiber
    $
        F
        \coloneqq f^{-1}(0)
    $
    is the union
    of the exceptional locus
    $
        E \coloneqq \bigcup_{i=1}^n G_i
    $
    of the resolution $\pi$
    and two copies of the affine line.
    Note that if $n=1$ the same situation can also be obtained from Example \ref{example_from_Atiyah_flop} by cutting with the hyperplane $\{z=w\}$,
    which are strict transforms of the
    coordinate lines
    $
        \Spec \bC[x,y,z]/(xy, z)
        \cong \Spec \bC[x] \cup \Spec \bC[y].
    $
    There is a homomorphism
    \begin{align}
        \rho \colon B_n^{(1)} \to \Auteq(D^b(S))
    \end{align}
    from the \emph{affine braid group}
    $
        B_n^{(1)}
    $
    generated by $\sigma_i$ for $i=0,\ldots,n$
    with relations
    \begin{align}
        \sigma_i \sigma_{i+1} \sigma_i & = \sigma_{i+1} \sigma_i \sigma_{i+1}, &  & i = 0, \ldots, n \quad(\sigma_{n+1} \coloneqq \sigma_0), \\
        \sigma_i \sigma_j              & = \sigma_j \sigma_i,                  &  & |i-j|>2,
    \end{align}
    sending $\sigma_0$ to $T_{\omega_E}$
    and $\sigma_i$ to
    $
        T_{\cO_{G_i}(-1)}
    $
    for $i=1, \ldots, n$.
    Proposition \ref{prop:restriction-to-fiber} gives
    a homomorphism
    \begin{align}
        \res_F \colon \Image \rho \to \Auteq D^b(F).
    \end{align}
    One has an equivalence
    \begin{align} \label{eq:HMS}
        D^b(F) \simeq D^b(\cW(\Sigma_{0,n+3}))
    \end{align}
    where $\cW(\Sigma_{0,n+3})$
    is the wrapped Fukaya category of a surface $\Sigma_{0,n+3}$
    of genus 0 with $n+3$ punctures
    equipped with a suitable grading
    \cite{MR3073884, MR3735868, MR3830878}.
    The composite $\res_F \circ \rho$
    together with the equivalence \eqref{eq:HMS}
    gives an affine braid group action
    on $D^b(\cW(\Sigma_{0,n+3}))$,
    which factors through the action of the graded mapping class group
    of $\Sigma_{0,n+3}$
    just as in the case of Kodaira fibers discussed above.
    This implies the injectivity of $\rho$,
    which is one of the main results of
    \cite{MR2629510}.
\end{remark}

\appendix
\section{Proof of Lemma \ref{lem:criterion-to-be-standard-functor} and Proposition \ref{prop:twist-functor-is-relative-fm}}\label{section:appendix}
\subsection{Proof of Lemma \ref{lem:criterion-to-be-standard-functor}}\label{subsection:proof-of-criterion-to-be-standard-functor}
In this subsection, we denote the product over the base field $k$ by $\times$ for simplicity.
\begin{lemma}[Lemma \ref{lem:criterion-to-be-standard-functor}]\label{lem:criterion-to-be-standard-functor-re}
    Let $k$ be an algebraically closed field with characteristic zero.
    Let $X$ and $Y$ be connected quasi-projective schemes over $k$, with $X$ reduced.
    Let $\cP \in D^b(X \times Y)$ be an object whose support is proper over $X$.
    Suppose that the integral functor $\Phi = \Phi_{\cP}$ satisfies the following condition:

    \begin{quote}
        For any closed point $x \in X$, there exists a closed point $y \in Y$ and an integer $n_x$ such that $\Phi(\cO_x) \cong \cO_y[n_x]$.
    \end{quote}

    Then there exists a morphism $f \colon X \to Y$, a line bundle $\cL \in \Pic X$, and an integer $n$ such that $\Phi \cong f_*(\cL \otimes -)[n]$.
\end{lemma}
\begin{proof}
    We basically repeat the proof of \cite[Lemma 2.14]{MR2155085} with necessarily modification.
    \textbf{From here to the end of the proof, the symbols such as $\pi_{X*}$, $i_*$, or $\otimes$ denote non-derived functors. We add $R$ or $L$ to derived functors to avoid confusion.}

    Let $\pi_X \colon X \times Y \to X$ and $\pi_Y \colon X \times Y \to Y$ be the natural projections and $i_x \colon \{x\} \times Y \hookrightarrow X \times Y$ be the natural inclusion for each closed point $x \in X$.
    \begin{equation}
        \begin{tikzcd}
            \{x\} \times Y \ar[r, "i_x"] \ar[d] & X \times Y \ar[d, "\pi_X"] \ar[r, "\pi_Y"]& Y \\
            \{x\} \ar[r] & X
        \end{tikzcd}
    \end{equation}

    We have $Li_x^*(\cP) \cong \cO_{(x, y)}[n_x]$ from the assumption.
    By Lemma \ref{lem:degree-is-locally-constant} below $x \mapsto n_x$ is a locally constant function on the set of closed points of $X$,
    and hence constant since we assumed that $X$ is connected.
    Let $n = n_x$ be the constant value.
    We may assume $n = 0$ by replacing $\cP$ with its $n$-shift.
    Then by \cite[Lemma 4.3]{MR1651025} the complex $\cP$ is a coherent sheaf on $X \times Y$ and flat over $X$.
    Note that then $\cP\vert_{\{x\} \times Y} = i_x^*\cP \cong Li_x^*\cP \cong \cO_{(x, y)}$ holds for each closed point $x \in X$.

    Next, we claim that the canonical morphism
    \begin{equation}\label{eq:canonical_surjection}
        \pi_X^* \pi_{X*}(\cP \otimes \pi_Y^*\cA^{\otimes d}) \to \cP \otimes \pi_Y^*\cA^{\otimes d}
    \end{equation}
    is surjective for sufficiently large $d$, where $\cA$ is a very ample line bundle on $Y$.
    To apply \cite[Theorem III.8.8]{MR0463157} which requires properness, let $i \colon Z \hookrightarrow X \times Y$ be the schematic support of $\cP$ and consider the following diagram:
    \begin{equation}
        \begin{tikzcd}
            & Z \ar[d, "i"] \ar[ldd, "p_X"', bend right] \ar[rdd, "p_Y", bend left]& \\
            & X \times Y \ar[ld, "\pi_X"] \ar[rd, "\pi_Y"']& \\
            X & & Y.
        \end{tikzcd}
    \end{equation}
    Note that $p_X \colon Z \to X$ is proper as $Z_{\red} = \Supp(\cP)$ is proper over $X$ by assumption and \cite[\href{https://stacks.math.columbia.edu/tag/0CYK}{Tag 0CYK}]{stacks-project}.
    Together with the fact that $Z$ and $X$ are quasi-projective, $p_X$ is projective.
    In addition, the line bundle $\pi_Y^*\cA$ is very ample on $X \times Y$ relative to $X$, and its restriction $p_Y^*\cA = (\pi_Y^*\cA)\vert_Z$ is very ample on $Z$ relative to $X$.
    Therefore, we have a surjection
    \begin{equation}
        p_X^* p_{X*}(\cP\vert_Z \otimes p_Y^*\cA^{\otimes d}) \twoheadrightarrow \cP\vert_Z \otimes p_Y^*\cA^{\otimes d}
    \end{equation}
    for large $d$ by \cite[Theorem III.8.8]{MR0463157}.
    By applying $i_*$ there is a horizontal surjective morphism
    \begin{equation}\label{eq:canonical_surjection_2}
        \begin{tikzcd}
            i_*p_X^* p_{X*}(\cP\vert_Z \otimes p_Y^*\cA^{\otimes d}) \ar[r, twoheadrightarrow] \ar[d, "\sim"] & i_*(\cP\vert_Z \otimes p_Y^*\cA^{\otimes d}) \ar[d, "\sim"] \\
            i_*i^*\pi_X^*\pi_{X*}(\cP \otimes \pi_Y^*\cA^{\otimes d}) & \cP \otimes \pi_Y^*\cA^{\otimes d}.
        \end{tikzcd}
    \end{equation}
    Here the vertical isomorphisms are due to $p_X = \pi_X \circ i$ and
    \begin{align}
        i_*(\cP\vert_Z \otimes p_Y^*\cA^{\otimes d}) & = i_*(\cP\vert_Z \otimes i^*\pi_Y^*\cA^{\otimes d})  \\
                                                     & \cong i_*(\cP\vert_Z) \otimes \pi_Y^*\cA^{\otimes d} \\
                                                     & \cong \cP \otimes \pi_Y^*\cA^{\otimes d},
    \end{align}
    where the latter follows from the non-derived projection formula and $i_*(\cP\vert_Z) \cong \cP$.
    Then our claim follows from \eqref{eq:canonical_surjection_2} and the surjection $\id \to i_*i^*$.

    On the other hand, the sheaf $\pi_{X*}(\cP \otimes \pi_Y^*\cA^{\otimes d})$ is a line bundle as follows.
    First, there is an isomorphism $\pi_{X*}(\cP \otimes \pi_Y^* \cA^{\otimes d}) \cong p_{X*}(P\vert_Z \otimes p_Y^* \cA^{\otimes d})$
    by the projection formula as above.
    In addition, we have
    \begin{align}\label{eq:rank-function}
        H^0((\cP\vert_Z \otimes p_Y^* \cA^{\otimes d})\vert_{p_X^{-1}(x)}) & \cong H^0((\cP \vert_{\{x\} \times Y})\vert_{p_X^{-1}(x)} \otimes p_Y^* \cA^{\otimes d}\vert_{p_X^{-1}(x)}) \\                                                     & \cong H^0(\cO_{(x, y)}) \\
                                                                           & \cong k
    \end{align}
    for every closed point $x \in X$.
    Here the first isomorphism is due to $p_X^{-1}(x) = Z \cap (\{x\} \times Y)$, the second is due to $\cP\vert_{\{x\} \times Y} \cong \cO_{(x, y)}$, and the third is due to $k = \overline{k}$.
    Then \cite[Section 5, Corollary 2]{MR282985} shows that the sheaf $p_{X*}(P\vert_Z \otimes p_Y^* \cA^{\otimes d})$ (and hence $\pi_{X*}(\cP \otimes \pi_Y^* \cA^{\otimes d})$) is locally free of rank one,
    because $p_X \colon Z \to X$ is proper, $X$ is reduced, $\cP\vert_Z$ is flat over $X$, and \eqref{eq:rank-function} (see also \cite[Corollary III.12.9]{MR0463157}).
    Therefore, the surjection \eqref{eq:canonical_surjection} gives rise to another surjection
    \begin{equation}\label{eq:surjection}
        \cO_{X \times Y} \twoheadrightarrow \cP \otimes \pi_X^*\cM^\vee \otimes \pi_Y^*\cA^{\otimes d},
    \end{equation}
    in which $\cM$ denotes the line bundle $\pi_{X*}(\cP \otimes \pi_Y^* \cA^{\otimes d})$.
    The sheaf $\cP \otimes \pi_X^*\cM^\vee \otimes \pi_Y^*\cA^{\otimes d}$ is flat over $X$ and there is an isomorphism
    \begin{equation}
        (\cP \otimes \pi_X^*\cM^\vee \otimes \pi_Y^*(\cA^{\otimes d}))\vert_{\{x\} \times Y} \cong \cO_{(x, y)}.
    \end{equation}
    It means that the surjection \eqref{eq:surjection} gives rise to a morphism $f \colon X \to \Hilb_{Y}^1$ and \eqref{eq:surjection} is identified with the pullback of the universal quotient
    \begin{equation}
        (f \times \id)^*(\cO_{\Hilb_{Y}^1 \times Y} \to \cU).
    \end{equation}
    Then, in view of $Y \cong \Hilb_{Y}^1$ and $\cU \cong \cO_{\Delta}$, we have
    \begin{align}
        \cP & \cong (f \times \id)^*\cO_{\Delta} \otimes \pi_X^*\cM \otimes \pi_Y^*(\cA^{\otimes d})^\vee \\
            & \cong \cO_{\Gamma_f} \otimes \pi_X^*\cM \otimes \pi_Y^*(\cA^{\otimes d})^\vee               \\
            & \cong \cO_{\Gamma_f} \otimes^L L\pi_X^*\cM \otimes^L L\pi_Y^*(\cA^{\otimes d})^\vee
    \end{align}
    where $\Gamma_f$ is the graph of $f$.
    This shows that $\Phi \cong (\cA^{\otimes d})^\vee \otimes Rf_*(\cM \otimes -) \cong Rf_*(f^* (\cA^{\otimes d})^\vee \otimes \cM \otimes -)$ as desired.
\end{proof}

\begin{lemma}\label{lem:degree-is-locally-constant}
    Let $k$ be an algebraically closed field of characteristic zero.
    Let $X$ and $Y$ be schemes separated of finite type over $k$ and $\cP \in D^b(X \times Y)$ be an object whose support is proper over $X$.
    Let $i_x \colon \{x\} \times Y \hookrightarrow X \times Y$ be the natural inclusion for a closed point $x \in X$.
    Then the set
    \begin{equation}
        \left\{x \in X \mid x\text{ is a closed point s.t. }\cH^m(i_x^*\cP) = 0 \text{ for } m \neq 0\right\}
    \end{equation}
    is the set of closed points of an open subset of $X$.
\end{lemma}
\begin{proof}
    If $Y$ is proper over $k$, then the claim is a special case of \cite[Lemma 3.1.6]{bridgeland2002fourier}.
    For general $Y$, we take a compactification of $Y$, i.e.~a proper scheme $\overline{Y}$ over $k$ with an open immersion $\iota \colon Y \hookrightarrow \overline{Y}$.
    We can always find such $\overline{Y}$ by Nagata's compactification theorem \cite[\href{https://stacks.math.columbia.edu/tag/0F41}{Tag 0F41}]{stacks-project}.
    Then for every closed point $x \in X$ we have cartesian squares
    \[
        \begin{tikzcd}
            \{x\} \times Y \ar[r, hookrightarrow, "\iota"] \ar[d, hookrightarrow, "i_x"] & \{x\} \times \overline{Y} \ar[r] \ar[d, hookrightarrow, "\overline{i_x}"]& \{x\} \ar[d, hookrightarrow]\\
            X \times Y \ar[r, "\id_X \times \iota"] & X \times \overline{Y} \ar[r] & X,
        \end{tikzcd}
    \]
    where $i_x$ and $\overline{i_x}$ are the natural inclusions.
    The push-forward $\overline{\cP} = (\id_X \times \iota)_* \cP$ is an object of $D^b(X \times \overline{Y})$ by the assumption of proper support and hence the statement holds for $X$, $\overline{Y}$, and $\overline{\cP}$.
    Therefore, it is enough to show the equivalence
    \begin{equation}
        \cH^m(i_x^* \cP) = 0 \Leftrightarrow \cH^m(\overline{i_x}^*(\overline{\cP})) = 0
    \end{equation}
    for each $m \in \bZ$.
    Notice that, since $\iota$ is an open immersion, we have the flat base change isomorphism $\overline{i_x}^*(\overline{\cP}) \cong \iota_*i_x^* \cP$.
    In addition, the higher direct images of the open immersion $\iota$ vanish, and thus the push-forward $\iota_*$ is non-derived so that commutes with the cohomology sheaf: $\cH^m \circ \iota_* \cong \iota_* \circ \cH^m$.
    Then we have
    \begin{equation}
        \iota_*\cH^m(i_x^* \cP) \cong \cH^m(\iota_*i_x^*\cP) \cong \cH^m(\overline{i_x}^*(\overline{\cP})),
    \end{equation}
    which immediately implies the direction $\cH^m(i_x^* \cP) = 0 \Rightarrow \cH^m(\overline{i_x}^*(\overline{\cP})) = 0$.
    The converse direction is obtained from the property $\iota^*\iota_* \cong \id$ for the open immersion $\iota$.
\end{proof}

\subsection{Commutative diagrams}\label{subsection:compatibilities}
We collect some commutative diagrams which are used in the proof of Proposition \ref{prop:twist-functor-is-relative-fm}.

\begin{definition}[relative cup product]\label{def:relative-cup-product}
    Let $f \colon Y \to X$ be a morphism of schemes.
    The \emph{relative cup product morphism}
    \begin{equation}\label{eq:relative-cup-product}
        \CHom(f_*\cF, f_*\cG) \otimes f_*\cF \to f_*(\CHom(\cF, \cG) \otimes \cF)
        f_*\cF \otimes f_*\cG \to f_*(\cF \otimes \cG)
    \end{equation}
    for $\cF, \cG \in D_{qc}(Y)$ is the morphism
    \begin{equation}
        f_*\cF \otimes f_*\cG \xrightarrow{\eta_f} f_*f^*(f_*\cF \otimes f_*\cG) \xrightarrow{\sim} f_*(f^*f_*\cF \otimes f^*f_*\cG)\xrightarrow{f_*(\varepsilon_f \otimes \varepsilon_f)} f_*(\cF \otimes \cG),
    \end{equation}
    where $\eta_f$ and $\varepsilon_f$ are the unit and counit morphisms of the adjunction $f^* \dashv f_*$.
    In other words, the morphism \eqref{eq:relative-cup-product} is adjoint to the natural morphism $f^*(f_*\cF \otimes f_*\cG) \cong f^*f_*\cF \otimes f^*f_*\cG \to \cF \otimes \cG$.
\end{definition}

\begin{proposition}[relative cup product and Grothendieck duality]\label{prop:relative-cup-product-and-Grothendieck-duality}
    Let $f \colon X' \to X$ be a proper morphism between noetherian schemes.
    Then for any $\cF \in D^b(X')$ and $\cG \in D^b(X)$, we have a commutative diagram
    \[
        \begin{tikzcd}[column sep=4em]
            f_*\CHom(\cF, f^!\cG)\otimes f_*\cF \ar[rr, "\text{Grothendieck duality}"]\ar[d, "\text{relative cup product}"']& & \CHom(f_*\cF, \cG)\otimes f_*\cF \ar[d, "\text{evalation}"]\\
            f_*(\CHom(\cF, f^!\cG)\otimes \cF) \ar[r, "f_*(\text{evaluation})"']& f_*f^!\cG \ar[r, "\text{counit}"'] & \cG.
        \end{tikzcd}
    \]
    Here ``evaluation'' means the natural evaluation morphism $\CHom(-, *) \otimes (-) \to (*)$.
\end{proposition}
\begin{proof}
    By the adjunction $(-\otimes f_*\cF) \dashv \CHom(f_*\cF, -)$, the statement is equivalent to the commutativity of the diagram
    \[
        \begin{tikzcd}
            f_*\CHom(\cF, f^!\cG) \ar[rr, "\text{Grothendieck duality}"]\ar[d]& & \CHom(f_*\cF, \cG) \ar[d, equal]\\
            \CHom(f_*\cF, f_*(\CHom(\cF, f^!\cG)\otimes \cF)) \ar[r]& \CHom(f_*\cF, f_*f^!\cG) \ar[r, "(\text{counit}) \circ -"] & \CHom(f_*\cF, \cG),
        \end{tikzcd}
    \]
    and it commutes by the very definition of the isomorphism of Grothendieck duality (Remark \ref{remark:grothendieck-duality-isomorphism}).
\end{proof}

From now on, let $k$ be a field, $f \colon Y \hookrightarrow X$ be a \emph{closed immersion} of $k$-schemes, and $g = f\times f \colon Y \times_k Y \to X \times_k X$.
Consider the cartesian square
\begin{equation}\label{eq:diagonal-diagram}
    \begin{tikzcd}
        Y \arrow[r,"f"]\arrow[d, "\delta"'] & X \arrow[d,"\Delta"]&\\
        Y \times_k Y \arrow[r, "g"]& X \times_k X,
    \end{tikzcd}
\end{equation}
where $\Delta$ and $\delta$ are the diagonal morphisms.

\begin{lemma}\label{lem:Kunneth-formula-details}
    The isomorphism
    \begin{equation}
        f_*\cF \boxtimes f_*\cG \xrightarrow{\sim} g_*(\cF \boxtimes \cG)
    \end{equation}
    of Kunneth formula (Proposition \ref{prop:Kunneth-formula}) coincides with the morphism
    \begin{equation}
        f_*\cF \boxtimes f_*\cG \xrightarrow{\eta_g} g_*g^*(f_*\cF \boxtimes f_*\cG) \xrightarrow{\sim} g_*(f^*f_*\cF \boxtimes f^*f_*\cG) \xrightarrow{g_*(\varepsilon_f \boxtimes \varepsilon_f)}g_*(\cF \boxtimes \cG).
    \end{equation}
    Here $\eta_g \colon \id \to g_*g^*$ is the adjunction unit, the middle isomorphism is by Lemma \ref{lem:trivial-isomorphism}, and $\varepsilon_f \colon f^*f_* \to \id$ is the adjunction counit.
\end{lemma}
\begin{proof}
    We show that the following diagram is commutative:
    \[
        \begin{tikzcd}
            f_*\cF \boxtimes f_*\cG \ar[r, equal] \ar[d, "\eta_g"]& \pi_1^*f_*\cF \otimes \pi_2^*f_*\cG \ar[r]  \ar[d, "\eta_g"]& f_{1*}p_1^*\cF \otimes f_{2*}p_2^*\cG  \ar[d, "\eta_g"]\\
            g_*g^*(f_*\cF \boxtimes f_*\cG) \ar[r, equal] & g_*g^*(\pi_1^*f_*\cF \otimes \pi_2^*f_*\cG) \ar[r] \ar[d, "\sim"]& g_*g^*(f_{1*}p_1^*\cF \otimes f_{2*}p_2^*\cG) \ar[d, "\sim"]\\
            & g_*(g^*\pi_1^*f_*\cF \otimes g^*\pi_2^*f_*\cG) \ar[r] \ar[d, "\sim"]& g_*(g^*f_{1*}p_1^*\cF \otimes g^*f_{2*}p_2^*\cG) \ar[d, "\sim"] \\
            & g_*(g_1^*f_1^*\pi_1^*f_*\cF \otimes g_2^*f_2^*\pi_2^*f_*\cG) \ar[r] \ar[d, "\sim"]& g_*(g_1^*f_1^*f_{1*}p_1^*\cF \otimes g_2^*f_2^*f_{2*}p_2^*\cG) \ar[dd, "\psi"]\\
            & g_*(g_1^*p_1^*f^*f_*\cF \otimes g_2^*p_2^*f^*f_*\cG) \ar[d, "\varphi"]&  \\
            &g_*(g_1^*p_1^*\cF \otimes g_2^*p_2^*\cG) \ar[r, equal]& g_*(g_1^*p_1^*\cF \otimes g_2^*p_2^*\cG) \ar[d, equal]\\
            && g_*(\cF \boxtimes \cG).
        \end{tikzcd}
    \]
    Here the vertical isomorphisms are obvious ones, the horizontal morphisms are induced from the natural base change morphisms $\pi_1^*f_* \to f_{1*}p_1$ and $\pi_2^*f_* \to f_{2*}p_2$
    In addition, the morphisms $\varphi$ and $\psi$ are induced from the counits of the adjunctions $f^* \dashv f_*$, $f_1^* \dashv f_{1*}$, and $f_2^* \dashv f_{2*}$.

    Once this is established, the statement follows since the left outer route $f_*\cF \boxtimes f_*\cG \to g_*(\cF \boxtimes \cG)$ of the diagram is the map in the statement, and the right outer route is the map in Proposition \ref{prop:Kunneth-formula}.

    The squares except for the bottom one are commutative by the naturality.
    For the bottom square, it is enough to check that the diagram
    \[
        \begin{tikzcd}
            f_j^*\pi_j^*f_* \ar[r] \ar[d, "\sim"']& f_j^*f_{j*}p_j^* \ar[dd, "\varepsilon_{f_j}"] \\
            p_j^*f^*f_* \ar[d, "p_j^*(\varepsilon_f)"']& \\
            p_j^* \ar[r, equal]& p_j^*
        \end{tikzcd}
    \]
    for $j = 1, 2$, where $\varepsilon_{f_j}$ is the counit of the adjunction $f_j^* \dashv f_{j*}$, is commutative.
    By the adjunction $f_j^* \dashv f_{j*}$, it is equivalent to the commutativity of the diagram
    \begin{equation}\label{eq:small-diagram-in-proof-of-Kunneth}
        \begin{tikzcd}
            \pi_j^*f_* \ar[r] \ar[d]& f_{j*}p_j^* \ar[dd, equal] \\
            f_{j*}p_j^*f^*f_* \ar[d, "f_{j*}p_j^*(\varepsilon_f)"']& \\
            f_{j*}p_j^* \ar[r, equal]& f_{j*}p_j^*
        \end{tikzcd}
    \end{equation}
    or
    \begin{equation}\label{eq:small-diagram-2-in-proof-of-Kunneth}
        \begin{tikzcd}
            f_{j}^*\pi_j^*f_* \ar[r] \ar[d, "\sim"]& p_j^* \ar[dd, equal] \\
            p_j^*f^*f_* \ar[d, "p_j^*(\varepsilon_f)"']& \\
            p_j^* \ar[r, equal]& p_j^*.
        \end{tikzcd}
    \end{equation}
    The last diagram \eqref{eq:small-diagram-2-in-proof-of-Kunneth} is indeed commutative because the horizontal base change morphism $\pi_j^*f_* \to f_{j*}p_j^*$ in \eqref{eq:small-diagram-in-proof-of-Kunneth} is defined to make the diagram \eqref{eq:small-diagram-2-in-proof-of-Kunneth} commutative.
\end{proof}

\begin{lemma}\label{lem:trivial-isomorphism}
    For $\cA, \cB \in D_{qc}(X)$, there is a natural isomorphism
    \begin{equation}
        g^*(\cA \boxtimes \cB) \cong f^*\cA \boxtimes f^*\cB.
    \end{equation}
\end{lemma}
\begin{proof}
    We have
    \begin{align*}
        g^*(\cA \boxtimes \cB) & = g^*(\pi_1^*\cA \otimes \pi_2^*\cB)                    \\
                               & \cong g^*\pi_1^*\cA \otimes g^*\pi_2^*\cB               \\
                               & \cong g_1^*f_1^*\pi_1^*\cA \otimes g_2^*f_2^*\pi_2^*\cB \\
                               & \cong g_1^*p_1^*f^*\cA \otimes g_2^*p_2^*f^*\cB         \\
                               & = f^*\cA \boxtimes f^*\cB.
    \end{align*}
    Note that $g_1^*p_1^*$ (resp. $g_2^*p_2^*$) is isomorphic to the pullback by the first (resp. second) projection $Y \times_k Y \to Y$.
\end{proof}

For $\cF, \cG \in D_{qc}(Y)$, there are two natural morphisms $f_*\cF \boxtimes f_*\cG \to g_*(\cF \boxtimes \cG)$ and $f_*\cF \otimes f_*\cG \to f_*(\cF \otimes \cG)$.
The first one is the Kunneth formula isomorphism (Proposition \ref{prop:Kunneth-formula} and Lemma \ref{lem:Kunneth-formula-details}), and the second one is the relative cup product (which is not necessarily an isomorphism).
The next proposition says they are compatible with each other when restricted to the diagonal.
\begin{proposition}\label{prop:Kunneth-and-cup-product}
    For $\cF, \cG \in D_{qc}(Y)$, we have a commutative diagram
    \[
        \begin{tikzcd}
            f_*\cF \boxtimes f_*\cG \arrow[rr, "\eta_\Delta"]\arrow[d, "\text{Kunneth formula}"', "\sim"]&&\Delta_{*}(f_*\cF \otimes f_*\cG)\arrow[d, "\Delta_{*}(\text{relative cup product})"]\\
            g_*(\cF \boxtimes \cG) \arrow[r, "g_*(\eta_\delta)"' yshift=-1.0ex]&g_*\delta_{*}(\cF \otimes \cG) \arrow[r, "\sim"', sloped]&\Delta_{*}f_*(\cF \otimes \cG).
        \end{tikzcd}
    \]
    Here $\eta_\Delta$ and $\eta_\delta$ are the unit morphisms $\id \to \Delta_* \Delta^*$ and $\id \to \delta_* \delta^*$, respectively.
\end{proposition}
\begin{proof}
    In the proof, we denote the unit and counit of the adjunction $\varphi^* \dashv \varphi_*$ for a morphism $\varphi$ by $\eta_{\varphi}$ and $\varepsilon_{\varphi}$, respectively.

    We first show that the diagram
    \begin{equation}\label{eq:commutative-square-in-Kunneth-on-diagonal}
        \begin{tikzcd}
            \Delta^*(f_*\cF \boxtimes f_*\cG) \ar[rr, "\sim"]\arrow[d, "\Delta^*(\text{Kunneth formula})"', "\sim"]& & f_*\cF \otimes f_*\cG \ar[d, "\text{relative cup product}"]\\
            \Delta^*g_*(\cF \boxtimes \cG) \ar[r] & f_*\delta^*(\cF \boxtimes \cG) \ar[r, "\sim"] & f_*(\cF \otimes \cG)
        \end{tikzcd}
    \end{equation}
    is commutative.
    Here $\Delta^*g_* \to f_*\delta^*$ is the natural morphism.
    By Lemma \ref{lem:Kunneth-formula-details} and Definition \ref{def:relative-cup-product}, the diagram \eqref{eq:commutative-square-in-Kunneth-on-diagonal} is the outer square of the following:
    \begin{equation}\label{eq:commutative-square-in-Kunneth-on-diagonal2}
        \begin{tikzcd}
            \Delta^*(f_*\cF \boxtimes f_*\cG) \ar[r, equal]\arrow[dd, "\Delta^*(\eta_g)"']& \Delta^*(f_*\cF \boxtimes f_*\cG) \ar[r, "\sim"] \ar[d, "\eta_f"]& f_*\cF \otimes f_*\cG \ar[d, "\eta_f"]\\
            & f_*f^*\Delta^*(f_*\cF \boxtimes f_*\cG)\ar[r, "\sim"] \ar[d, "\sim"]&  f_*f^*(f_*\cF \otimes f_*\cG)\ar[dd, "\sim"]\\
            \Delta^*g_*g^*(f_*\cF \boxtimes f_*\cG) \ar[r]\arrow[d, "\sim"]&f_*\delta^*g^*(f_*\cF \boxtimes f_*\cG) \ar[d, "\sim"]& \\
            \Delta^*g_*(f^*f_*\cF \boxtimes f^*f_*\cG) \ar[r]\arrow[d, "\Delta^*g_*(\varepsilon_f \boxtimes \varepsilon_f)"']&f_*\delta^*(f^*f_*\cF \boxtimes f^*f_*\cG) \ar[r, "\sim"] \ar[d, "f_*\delta^*(\varepsilon_f \boxtimes \varepsilon_f)"]& f_*(f^*f_*\cF \otimes f^*f_*\cG)\ar[d, "f_*(\varepsilon_f \otimes \varepsilon_f)"]\\
            \Delta^*g_*(\cF \boxtimes \cG) \ar[r] & f_*\delta^*(\cF \boxtimes \cG) \ar[r, "\sim"] & f_*(\cF \otimes \cG).
        \end{tikzcd}
    \end{equation}
    Then it suffices to show that all the small squares are commutative.
    The commutativity of the left upper square is equivalent to that of
    \begin{equation}
        \begin{tikzcd}
            f^*\Delta^* \ar[r, equal]\arrow[d, "f^*\Delta^*(\eta_g)"']& f^*\Delta^* \ar[d, "\sim"]\\
            f^*\Delta^*g_*g^* \ar[r]& \delta^*g^*
        \end{tikzcd}
    \end{equation}
    by $f^* \dashv f_*$, and that of
    \begin{equation}\label{eq:commutative-square-in-Kunneth-on-diagonal3}
        \begin{tikzcd}
            \id \ar[r, equal]\arrow[d, "\eta_g"']& \id \ar[d]\\
            g_*g^* \ar[r]& \Delta_*f_*\delta^*g^*
        \end{tikzcd}
    \end{equation}
    by $f^*\Delta^* \dashv \Delta_*f_*$.
    By definition of $\Delta^*g_* \to f_*\delta^*$, under the identification $\Delta_*f_* \cong g_*\delta_*$ the lower horizontal morphism and the right vertical morphism in \eqref{eq:commutative-square-in-Kunneth-on-diagonal3} are identified with $g_*(\eta_\delta)$ and $\eta_{g \circ \delta}$, respectively:
    \begin{equation}\label{eq:commutative-square-in-Kunneth-on-diagonal4}
        \begin{tikzcd}
            \id \ar[r, equal]\arrow[d, "\eta_g"']& \id \ar[d, "\eta_{g \circ \delta}"]\\
            g_*g^* \ar[r, "g_*(\eta_\delta)"]& g_*\delta_*\delta^*g^*.
        \end{tikzcd}
    \end{equation}
    This is commutative by the compatibility of the pullback-pushforward adjunction with composition of morphisms.
    Then \eqref{eq:commutative-square-in-Kunneth-on-diagonal3}, and hence the left upper square of \eqref{eq:commutative-square-in-Kunneth-on-diagonal2}, are commutative.
    The left middle square and the left lower square in \eqref{eq:commutative-square-in-Kunneth-on-diagonal2} are commutative by the naturality of $\Delta^*g_* \to f_*\delta^*$.
    Similarly, the right upper square and the right lower square in \eqref{eq:commutative-square-in-Kunneth-on-diagonal2} are commutative by the naturality of the horizontal isomorphisms.
    Finally, the morphism $f_*\delta^*g^*(f_*\cF \boxtimes f_*\cG) \xrightarrow{\sim} f_*(f^*f_*\cF \otimes f^*f_*\cG)$ is the natural isomorphism of Lemma \ref{lem:trivial-isomorphism}, which is defined to make the right middle square commutative.
    Therefore, the diagram \eqref{eq:commutative-square-in-Kunneth-on-diagonal} is commutative.

    By applying the adjunction $\Delta^* \dashv \Delta_*$ to \eqref{eq:commutative-square-in-Kunneth-on-diagonal}, we obtain the commutative diagram
    \begin{equation}
        \begin{tikzcd}
            f_*\cF \boxtimes f_*\cG \ar[rr, "\eta_\Delta"]\arrow[d, "\text{Kunneth formula}"', "\sim"]& & \Delta_*f_*(\cF \otimes f_*\cG) \ar[d, "\text{relative cup product}"]\\
            g_*(\cF \boxtimes \cG) \ar[r] & \Delta_*f_*\delta^*(\cF \boxtimes \cG) \ar[r, "\sim"] & \Delta_*f_*(\cF \otimes \cG)
        \end{tikzcd}
    \end{equation}
    In addition, there is a commutative diagram
    \begin{equation}
        \begin{tikzcd}
            g_*(\cF \boxtimes \cG) \ar[r] \ar[rd, bend right, "g_*(\eta_\delta)"']& \Delta_*f_*\delta^*(\cF \boxtimes \cG)\\
            &g_*\delta_*\delta^*(\cF \boxtimes \cG) \ar[u, "\sim"]
        \end{tikzcd}
    \end{equation}
    by definition of $\Delta^*g_* \to f_*\delta^*$.
    These two diagrams show the statement.
\end{proof}

\subsection{Proof of Proposition \ref{prop:twist-functor-is-relative-fm}}\label{subsection:proof-of-twist-functor-is-relative-fm}
In this subsection, we prove Proposition \ref{prop:twist-functor-is-relative-fm}.
Recall our situation:
Let $X$ and $T$ be smooth quasi-projective varieties and $\pi \colon X \to T$ be a flat morphism.
Let $0 \in T$ be a closed point and $i \colon X_0 = \pi^{-1}(0) \hookrightarrow X$ be the fiber over $0$.
We have a diagram
\[
    \begin{tikzcd}
        X_0 \arrow[r,"i"]\arrow[d, "\delta"'] & X \arrow[d,"\Delta_T"]\arrow[rd, bend left, "\Delta_k"]&\\
        X_0 \times_k X_0 \arrow[r, "j"] \ar[d]& X \times_T X\arrow[r, "\iota"] \ar[d]&X \times_k X \\
        \{0\} \ar[r] & T &
    \end{tikzcd}
\]
in which $j$ and $\iota$ are the natural inclusions and $\delta$, $\Delta_T$, and $\Delta_k$ are the diagonal morphisms.
The two squares are cartesian and tor-independent by \cite[Lemma 2.25]{MR2238172}.
\begin{proposition}[Proposition \ref{prop:twist-functor-is-relative-fm}]\label{prop:twist-functor-is-relative-fm-recap}
    Let $\cE \in D^b(X_0)$ be a half-spherical object.
    \begin{enumerate}
        \item  The twist functor $T_{i_*\cE}$ is a relative integral functor with respect to $T$, i.e.~there exists an object $\cR \in D_{qc}(X \times_T X)$ such that $\Phi_{\cR} \cong T_{i_*\cE}$.
        \item The choice of $\cR$ in (1) is unique up to isomorphism and $\cR \in \FM_T(X)$.
    \end{enumerate}
\end{proposition}

\begin{proof}[Proof of (1)]
    Since the kernel $\cP_{\cE}$ of $T_{i_*\cE}$ is the cone of the map
    \begin{equation}
        \ev \colon (i_*\cE)^\vee \boxtimes i_*\cE \to \Delta_{k*}\cO_X,
    \end{equation}
    one simply has to find a morphism $\widetilde{\ev}$ in $D_{qc}(X \times_T X)$ that satisfies $\iota_*(\widetilde{\ev}) = \ev$, up to isomorphism.
    Denote $\CHom_{X_0}(\cE, i^!\cO_X)$ by $\cE'$.
    Consider the morphism
    \begin{equation}
        \widetilde{\ev} \colon j_*(\cE' \boxtimes \cE) \to j_*\delta_*(\cE' \otimes \cE)\cong \Delta_{T*}i_*(\cE' \otimes \cE) \to \Delta_{T*}\cO_X
    \end{equation}
    with
    \begin{itemize}
        \item $\cE' \boxtimes \cE \to \delta_*(\cE' \otimes \cE)$ is the adjoint to the natural isomorphism $\delta^*(\cE' \boxtimes \cE) \cong \cE' \otimes \cE$,
        \item $j_*\delta_* \cong \Delta_{T*}i_*$ is the natural isomorphism, and
        \item $i_*(\cE' \otimes \cE) \to \cO_X$ is the adjoint to the natural pairing $\cE' \otimes \cE \to i^!\cO_X$.
    \end{itemize}
    The domain and codomain of this map satisfy
    \begin{enumerate}
        \item[(a)] $(i_*\cE)^\vee \boxtimes i_*\cE \cong \iota_*j_*(\cE' \boxtimes \cE)$ by Grothendieck duality and Kunneth formula, and
        \item [(b)] $\Delta_{k*}\cO_X \cong \iota_*\Delta_{T*}\cO_X$.
    \end{enumerate}
    Then the following proposition shows that $\widetilde{\ev}$ is the morphism we are looking for.
\end{proof}

\begin{proposition}
    We have a commutative diagram
    \begin{equation}
        \begin{tikzcd}\label{eq:commutative-diagram-for-evaluation-map}
            (i_*{\cE})^\vee \boxtimes i_*{\cE}\ar[r, "\ev"] \ar[d, "\sim", sloped] & \Delta_{k*}\cO_X \ar[d, "\sim", sloped]\\
            \iota_*j_*(\cE' \boxtimes \cE) \ar[r, "\iota_*(\widetilde{\ev})"]& \iota_*\Delta_{T*}\cO_X
        \end{tikzcd}
    \end{equation}
    in $D_{qc}(X \times_k X)$ with the vertical isomorphisms being (a) and (b) in the previous proof.
\end{proposition}
\begin{proof}
    The diagram in the statement can be decomposed into the following:
    \[
        \begin{tikzcd}
            (i_*{\cE})^\vee \boxtimes i_*{\cE} \ar[rr]\ar[d, "\sim", sloped] &  & \Delta_{k*}((i_*{\cE})^\vee \otimes i_*{\cE}) \ar[rr]\ar[d, "\sim", sloped] &  & \Delta_{k*}\cO_X \ar[ddd, equal]\\
            i_*\cE'\boxtimes i_*{\cE} \ar[rr]\ar[dd]&  & \Delta_{k*}(i_*\cE'\otimes i_*{\cE}) \ar[dd] &  & \\
            & (A) &  & (B) & \\
            \iota_*j_*(\cE' \boxtimes \cE) \ar[r] & \iota_*j_*\delta_*(\cE' \otimes \cE) \ar[r, "\sim"]\ar[rd, out=-90, in=180, "\sim"] & \Delta_{k*}(i_*(\cE' \otimes \cE)) \ar[rr]\ar[d, "\sim", sloped] &  & \Delta_{k*}\cO_X \ar[d, "\sim", sloped] \\
            &  & \iota_*\Delta_{T*}i_*(\cE' \otimes \cE) \ar[rr] &  & \iota_*\Delta_{T*}\cO_X.
        \end{tikzcd}
    \]
    Here the upper and middle horizontal maps in the left column are given by the unit morphism
    $\cA \boxtimes \cB \to \Delta_{k*}\Delta_k^*(\cA \boxtimes \cB) \cong \Delta_{k*}(\cA \otimes \cB)$. The lower horizontal map is defined similarly.
    The square (A) is commutative by Proposition \ref{prop:Kunneth-and-cup-product} for $f = i$, $\cF = \cE'$, and $\cG = \cE$.
    The square (B) is commutative by Proposition \ref{prop:relative-cup-product-and-Grothendieck-duality} for $f = i$, $\cF = \cE$, and $\cG = \cO_X$.
    The other parts of the diagram are commutative by the naturality.
\end{proof}

\begin{proof}[Proof of (2)]
    We first show that the specific choice of $\cR = \Cone(\widetilde{\ev})$ constructed in (1) is an element of $\FM_T(X)$ by using the criterion in Lemma \ref{lem:criterion-for-invertible-kernel}.
    To check that the assumption of the lemma satisfied, by Theorem \ref{thm:adjoint-to-integral-functors}, it suffices to show that $\cR$ is in $D^b(X \times_T X)$, $\pi_1$-perfect, and has proper support with respect to $\pi_1 \colon X \times_T X \to X$.
    As these conditions have the two-out-of-three property it is enough to examine $j_*(\cE' \boxtimes \cE)$ and $\Delta_{T*}\cO_X$.

    The coherent sheaf $\Delta_{T*}\cO_X$ is clearly in $D^b(X \times_T X)$ and has proper support over $X$.
    It is also $\pi_1$-perfect, since $\cO_X$ is perfect and $\Delta_{T}$ is proper (Proposition \ref{prop:proper-push-forward-of-f-perfect}).

    Since $\iota_*$ preserves cohomology sheaves, the isomorphism $\iota_*(j_*(\cE' \boxtimes \cE)) \cong (i_*{\cE})^\vee\boxtimes i_*{\cE}$ shows that $j_*(\cE' \boxtimes \cE)$ is in $D^b(X \times_T X)$ and has proper support over $X$.
    To show that $j_*(\cE' \boxtimes \cE)$ is $\pi_1$-perfect, we use Proposition \ref{prop:crirtion_for_relative_perfection}.
    Let $\cF \in D^b_{qc}(X)$.
    We have
    \begin{align}
        \iota_*(j_*(\cE' \boxtimes \cE) \otimes \pi_1^*\cF) & \cong \iota_*j_*((\cE' \otimes i^*\cF) \boxtimes \cE)  \\
                                                            & \cong i_*(\cE' \otimes i^*\cF) \boxtimes i_*{\cE}      \\
                                                            & \cong (i_*(\cE') \otimes \cF) \boxtimes i_*{\cE}       \\
                                                            & \cong ((i_*{\cE})^\vee \otimes \cF) \boxtimes i_*{\cE}
    \end{align}
    in $D_{qc}(X \times_k X)$, where we used the projection formula for the first and third isomorphisms, and the Kunneth formula for the second isomorphism.
    Since $i_*\cE$ and $(i_*{\cE})^\vee$ are perfect and $\cF$ is bounded, the last term is bounded.
    This shows that $j_*(\cE' \boxtimes \cE)$ is $\pi_1$-perfect.

    Next, suppose that we are given another object $\cR' \in D_{qc}(X \times_T X)$ such that $\Phi_{\cR'} \cong T_{i_*\cE}$.
    It satisfies the assumption of Lemma \ref{lem:criterion-for-invertible-kernel} since we have already shown that the adjoint of $\Phi_{\cR'} \cong \Phi_{\cR}$ is an integral functor.
    Then the condition (2) in Lemma \ref{lem:criterion-for-invertible-kernel} satisfied for $\Phi_{\cR'} \cong \Phi_{\cR}$ and thus $\cR' \in \FM_T(X)$.

    Finally, we have the uniqueness since all the possible $\cR$'s are in $\FM_T(X)$ and the map $\Phi_{(-)} \colon \FM_T(X) \to \Auteq D^b(X)$ is injective by Corollary \ref{cor:FM-to-Auteq} and Proposition \ref{prpo:forget-base}.
\end{proof}

\printbibliography
\end{document}